\documentclass{amsart}
\usepackage{amssymb,amsmath,amscd,amsthm}
\usepackage{amsfonts,epsfig,latexsym,graphicx,amssymb}
\usepackage{color}
\usepackage{hyperref}

\date{\today}

\newcommand{\Z}{{\mathbb Z}}
\newcommand{\R}{{\mathbb R}}

\newcommand{\N}{{\mathbb N}}

\newcommand{\osc}{\mathop{\mathrm{osc}}}

\newcommand{\lR}{\lambda_{RD}}
\newcommand{\Leb}{{\mathrm{Leb}}}
\newcommand{\DOS}{{\mathrm{DOS}}}
\newcommand{\SL}{{\mathrm{SL}}}

\newcommand{\bo}{\bar{\omega}}
\newcommand{\bv}{\bar{v}}

\newcommand{\Dirac}{\delta}

\newcommand{\bmin}{b_-}
\newcommand{\bmax}{b_+}

\newcommand{\tf}{\tilde f}

\newcommand{\tg}{\tilde g}
\newcommand{\tG}{\tilde G}
\newcommand{\tx}{\tilde x}
\newcommand{\ty}{\tilde y}

\newcounter{mit}
\newenvironment{mainit}
	{\setcounter{mit}{0}\begin{itemize}}
	{\end{itemize}}
\newcommand{\mitem}{\refstepcounter{mit}\item[\themit]}
\renewcommand{\themit}{\textbf{\Roman{mit}}}

\newtheorem{theorem}{Theorem}[section]
\newtheorem{lemma}[theorem]{Lemma}
\newtheorem{prop}[theorem]{Proposition}
\newtheorem{coro}[theorem]{Corollary}

\theoremstyle{definition}
\newtheorem{remark}[theorem]{Remark}

\newtheorem{defi}[theorem]{Definition}

\newcommand{\edited}[1]{\textcolor{black}{#1}}

\newcommand{\Sc}{{\mathbb S}^1}

\newcommand{\Lx}{L}
\newcommand{\La}{L_p}
\newcommand{\ep}{\varepsilon'}

\newcommand{\bm}{\bar{m}}

\def\dist{{\rm dist}}
\def\supp{\mathop{\rm supp}}

\newcommand{\bi}{{\bf i}}

\def\N{{\mathbb N}}

\newcommand{\E}{{\mathbb E}\,}
\newcommand{\Prob}{{\mathbb P}\,}
\def\P{\Prob}

\def\be{\begin{equation}}
\def\ee{\end{equation}}

\newcommand{\eps}{{\varepsilon}}

\newcommand{\const}{{\rm const}}

\begin{document}

\title[ Parametric Furstenberg Theorem]{Parametric Furstenberg Theorem \\ on Random Products of $SL(2, \mathbb{R})$ matrices}

\author[A.\ Gorodetski]{Anton Gorodetski}

\address{Department of Mathematics, University of California, Irvine, CA~92697, USA\\
and  National Research University Higher School of Economics, Russian Federation}

\email{asgor@uci.edu}

\thanks{A.\ G.\ was supported in part by Simons Fellowship (grant number 556910), Simons Visiting Professor Award, and NSF grant DMS--1855541} 

\author[V. Kleptsyn]{Victor Kleptsyn}

\address{CNRS, Institute of Mathematical Research of Rennes, IRMAR, UMR 6625 du CNRS}

\email{victor.kleptsyn@univ-rennes1.fr}

\thanks{V.K. was supported in part by RFBR projects 16-01-00748-a and 13-01-00969-a, by
Centre Henri Lebesgue ANR-11-LABX-0020-01, and by ANR Gromeov (ANR-19-CE40-0007).}

\thanks{Both authors were supported in part by Laboratory of Dynamical Systems and Applications NRU HSE, grant of the Ministry of science and higher education of the RF ag. N 075-15-2019-1931.}

\begin{abstract}
We consider random products of $SL(2, \mathbb{R})$ matrices that depend on a parameter in a non-uniformly hyperbolic regime. We show that if the dependence on the parameter is monotone then almost surely the random product has upper (limsup) Lyapunov exponent that is equal to the value prescribed by the Furstenberg Theorem (and hence positive) for all parameters, but the lower (liminf) Lyapunov exponent is equal to zero for a dense $G_\delta$ set of parameters of zero Hausdorff dimension. As a byproduct of our methods, we provide a purely geometrical proof of Spectral Anderson Localization for discrete Schr\"odinger operators with random potentials (including the Anderson-Bernoulli model) on a one dimensional lattice.
\end{abstract}

\maketitle


%

\section{Introduction}

Random products of matrices appear naturally in smooth dynamical systems \cite{V, W}, probability theory  \cite{Bel, Ber, FK, KS}, 
  spectral theory and mathematical physics \cite{D,S}, geometric measure theory \cite{HS, PT, Sh}. The main questions are usually focused on the rate of growth of these products. In this context an important step was made in 1960 by Furstenberg and Kesten \cite{FK}.
 They proved that products of random matrices generated by a stationary process have well defined asymptotic exponential growth rate. This rate of growth is usually called {\it Lyapunov exponent}. It corresponds exactly to the logarithm of the spectral radius when all the random matrices degenerate to a single matrix. In \cite{Fur1, Fur2} Furstenberg  showed that in most cases the Lyapunov exponent must be positive; see also \cite{Vi} for a different proof. Here is the classical version of the Furstenberg Theorem.

\begin{theorem}\label{t.F}
Let $\{X_k, k\ge 1\}$ be independent and identically distributed random variables, taking values in $SL(d, \mathbb{R})$, the $d\times d$ matrices with determinant one, let $G_X$ be the smallest closed subgroup of $SL(d, \mathbb{R})$ containing the support of the distribution of $X_1$, and assume that
$$
\E[\log\|X_1\|]<\infty.
$$
Also, assume that $G_X$ is not compact, and there exists no $G_X$-invariant finite union of proper subspaces of~$\R^d$. Then there exists a positive constant $\lambda_F$ such that with probability one
$$
\lim_{n\to \infty}\frac{1}{n}\log\|X_n\ldots X_2X_1\|=\lambda_F>0.
$$
\end{theorem}

This result was generalized and improved in many different ways, see \cite{BL, CKN, GM} for classical surveys, and \cite{Fu} for a more recent one.
For example, Oseledets Theorem \cite{O} claims that Lyapunov exponent exists for large class of linear cocycles, not only for the random products of matrices, and describes the structure of subspaces of vectors with different growth rates. Dependence of 
the Lyapunov exponent $\lambda_F$ on the distribution in the space of matrices \edited{(e.g. continuous, H\"older continuous, or smooth dependence)} was considered in \cite{BV, FKif, Kif, KifS, He, Per, TV,  DK1}. Also, in the case of random products of matrices that depend on a parameter the properties of $\lambda_F$ as a function of the parameter were heavily studied. In particular, it is known that for the uniformly hyperbolic case (the formal definition is provided below) $\lambda_F$ is an analytic function of the parameter~\cite{R}, but in general only H\"older continuity can be guaranteed \cite{L}.

The focus of our paper is also on the case when the matrices in the random product depend on a parameter. But instead of studying the properties of $\lambda_F$ as a function of the parameter, we want to fix a (generic) sequence of matrices, and ask whether the Lyapunov exponent exists for all parameters for the product formed by this specific sequence.
In Section \ref{ss.11} below we present two examples to motivate this question, and discuss the case of uniformly hyperbolic set of matrices. Then in Section \ref{ss.12} we formulate our main result, a parametric version of Furstenberg Theorem. In order to illustrate the power of our approach, in Section 1.3 we consider the Anderson model (including the Anderson-Bernoulli model) in the case of discrete Schr\"odinger operators on one dimensional lattice, and give a purely geometrical proof of Anderson Localization (pure point spectrum and exponential decay of eigenfunctions). Finally, in Section \ref{ss.14} we complete the introduction with the statement of the result on properties of finite random products of $SL(2, \mathbb{R})$ matrices; this result is the main technical part of the proof of parametric version of Furstenberg Theorem, but is also of interest by itself.

\subsection{Two examples}\label{ss.11}

Before providing the formal statement of our results let us consider two examples.

\vspace{4pt}

{\bf Example 1.}
Consider two matrices $A,B\in SL(2, \mathbb{R})$, and the family of matrices $\{R_\alpha\circ A, R_\alpha\circ B\}$, where $R_\alpha=\left(
                                                                        \begin{array}{cc}
                                                                          \cos \alpha & -\sin \alpha \\
                                                                          \sin \alpha & \cos \alpha \\
                                                                        \end{array}
                                                                      \right)$
is a rotation by angle $\alpha$,  $\alpha\in [\alpha_1, \alpha_2]$.  Denote $A_\alpha=R_\alpha\circ A, B_\alpha=R_\alpha\circ B$ and consider random products of $A_\alpha$ and $B_\alpha$ (chosen with some given probabilities $p$ and $1-p$) . Assume that for each $\alpha\in [\alpha_1, \alpha_2]$ the set of matrices $\{A_\alpha, B_\alpha\}$ satisfies the Furstenberg genericity conditions, i.e. the group generated by $A_\alpha$ and $B_\alpha$  is not
contained in any compact subgroup of $SL(2, \R)$, and there is no  finite union of proper subspaces of~$\R^2$ that would be invariant under both $A_\alpha$ and $B_\alpha$.
Then due to Theorem \ref{t.F} for any $\alpha\in [\alpha_1, \alpha_2]$ for almost every sequence $\{C_i(\alpha)\}$, $C_i(\alpha)\in \{A_\alpha, B_\alpha\}$, there is a limit
\begin{equation}\label{e.C}
 \lim_{n\to \infty}\frac{1}{n}\|C_1(\alpha)C_2(\alpha)\ldots C_n(\alpha)\|=\lambda_F(\alpha)>0.
\end{equation}
{\it Is it true that for almost every sequence $\{C_i\}$ the limit (\ref{e.C}) exists for all $\alpha\in [\alpha_1, \alpha_2]$?}

\vspace{4pt}

{\bf Example 2.}
Let us consider Schr\"odinger cocycle associated with the one-dimensional Anderson model, where the role of parameter is played by the energy. Namely, we consider Schr\"odinger operators $H$ acting on $\ell^2(\Z)$ via
\begin{equation}\label{e.oper}
[H u](n) = u(n+1) + u(n-1) + V(n) u(n).
\end{equation}
We will assume that $\{V(n)\}$ are i.i.d. random variables, distributed with respect to some compactly supported non-degenerate (support contains more than one point) probability measure $\mu$. Notice that we do not require the distribution $\mu$ to be continuous; in particular, the Anderson-Bernoulli model (when potential $V(n)$ can takes only two different values) is included in our setting. We will denote by $V_\omega$, where $\omega\in (\text{supp}\, \mu)^{\mathbb{Z}}$, the particular choice of the potential $V$, and by $H_\omega$ the corresponding operator (\ref{e.oper}).

A sequence $u\in \ell^2(\Z)$ is an eigenvector of $H$, that is, satisfies $Hu=Eu$ for some eigenvalue (``energy'')~$E$, if and
only if it solves the difference equation
\begin{equation}\label{e.eveom}
u(n+1) + u(n-1) + V_\omega(n) u(n) = E u(n), \quad n \in \Z.
\end{equation}
Now, $u$ solves \eqref{e.eveom} if and only if
\begin{equation}\label{e.ostmom}
\begin{pmatrix} u(n+1) \\ u(n) \end{pmatrix} = \begin{pmatrix} E - V_\omega(n) & -1 \\ 1 & 0 \end{pmatrix} \begin{pmatrix} u(n) \\ u(n-1) \end{pmatrix}, \quad n \in \Z.
\end{equation}
One naturally defines
$$
\Pi_{n,E,\omega} = \begin{pmatrix} E - V_\omega(n) & -1 \\ 1 & 0 \end{pmatrix},
$$
so that \eqref{e.ostmom} implies
$$
\begin{pmatrix} u(n+1) \\ u(n) \end{pmatrix} = \Pi_{n,E,\omega} \times \cdots \times \Pi_{1,E,\omega} \begin{pmatrix} u(1) \\ u(0) \end{pmatrix}
$$
for $n \ge 1$ and any solution $u$ to~\eqref{e.eveom}. Thus, the study of spectral properties of~$H$ motivates the study of such random products;
we set $T_{n, E,\omega} = \Pi_{n,E,\omega} \times \cdots \times \Pi_{1,E,\omega}$.

Due to Theorem \ref{t.F} for any $E\in \mathbb{R}$ for almost every $\omega\in (\text{supp}\, \mu)^{\mathbb{Z}}$ there is a limit
\begin{equation}\label{Lyap}
\lim_{n\to \infty} \frac{1}{n} \log \|T_{n, E,\omega} \| = \lambda_F(E)>0.
\end{equation}
However, from the spectral point of view it makes sense to fix the potential $V_\omega$ first, and then vary the value of the energy $E$.
{\it Is it true that for almost every $\omega\in (\text{supp}\, \mu)^{\mathbb{Z}}$ the limit (\ref{Lyap}) exists for all $E\in \mathbb{R}$? For all $E$ from a given interval $J\subset \mathbb{R}$? }

\vspace{4pt}

To give a comprehensive answer to the questions in both examples let us introduce a more general framework.

\subsection{Parametric version of Furstenberg Theorem}\label{ss.12}

Let $(\Omega, \mu)$ be a probability space, $J\subset \R$ be a compact interval of parameters, and
$F:\Omega\times J \to SL(2,\R)$ be a bounded measurable (and continuous in second argument) map that to any $\omega\in\Omega$ puts in correspondence
a matrix $F_a(\omega)$ that depends continuously on the parameter $a\in J$. In Example 1 above the role of parameter was played by the angle $\alpha$, and in Example~2~--- by the value of energy $E$. For a given sequence $\bo\in \Omega^\mathbb{N}, \bar \omega=\omega_1\omega_2\ldots$ denote
$$
T_{n,a, \bar \omega} = F_a(\omega_n)F_a(\omega_{n-1})\ldots F_a(\omega_1).
$$ Furstenberg-Kesten Theorem \cite{FK} implies that for each value of the parameter $a\in J$  there is a subset $\Omega_a\subseteq \Omega^\mathbb{N}$ with $\mu^\mathbb{N}(\Omega_a)=1$ such that for any $\bar\omega\in \Omega_a$ the limit
\begin{equation}\label{e.lim}
\lambda_F(a):= \lim_{n\to \infty}\frac{1}{n}\log\|T_{n,a, \bar \omega}\|
\end{equation}
exists.

{\it Is it possible to choose $\Omega_a$ uniformly in the parameter? In other words, is it true that $\mu^\mathbb{N}$-almost surely the limit (\ref{e.lim}) exists for all values of the parameter $a\in J$?}

Notice that the questions stated in Section \ref{ss.11} are partial cases of this one. It turns out that the answer to these questions is drastically different depending of presence or absence of uniform hyperbolicity.

\begin{defi}
A collection of $SL(2, \mathbb{R})$ (or $SL(k, \mathbb{R})$) matrices $\{M_\alpha\}_{\alpha\in \mathcal{A}}$  is called \emph{uniformly hyperbolic} if there exists a constant $\eta>1$ such that for any finite sequence of matrices $M_{\alpha_1}, M_{\alpha_2}, \ldots, M_{\alpha_n}$ we have
$
\|M_{\alpha_1} M_{\alpha_2} \ldots M_{\alpha_n}\|>\eta^n.
$
\end{defi}

There is a number of equivalent ways to describe uniform hyperbolicity of $SL(2, \mathbb{R})$ (or $SL(k, \mathbb{R})$) cocycles, such as an invariant splitting into stable and unstable directions, or the absence of a Sacker-Sell solution; compare, for example, \cite{ABY, DFLY14, Y04, Z}. 
In particular, existence of invariant one-dimensional stable and unstable directions for uniformly hyperbolic $SL(2, \mathbb{R})$ cocycles combined with Birkhoff Ergodic Theorem immediately implies the following statement:
\begin{prop}\label{p.uh}
In the setting above, assume that for each $a\in J$ the collection of matrices $\{F_a(\omega)\}_{\omega\in \Omega}$ is uniformly hyperbolic.
Then, for $\mu^{\N}$-a.e. $\bo\in\Omega^{\N}$ the limit
$$
\lim_{n\to \infty} \frac{1}{n} \log \|T_{n,a, \bar \omega} \| = \lambda_F(a)>0
$$
exists for all $a\in J$.
\end{prop}

\begin{remark} In the case of $SL(k, \mathbb{R})$, $k>2$, even uniform hyperbolicity does not guarantee the convergence uniformly in parameter, or even pointwise convergence for all parameters. More restrictive assumptions (e.g. positivity of all entries of the matrices, as in \cite{CN, Pol}, \edited{or existence of a dominated splitting of index or co-index 1, see \cite{BoGo}}) are needed; see also \cite[Theorem 2.2]{G}.
\end{remark}

The goal of this work is to provide the detailed description of the case complementary to the setting of Proposition \ref{p.uh}. This case (positive Lyapunov exponent in absence of uniform hyperbolicity) is usually referred to as {\it  non-uniformly hyperbolic case.}

From now on, we will proceed under the following standing assumptions:
\begin{itemize}
\item[(A1)] \textbf{(Furstenberg condition)} Denote by $\mu_a$ the measure $\mu_a=(F_a)_*(\mu)$. We assume that for
each $a\in J$ the measure $\mu_a$ on $SL(2,\R)$ satisfies the (individual) Furstenberg non-degeneracy condition, that is, its support is not
contained in any compact subgroup of $SL(2, \R)$, and there is no $\supp \mu_a$-invariant finite union of proper subspaces of~$\R^2$.

\vspace{4pt}

\item[(A2)] \textbf{($C^1$-boundedness)} The maps $F_a(\omega)$ are $C^1$-smooth in the parameter $a\in J$, with  uniformly bounded
$C^1$-norm, i.e. there exists $M>0$ such that for all $\omega\in \Omega$ and all $a\in J$
$$
 \|F_{a}(\omega)\|, \left\|\frac{d}{da}F_{a}(\omega)\right\| \le M.
$$
\item[(A3)] \textbf{(Non-uniform hyperbolicity)} For each $a\in J$ the collection of matrices $\{F_a(\omega)\}_{\omega\in \Omega}$  is \emph{not} uniformly  hyperbolic.
    \vspace{4pt}
\item[(A4)] \textbf{(Monotonicity)} There exists $\delta>0$ such that
$$
\frac{d}{da}\mathrm{arg}(F_a(\omega)\bar v)>\delta>0
$$
for all $a\in J, \omega\in \Omega, \bar v \in \mathbb{R}^2\backslash \{0\}$. In other words, as we increase the parameter, the image of
any given vector $\bar v$ spins in the positive direction with a speed that is bounded from below.
\end{itemize}

\begin{remark}\label{r.cont}
\edited{The condition $(A1)$ is sometimes  referred to as {\it strong irreducibility} (non existence of proper subspaces invariant under
the closed semigroup generated by the support of the measure $\mu_a$) and
{\it contractivity} (existence of matrices of arbitrarily large norm in that semigroup) assumptions. In it known that under these assumptions the Lyapunov exponent $\lambda_F(a)$ is continuous \cite{FK}. For the current state of art regarding continuity of Lyapunov exponents see the monographs \cite{DK1} and \cite{V}. }
\end{remark}

\vspace{2pt}

Our main result is the following theorem, describing the behaviour of the random parameter-dependent products of $SL(2, \R)$ matrices:
\begin{theorem}[Parametric version of Furstenberg Theorem]\label{t:product}
Under the assumptions $(A1)-(A4)$ above, for $\mu^{\N}$-almost every $\bo\in\Omega^{\N}$ the following holds:
\begin{itemize}
\item[$\bullet$] \textbf{(Regular upper limit)} For every $a\in J$ we have
$$
 \quad \limsup_{n\to\infty} \frac{1}{n} \log \|T_{n,a, \bar \omega} \| = \lambda_F(a)>0.
$$
\item[$\bullet$] \textbf{($G_{\delta}$-vanishing)} The set
$$
S_0(\bo):=\left\{a \in J \mid \liminf_{n\to\infty} \frac{1}{n} \log \|T_{n,a, \bar \omega} \| =0 \right\}
$$
is a (random) dense $G_{\delta}$-subset of the interval~$J$.
\vspace{4pt}
\item[$\bullet$] \textbf{(Hausdorff dimension)} The (random) set of parameters with exceptional behaviour,
$$
S_{e}(\bo):=\left\{a \in J \mid \liminf_{n\to\infty} \frac{1}{n} \log \|T_{n,a, \bar \omega} \| <\lambda_F(a) \right\},
$$
has zero Hausdorff dimension:
$$
\dim_H S_{e}(\bo)=0.
$$
\end{itemize}
\end{theorem}

\begin{remark}\label{r.ex}
Let us  consider the properties  $(A1)-(A4)$ in the context of Examples 1 and 2 from Section \ref{ss.11} to show that Theorem \ref{t:product} can be applied to both of them.

\vspace{4pt}

Example 1: The assumptions $(A2)$ and $(A4)$ obviously hold. It is also not hard to give an explicit example of $\{A_\alpha, B_\alpha\}$ and an interval $J$ such that $\{A_\alpha, B_\alpha\}$ is not uniformly hyperbolic and satisfy Furstenberg non-degeneracy conditions for all $\alpha\in J$, e.g. see  \cite[Example 2.2]{GI}. It is interesting to compare Theorem \ref{t:product} in the context of Example 1 with \cite[Corollary 4]{AB}.

\vspace{4pt}

Example 2: Assumption $(A2)$ is certainly satisfied. Notice that the Furstenberg conditions are satisfied automatically for transition matrices $\{\Pi_{n, E, \omega}\}$, e.g. see the proof of Theorem 2.17 from \cite{D}. As for assumption $(A3)$, Johnson showed in \cite{J86} that the set of energies $E$ for which the collection of matrices $\{\Pi_{n, E, \omega}\}_{\omega\in (\text{supp}\, \mu)^{\mathbb{Z}}}$ is uniformly hyperbolic, is equal to the resolvent set of $H_\omega$ for $\mu^{\mathbb{Z}}$-almost every $\omega$. Besides, $\mu^{\mathbb{Z}}$-almost surely the spectrum of $H_\omega$ is a finite union of intervals of length at least four (more precisely, it is equal to $[-2,2]+\supp \mu$), e.g. see Theorem 4.1 from \cite{D16}. Therefore, an interval of energies inside of the spectrum corresponds to the non-uniformly hyperbolic case. Finally, notice that while the condition $(A4)$ (monotonicity) does not hold in general for matrices $\Pi_{n, E, \omega}$, it is a straightforward calculation to check that it does hold for a product of two consecutive matrices $\Pi_{n, E, \omega}\Pi_{n+1, E, \omega}$, and this allows to apply Theorem \ref{t:product} in the context of Example~2.

Notice that in this case existence of a dense subset of energies in the spectrum for which the limit that defines the Lyapunov exponent does not exist was shown in \cite[Theorem 6.2]{G}.
\end{remark}

\begin{remark}
Monotonic cocycles (i.e. satisfying the property $(A4)$) were considered previously, for example, by Avila and Krikorian in \cite{AvK}. There they developed, in particular, a dynamical analog of Kotani Theory, see \cite[Theorem 1.7]{AvK}. Theorem \ref{t:product} also has some counterparts in spectral theory. Namely, the statement on ``Regular upper limit'' can be considered as a dynamical analog (and, in fact, improvement) of the result by Craig and Simon \cite[Theorem 2.3]{CS}. Also, ``$G_{\delta}$-vanishing'' part seems to be related to \cite[Theorem 2, Theorem 2.1]{DMS}, see also \cite[Theorem 2]{Gor}. Namely, the set of exceptional parameters $S_e$ from Theorem \ref{t:product} is analogous to the set of ``exceptional energies'' for rank one perturbations of a (continuous) Schr\"odinger operator without a.c. spectrum, see \cite[Example 5.2]{DMS}. Moreover, one could extract  from the proofs in \cite{DJLS} the arguments needed to show that in the case of random potential the set of ``exceptional energies'' must have zero Hausdorff dimension \cite{J}.  We are grateful to Lana Jitomirskaya for this remark.
\end{remark}

\begin{remark}
It is interesting to compare Theorem \ref{t:product} with the result from \cite{Bo} that claims that for any fixed invertible ergodic dynamical system over a compact space, there is a residual set of continuous $SL(2,\mathbb{R})$-cocycles which are either  uniformly hyperbolic or have zero
exponents a.e.; for related results on $SL(k,\mathbb{R})$ cocycles see \cite{BoV1, BoV2}. In the opposite direction, denseness of $SL(k,\mathbb{R})$ cocycles with non-zero Lyapunov exponents was shown in \cite{Av}. Moreover, for a generic smooth (or H\"older) cocycle over a hyperbolic base positivity of Lyapunov exponents was shown in \cite{V2, BGV}; see also \cite{BV, BocV, VY} for other related results. The question about positivity of Lyapunov exponent for Schr\"odinger cocycles over a hyperbolic base in some specific cases was studied in \cite{ChS, Z2}; in full generality essential progress was also announced \cite{Dpers}.


%

\end{remark}

\begin{remark}
\edited{One of the powerful methods currently available to study the properties of cocycles with positive Lyapunov exponent is Avalanche Principle, see \cite{GS}, \cite{DK1}. Notice that this is not an approach we are using in this paper. Indeed, Avalanche Principle allows to establish an inductive procedure by using estimates on the norms of products of ``blocks'' of matrices under an assumption that no critical cancelations happen between two subsequent ``blocks''. We do not establish any inductive procedure; instead we analyze the properties of large finite products of parameter dependent matrices directly, see Theorem \ref{t:main} below.
We are grateful to one of the referees for this remark.}
\end{remark}
%

\subsection{Anderson Localization}\label{ss.13}

One  important application of Furstenberg's Theorem on random matrix products lies in the context of Anderson Localization for discrete Schr\"odinger operators with random potentials on one dimensional lattice; this model is described in Example 2 from Section \ref{ss.11}. The following result is well known.

\begin{theorem}[Spectral Anderson Localization, 1D]\label{t.al}
The spectrum of the operator $H_\omega$ defined by (\ref{e.oper}) is $\mu^{\mathbb{Z}}$-almost surely pure point, with exponentially decreasing eigenfunctions. The same statement holds for spectrum of discrete Schr\"odinger operator with random potential in $\ell^2(\mathbb{N})$ with Dirichlet boundary condition.
\end{theorem}

\begin{remark}
In Example 2 above we assume that $\mu$ is a non-degenerate compactly supported measure on $\mathbb{R}$. Theorem \ref{t.al} is known to hold also for the case of $\mu$ with unbounded support (under some extra conditions), e.g. see Theorem 2.1 from \cite{CKM}. We believe that our approach and results (including Theorem \ref{t:product}) can also be extended to the case of distribution with unbounded support under some reasonable conditions, but do not elaborate on it in this paper.
\end{remark}

There are many different proofs of Theorem \ref{t.al}, see  \cite{GMP, KuS} the initial proofs of related statements, and \cite{D} for a survey. Most of the proofs rely either on Furstenberg Theorem (Theorem \ref{t.F}), or on Kunz-Souillard method \cite[Section 4]{D} (but there are exceptions, e.g. see \cite[Remark 4.2]{FLSSS}). The Kunz-Souillard method requires absolute continuity  of the distribution $\mu$. The same condition (or at least existence of  an absolutely continuous component) is needed for shorter proofs that use Furstenberg Theorem, e.g. the method of Spectral Averaging  \cite{SW} (see also \cite[Section 3.2]{D}). The first complete proof of Theorem \ref{t.al} that would also cover the Anderson-Bernoulli model  (the case when the support of $\mu$ consists of two points) was given by Carmona, Klein, and Martinelli in \cite{CKM}, see also \cite{DSS} for continuum case. When this paper was at the final stage of preparation, we learned about two other proofs. The  paper \cite{BDFGVWZ} provides a proof of Anderson Localization in 1D that is relatively elementary and avoids multi-scale analysis, using Furstenberg Theorem as the main tool. Also, the very recent paper \cite{JZh} gives a short proof of Theorem \ref{t.al}. \edited{Anderson Localization of random Jacobi operators (and related version of Large Deviation Estimates) was studied by Duarte and Klein in \cite{DK2}.}

We would like to present here a purely geometrical proof of Theorem \ref{t.al} based on techniques similar to the parametric version of Furstenberg Theorem above, that shows that in 1D case Anderson Localization can arguably be considered as a dynamical rather than purely spectral phenomenon.

More specifically, we can show that the following statement holds:

\begin{theorem}\label{t.vector}
Under the assumptions $(A1)-(A4)$ we have:
\begin{itemize}
\item[$\bullet$] For almost all $\bo\in\Omega^{\N}$, for all $a\in J$ the following holds. If
\begin{equation}\label{eq:lim-less}
\limsup_{n\to+\infty} \frac{1}{n} \log |T_{n,a, \bar \omega} \left( \begin{smallmatrix} 1 \\ 0 \end{smallmatrix} \right)| < \lambda_F(a),
\end{equation}
then in fact $|T_{n,a, \bar \omega} \left(  \begin{smallmatrix} 1 \\ 0 \end{smallmatrix} \right)|$ tends to zero exponentially as $n\to \infty$. Namely,
$$
\lim_{n\to+\infty} \frac{1}{n} \log |T_{n,a, \bar \omega} \left( \begin{smallmatrix} 1 \\ 0 \end{smallmatrix} \right)| = - \lambda_F(a).
$$

\item[$\bullet$] For almost all $\bo\in\Omega^{\Z}$, for all $a\in J$ the following holds. If for some $\bar{v}\in \R^2\setminus \{0\}$ we have
\begin{equation}\label{eq:lim-both-less}
\limsup_{n\to+\infty} \frac{1}{n} \log |T_{n,a, \bar \omega} \bar v| < \lambda_F(a), \ \ \text{and}\ \ \
\limsup_{n\to+\infty} \frac{1}{n} \log |T_{-n,a, \bar \omega} \bar v| < \lambda_F(a),
\end{equation}
where $$T_{-n,a, \bar \omega}:=F_a(\omega_{-n})^{-1}\dots F_a(\omega_{-1})^{-1}F_a(\omega_0)^{-1},$$
then both $|T_{n,a, \bar \omega} \bar v|, |T_{-n,a, \bar \omega} \bar v|$ in fact tend to zero exponentially. Namely,
$$
\lim_{n\to+\infty} \frac{1}{n} \log |T_{n,a, \bar \omega} \bar v| = - \lambda_F(a), \ \ \text{and}\ \ \
\lim_{n\to+\infty} \frac{1}{n} \log |T_{-n,a, \bar \omega} \bar v| = - \lambda_F(a).
$$
\end{itemize}
\end{theorem}

\begin{remark}
In the first claim of Theorem \ref{t.vector} it is crucially important that the initial vector (in our case $\left(\begin{smallmatrix} 1 \\ 0 \end{smallmatrix} \right)$) is fixed. Otherwise the statement would not hold. In the context of Example 2 above this is related to results on rank one perturbations, see \cite{DMS}, \cite[Theorem 3]{Gor}.
\end{remark}

\begin{remark}
It is interesting to notice that exponential decay of eigenfunctions (this is how Theorem \ref{t.vector} can be interpreted in the context of Example 2) is a specific property of Anderson Model that does not have to hold in general. For example, there are regimes where Almost Mathieu operator exhibits Anderson Localization with sub-exponential decay of eigenfunctions, see \cite[Theorem 1.2]{JL}.
\end{remark}

The following result is usually referred to as ``Schnol Theorem'', due to a similar result in the paper \cite{Sch} (see also \cite{Gl1, Gl2}):

\begin{theorem}\label{t.shnol}
Let $H:\ell^2(\mathbb{Z})\to \ell^2(\mathbb{Z})$ be an operator of the form
$$
Hu(n)=u(n-1)+u(n+1)+V(n)u(n),
$$
with a bounded potential $\{V(n)\}_{n\in \mathbb{Z}}$. If every polynomially bounded solution to $Hu=Eu$ is in fact exponentially decreasing, then $H$ has pure point spectrum, with exponentially decaying eigenfunctions. Similar statement holds for operators on $\ell^2(\mathbb{N})$ with Dirichlet boundary condition.
\end{theorem}

In continuum case Theorem \ref{t.shnol} follows also from \cite[Theorem 1.1]{Sim}. For the formal proof in the discrete case see, for example, \cite[Theorem 7.1]{Kir}; some improved versions of this result can be found in  \cite[Lemma 2.6]{JZ} or \cite{H}.

Now Theorem \ref{t.al} follows directly from Theorem~\ref{t.vector}, Remark~\ref{r.ex}, and Theorem~\ref{t.shnol}.

\subsection{Properties of finite matrix products and density of states measure}\label{ss.14}
Here we discuss the statement that forms the main technical part of the proof of Theorem \ref{t:product}, but is also of independent interest. Namely, we consider random matrices that depend on a parameter and satisfy the conditions $(A1)-(A4)$, and study the growth of products of large but finite number of these matrices. It turns out that for most parameters the growth is ``uniformly exponential'' with exponent prescribed by Furstenberg Thereom, but there are exceptional parameters that have well defined asymptotic distribution. This asymptotic distribution is a generalization of the {\it density of states measure}, the key notion in the theory of ergodic Schr\"odinger operators.

To give the formal statement we need the notion of a rotation number. In our case this is given by the following construction. For each $a\in J$ and each linear map $F_a(\omega)\in SL(2, \mathbb{R})$ denote by $f_{a,\omega}:\mathbb{S}^1\to \mathbb{S}^1$, $\Sc\cong \R P^1$, the projectivization of the map $F_a(\omega):\mathbb{R}^2\to \mathbb{R}^2.$  Recall that the map $F:\Omega\times J\to SL(2, \mathbb{R})$ is measurable, continuous in $a\in J$, and bounded (due to $(A2)$). Therefore one can also choose the lifts $\tilde f_{a,\omega}:\mathbb{R}\to \mathbb{R}$, $\tilde f_{a,\omega}(x)(\text{mod}\, 1)= f_{a,\omega}(x\,(\text{mod}\, 1))$,  in a measurable way, depending continuously on $a\in J$, and such that the set $\{\tilde f_{a,\omega}(0)\}_{\omega\in \Omega}\subseteq \mathbb{R}$ is uniformly bounded in $a\in J$.
\begin{prop}\label{p.rotnumlin}
There exists a continuous function $\rho:J\to \mathbb{R}$ such that for all $a\in J$, a.e. $\bo\in \Omega^{\mathbb{N}}$, and every $x\in \mathbb{R}$ the limit
$$
\lim_{n\to \infty}\frac{1}{n} \tilde f_{a,\omega_n}\circ \tilde f_{a,\omega_{n-1}}\circ \ldots \circ \tilde f_{a,\omega_2}\circ\tilde f_{a,\omega_1}(x)
$$
exists and is equal to $\rho(a)$.
\end{prop}

The number $\rho$ (that depends on a parameter $a\in J$) from Proposition \ref{p.rotnumlin} is called {\it rotation number}. \edited{For iterates of a homeomorphism of the circle the notion of rotation number goes back to Poincare; for cocycles it appeared, for example, in \cite{Her}.} Notice that it depends on the choice of the lifts $\{\tilde f_{a,\omega}\}_{\omega\in \Omega}$, but a different choice of the lifts will only add a constant to the function $\rho$. Also, it is clear that due to the monotonicity assumption $(A4)$ the function $\rho$ must be non-decreasing. Hence, it can be used to define a (non-atomic, non-probability) measure
on $J$ that we will denote $DOS$:
$$
DOS([b,b']) = \rho(b') - \rho(b) \quad \forall b<b', \, b,b'\in J
$$
(the notation reminds that this is a generalization of the Density of States Measure from the spectral theory of ergodic Schr\"odinger operators).
Moreover, Theorem \ref{t.gjt} (generalized Johnson's Theorem) together with the assumption $(A3)$ imply that $DOS$ has the whole
interval $J$ as its support.
\begin{remark}
The rotation number of a Schr\"odinger cocycle is the distribution function of the density of states measure (that can be defined in purely spectral terms) of the corresponding ergodic Schr\"odinger operator. This holds for a large class of ergodic potentials, not only for random potentials, see \cite{DS, JM}.
\end{remark}
In order to study the properties of finite products of matrices of length $n$, we split the interval of parameters $J$ into $N=[\exp(\sqrt[4]{n})]$ equal intervals $J_1,\dots,J_N$, and denote $J_i=[b_{i-1},b_{i}]$, $i=1, \ldots, N$. Notice that $b_0$ and $b_N$ are the endpoints of the interval $J$. To emphasize their independence of $n$, let us denote these endpoints by $b_-$ and $b_+$, so $J=[b_-, b_+]$.
 The number of small intervals $N$ and the whole construction depend on the length of the products $n$; to simplify the formulas we do not reflect it in the notation.

 By $U_\varepsilon(x)$ we denote the $\varepsilon$-neighborhood of the point~$x$.

\begin{theorem}\label{t:main}
For any $\eps>0$ there exist $n_0=n_0(\eps)$ and $\delta_0=\delta_0(\eps)$ such that
for any $n>n_0$ the following statement hold. With probability $1-\exp(-\delta_0 \sqrt[4]{n})$, there exists a
number $M\in \mathbb{N}$, exceptional intervals $J_{i_1},\dots,J_{i_M}$ (each of length $\frac{|J|}{N}$), and corresponding numbers $m_1,\dots, m_M\in \{1,\dots,n\}$, such that:
\vspace{4pt}
\begin{mainit}
\mitem \textbf{(Quantity)}
\label{i:m1}
The number $M$ is $n\eps$-close to $(\rho(b_+)-\rho(b_-))\cdot n$.
\vspace{4pt}
\mitem \textbf{(Uniform growth in typical subintervals)}
\label{i:m2}
For any $i$ different from $i_{1},\dots,i_{M}$, for any $a\in J_i$, and for any $m=1,\dots, n$ one has
$$
\log \| T_{m,a,\bo} \| \in U_{n\eps} (\lambda_F(a) m).
$$
\mitem \textbf{(Uniform growth in exceptional subintervals)}
\label{i:m3}
For any $k=1,\dots, M$, for any $a\in J_{i_k}$, and for any $m=1,\dots,m_k$ one has
$$
\log \| T_{m,a,\bo} \| \in U_{n\eps} (\lambda_F(a) m);
$$
for any $m=m_k+1,\dots, n$ one has
$$
\log \| T_{[m_k,m],a,\bo} \| \in U_{n\eps} (\lambda_F(a) (m-m_k)),
$$
where
$$
T_{[m_k,m],a,\bo}:=T_{m,a,\bo} T_{m_k,a,\bo}^{-1} = F_a(\omega_m)F_a(\omega_{m-1})\dots F_{a}(\omega_{m_k+1}).
$$
\mitem \textbf{(Cancellation)}
\label{i:m4}
For any $k=1,\dots,M$ there exists $a_{k}\in J_{i_k}$ such that for any $m=1,\dots, n$
\begin{equation}\label{e.IV}
\log \|  T_{m,a_k,\bo} \| \in U_{n\eps} (\lambda_F(a_k) \cdot \psi_{m_k}(m)),
\end{equation}
where
$$\psi_{m'}(m)=
\begin{cases}
m , & m<m', \\
2m'-m, & m'\le m < 2m', \\
m-2m', & m\ge  2m';
\end{cases}$$
in other words, for $m\ge m_k$ the parts of the product over the intervals $[1,m_k]$ and $[m_k,m]$ cancel each other in the best possible way.
\vspace{4pt}
\mitem \textbf{(Measure)}
\label{i:m5}
For each $k=1, \ldots, M$ consider the point $\left(\frac{m_k}{n}, a_{k} \right)\in [0,1]\times J.$  The measure
$$
\frac{1}{n}\sum_{k=1}^{M}\Dirac_{\left(\frac{m_k}{n}, a_{k} \right)}
$$
is $\eps$-close (in Levy-Prokhorov metric\footnote{\edited{If $\mu_1, \mu_2$ are two measures on a compact metric space $M$, the Levi-Prohorov distance can be defined as infimum of $\eps>0$ such that for any Borel $E\subseteq M$ one has $\mu_1(E)\le \mu_2(E^\eps)+\eps$ and $\mu_2(E)\le \mu_1(E^\eps)+\eps$, where $E^\eps$ is an $\eps$-neighborhood of $E$.}}, i.e. in \edited{a} metric that defines weak-* topology) to the measure $\text{Leb}\times \DOS$ on $[0,1]\times J.$
\end{mainit}
\end{theorem}

\begin{remark}
\edited{The Levy-Prokhorov metric in Theorem \ref{t:main} can be replaced by any other metric that induces the weak-* convergence, e.g. by ``earth mover's distance'' or, more generally, any of the Wasserstein metrics.}
\end{remark}
\begin{remark}
We expect that the statement on representation of $DOS$ measure as distribution of ``exceptional'' intervals in Theorem \ref{t:main} allows numerous and far reaching generalizations. For example, compare it with the notion of bifurcation current (supported on bifurcation locus) from \cite{DD}.
\end{remark}

\subsection{Structure of the paper}

In Section \ref{s.3} we show that almost surely the Furstenberg Lyapunov exponent gives an upper bound on upper Lyapunov exponent for all values of the parameter. This can be considered as a dynamical analog of Craig-Simon's result \cite[Theorem 2.3]{CS} on Schr\"odinger cocycles.

In Section \ref{s:discretization} we deduce the main result of the paper, Theorem \ref{t:product}, from the properties of finite matrix
products described in Theorem~\ref{t:main}.

Section \ref{s.5} is devoted to the proof of Theorem \ref{t:main}. This is the most technical part of the paper. In Section \ref{ss.5.1}, we introduce the language of projective dynamics on the circle and study possible behaviors of an image of a given point under finite random compositions of maps when the parameter is changing along a small interval. Proposition \ref{p:classes} gives the list of scenarios that exhaust all the possibilities with probability close to one. The rest of Section \ref{ss.5.1} provides an informal non-technical explanation how Theorem \ref{t:main} follows from Proposition \ref{p:classes}, and the main idea of the proof of Proposition \ref{p:classes}. Then, after providing technical tools (distortion control in Section \ref{ss:distortion}, large deviation estimates in Section \ref{ss.5.3}, and quantitative estimates on exponential contraction in Section \ref{ss.5.6}), we deduce parts {\bf II} and {\bf III} of Theorem \ref{t:main} from Proposition \ref{p:classes} in Section \ref{ss:norms}, part {\bf IV} -- in Section~\ref{ss:cancellation}, and parts {\bf I} and {\bf V} -- in Section \ref{ss:law}. In Section \ref{ss:jump-classes} we give the formal proof of Proposition~\ref{p:classes}.

Dynamical analog of Anderson Localization, Theorem \ref{t.vector}, is proven in Section~\ref{s.al}.

Finally, in Appendix \ref{s.johnson}  we provide a dynamical analog of Johnson's Theorem, that in the context of ergodic Schr\"odinger operators claims that a given energy belongs to the spectrum if and only if the corresponding Schr\"odinger cocycle is not uniformly hyperbolic. While this statement is certainly not surprising to the experts in spectral theory of ergodic Schr\"odinger operators, it is probably less known to the dynamical community, and we include it here formulated in the form convenient for a reader with background in dynamical systems.

\section{Upper bound for the upper limit}\label{s.3}

The following statement can be considered as a dynamical analog of Craig-Simon's result \cite[Theorem 2.3]{CS} on Schr\"odinger cocycles.

\begin{prop}\label{p.upper}
For a.e. $\bar\omega\in \Omega^{\mathbb{N}}$ and \emph{any} $a\in J$ one has $$\limsup_{n\to \infty}\frac{1}{n}\log \|T_{n, a, \bar \omega}\|\le \lambda_F(a).$$
\end{prop}

\begin{proof}
This event is an intersection of a countable number of events of the type
\begin{equation}\label{eq:Lyap-eps}
\limsup_{n\to \infty}\frac{1}{n}\log \|T_{n, a, \bar \omega}\|\le \lambda_F(a)+\eps  \ \ \text{for all}\ \ a\in J \quad
\end{equation}
along a sequence of values of $\eps>0$ that tend to zero. Hence, it suffices to show that each value $\eps>0$ the event (\ref{eq:Lyap-eps}) has full probability.

Fix $\eps>0$. Note that (due to the subadditive ergodic theorem) for any fixed $a\in J$ we have
$$
\lambda_F(a)=\lim_{n\to\infty} \frac{1}{n} \int_{|w|=n} \log \|T_{n,a, w}\| \, dP(w).
$$
In particular, for any $a\in J$ there exists $n_0=n_0(a)$ such that for any $n\ge n_0$
$$
\frac{1}{n} \int_{|w|=n} \log \|T_{n, a, w}\| \, dP(w) < \lambda_F(a)+\frac{\eps}{2}.
$$
As both $\lambda_F(a)$ and $T_{n,a, w}$ (for any fixed $n$) depend on $a$ continuously \edited{(see Remark \ref{r.cont})}, any $a$ is contained in a neighborhood $J_{a}$ such that for $n_0=n_0(a)$ one has
\begin{equation}\label{e.expect}
\frac{1}{n_0} \int_{|w|=n_0} \log \max_{a'\in J_{a}} \|T_{n_0, a', w}\| \, dP(w) < \min_{a'\in J_a} \lambda_F(a')+\eps.
\end{equation}
Extracting a finite subcover, we see that the whole interval $J$ is covered by finitely many such intervals $J_{a}$. Let us recall the notation
$$
T_{[m',m''],a,\bar\omega} = F_a(\omega_{m''})F_a(\omega_{m''-1})\dots F_{a}(\omega_{m'}).
$$
On each interval $J_a$, we have for any $a''\in J_a$ and $n=n_0(a)$
\begin{multline*}
\limsup_{N\to\infty} \frac{1}{N} \log \|T_{N, a'', \bar\omega}\| = \limsup_{k\to\infty} \frac{1}{kn} \log \|T_{kn, a'', \bar\omega}\|
\\ \le
\limsup_{k\to\infty} \frac{1}{k} \sum_{j=1}^k \frac{1}{n} \log \|T_{[(j-1)n+1, jn], a'', \bo}\|
\\ \le
\limsup_{k\to\infty} \frac{1}{k} \sum_{j=1}^k \frac{1}{n} \log \max_{a'\in J_{a}} \|T_{[(j-1)n+1, jn], a', \bo}\|.
\end{multline*}
The right hand side almost surely (and independently of $a''$) equals to
$$
\frac{1}{n} \int_{|w|=n} \log \max_{a'\in J_{a}} \|T_{n,a', w}\| \, dP(w) < \min_{a'\in J_a} \lambda_F(a')+\eps.
$$
Thus we get the desired estimate for the parameters from $J_{a}$. Intersecting finitely many such events, associated to the chosen intervals that form the finite subcover, we get \eqref{eq:Lyap-eps}. Proposition~\ref{p.upper} follows.
\end{proof}
\begin{remark}\edited{
While we assume that the assumptions $(A1)-(A4)$ hold throughout the paper, it is worth to note that the proof of Proposition \ref{p.upper} does not really use the monotonicity assumption $(A4)$.}
\end{remark}

One can combine the above arguments with the Large Deviation Theorem. 
This gives the following useful finite-$n$ upper bound:

\begin{prop}\label{l:upper-finite}
For any $\eps'>0$ there exists $c_3>0$ and $n_1\in\N$ such that for any $n>n_1$ with the probability at least $1-\exp(-c_3 n)$ the following statement holds.
For any $a\in J$ and any $m,m'$, $1\le m\le m'\le n$ one has
\begin{equation}\label{eq:m-all}
\log \|T_{[m, m'],a,\bo}\| \le n\eps' + \lambda_F(a) \cdot (m'-m).
\end{equation}
\end{prop}

\begin{proof}
It suffices to obtain an upper bound of the form $1-P_2(n) \exp(-c_3' n)$ for some $c_3'>0$ and a quadratic polynomial $P_2$. Indeed,
taking $c_3>0$, $c_3<c_3'$, we have for all sufficiently large $n$
$$
\exp(-c_3 n) > P_2(n) \exp(-c_3' n).
$$
For any given $\eps'>0$, set $\eps=\frac{\eps'}{4}$, and consider the finite cover of $J$ by intervals of the form~$J_a$, constructed in the
proof of Proposition~\ref{p.upper}. It is enough to obtain the desired estimate for each of them separately:
indeed, the probabilities that~\eqref{eq:m-all} does not hold at most add up.

Fix an interval $J_a\subset J$ such that for some $n_0=n_0(a)$ the inequality (\ref{e.expect}) holds. For all $a\in J$ and $\omega\in \Omega$ we have $\|F_{a'}(\omega)\|\le M$. Therefore, if $m'-m<\frac{n\eps'}{2\edited{\log} M}$, the
inequality~\eqref{eq:m-all} holds. So we have to handle less than $n^2$ pairs $(m,m')$ with $m'-m\ge \frac{\eps'}{2\edited{\log} M} n$.

Define $\Psi:\Omega^{n_0}\to \mathbb{R}$,
$$
\Psi(\omega_1, \omega_2, \ldots, \omega_{n_0})=\frac{1}{n_0}\log\max_{a'\in J_a}\|F_{a'}(\omega_{n_0})F_{a'}(\omega_{n_0-1})\ldots F_{a'}(\omega_{1})\|.
$$
Given $\bar \omega\in \Omega^{\mathbb{N}}$, set $\Psi_j(\bar \omega)=\Psi(\omega_{jn_0+1}, \ldots, \omega_{(j+1)n_0})$. Then $\{\Psi_j\}$ is a sequence of i.i.d. random variables on $\Omega^{\mathbb{N}}$. By (\ref{e.expect}) we have
$$
\mathbb{E}\Psi<\min_{a'\in J_a}\lambda_F(a')+\eps.
$$
Therefore, if we denote $\mathbb{P}=\mu^{\mathbb{N}}$, then due to the Large Deviation Theorem for random matrix products we have
$$
\mathbb{P}\left\{\frac{1}{k}\sum_{j=1}^k\Psi_j(\bar\omega)>\min_{a'\in J_a}\lambda_F(a')+2\eps \right\}\le e^{-\zeta k}
$$
for some $\zeta=\zeta(\eps)>0$.

Therefore, if $m'-m$ is large, $m'-m=n_0k+r$, where $0\le r<n_0$, $a''\in J_a$, and $M>0$ is given by $(A2)$, then
\begin{multline*}
\frac{1}{m'-m}\log \|T_{[m, m'],a,\bo}\|\le \\
\frac{1}{m'-m}\left(\sum_{j=0}^{k-1}\log\|T_{[m+n_0j, m+n_0(j+1)-1], a'', \bo}\|\right)+\frac{r\edited{\log}M}{m'-m}\le\\
\frac{1}{k}\sum_{j=0}^{k-1}\Psi_j(\bo^*)+\eps,
\end{multline*}
where $\bar\omega^*=\omega_m\omega_{m+1}\ldots$.

Hence
\begin{multline*}
    \mathbb{P}\left(\frac{1}{m'-m}\log \|T_{[m, m'],a'',\bo}\|>\eps'+\lambda_F(a'') \ \ \text{for some}\ \ a''\in J_a\right)\le \\
    \mathbb{P}\left(\frac{1}{m'-m}\max_{a'\in J_a}\log \|T_{[m, m'],a'',\bo}\|>4\eps+\min_{a'\in J_a}\lambda_F(a') \right)\le
    \\
    \mathbb{P}\left(\frac{1}{k}\sum_{j=0}^{k-1}\Psi_j(\bo^*)+\eps >4\eps+\min_{a'\in J_a}\lambda_F(a')\right)\le e^{-\zeta k}\le e^{-\zeta\frac{m'-m}{n_0}}<e^{-\left(\frac{\zeta}{n_0}\frac{\eps'}{2\edited{\log} M}\right)n}.
\end{multline*}
This completes the proof of Proposition \ref{l:upper-finite}.
\end{proof}

\section{Proof of parametric Furstenberg Theorem \\ \ via parameter discretization}\label{s:discretization}

Here we derive  Theorem~\ref{t:product} (parametric Furstenberg Theorem) from Theorem~\ref{t:main} (on properties of finite products of random matrices).

\begin{proof}[Proof of~Theorem~\ref{t:product}] Combining Borel-Cantelli Lemma with Theorem~\ref{t:main} we observe that for any $\eps>0$ \ $\mu^\mathbb{N}$-almost surely there exists $n_0=n_0(\eps)$ such that for any $n\ge n_0$ there are $M_n\in \mathbb{N}$ and exceptional intervals $J_{i_1, n}, J_{i_2, n}, \ldots J_{i_{M_n}, n}$ such that the properties {\bf I--V} from Theorem \ref{t:main} hold. Notice that comparing to the notation used in Theorem \ref{t:main} we add $n$ as an index to emphasize the dependence of these objects on $n$. Let us also define
$$
V_{n',\eps}:= \bigcup_{n\ge n'} \, \bigcup_{k=1,\dots, M_n} J_{i_k,n},
$$
and
$$H_{\eps}=\bigcap_{n'\ge n_0(\eps)} V_{n',\eps}.$$

{\bf Regular upper limit:} Due to Proposition \ref{p.upper} we only need to show that $\mu^\mathbb{N}$-almost surely for all $a\in J$ we have
\begin{equation}\label{e.uplim}
\limsup_{n\to\infty} \frac{1}{n}\log \|T_{n,a,\bo} \| \ge \lambda_F(a).
\end{equation}
If a given $a\in J$ does not belong to $H_{\eps}$, then it does not belong to exceptional intervals $J_{i_k,n}$ for all sufficiently large $n$. Therefore due to property~\ref{i:m2} from Theorem \ref{t:main} for all sufficiently large $n$ we have  $\log \|T_{n,a,\bo} \|\ge (\lambda_F(a)-\eps)n$, or $\frac{1}{n}\log \|T_{n,a,\bo} \|\ge \lambda_F(a)-\eps.$ Hence
\begin{equation}\label{e.n11}
\limsup_{n\to \infty}\frac{1}{n}\log \|T_{n,a,\bo} \|\ge \lambda_F(a)-\eps.
\end{equation}
If $a\in H_{\eps}$, there is an arbitrarily large $n$ such that $a\in J_{i_k,n}$ for some exceptional interval $J_{i_k,n}$. Consider the corresponding value $m_{k,n}$ and notice that the property~\ref{i:m3} from Theorem \ref{t:main} implies the following. If $\frac{m_{k,n}}{n}>\sqrt{\eps}$, then $\log \|T_{m_{k,n},a,\bo} \|\ge \lambda_F(a)m_{k,n}-\eps n$, or
\begin{equation}\label{e.n1}
\frac{1}{m_{k,n}}\log \|T_{m_{k,n},a,\bo} \|\ge \lambda_F(a) -\eps \frac{n}{m_{k,n}}\ge \lambda_F(a) -\sqrt{\eps}.
\end{equation}
If $\frac{m_{k,n}}{n}\le\sqrt{\eps}$, then
\begin{multline*}
    \log\|T_{n,a,\bo} \|\ge \log \|T_{[m_{k,n}, n],a,\bo} \| -\log \|T_{m_{k,n}-1,a,\bo} \|\ge\\
    \lambda_F(a)(n-m_{k,n})-\eps n-(\lambda_F(a)m_{k,n}+\eps n)= \\ \lambda_F(a)(n-2m_{k,n})-2\eps n\ge
    \lambda_F(a)n - (2\eps +2\lambda_F(a)\sqrt{\eps})n,
\end{multline*}
hence
\begin{equation}\label{e.n2}
\frac{1}{{n}}\log\|T_{n,a,\bo} \|\ge \lambda_F(a) -(2\eps +2\lambda_F(a)\sqrt{\eps}).
\end{equation}

Therefore, in any case from (\ref{e.n1}) and (\ref{e.n2}) we get
\begin{equation}\label{e.n22}
\limsup_{n\to \infty}\frac{1}{n}\log \|T_{n,a,\bo} \|\ge \lambda_F(a)-\max(\sqrt{\eps}, 2\eps +2\lambda_F(a)\sqrt{\eps}).
\end{equation}
Finally, applying (\ref{e.n11}) and (\ref{e.n22}) along a sequence of values of $\eps>0$ that tends to zero, we observe that $\mu^\mathbb{N}$-almost surely  (\ref{e.uplim}) holds, and hence the first claim of Theorem~\ref{t:product} (on regular upper limit) follows.

\vspace{4pt}

{\bf $G_\delta$ vanishing:} For each $n, p\in \mathbb{N}$ introduce the set
$$
W_{n, p}=\left\{a\in J\ |\ \text{for some}\ m\ge n\ \text{we have} \ \frac{1}{m}\log\|T_{m, a, \bo}\|<\frac{2}{p}\right\}.
$$
We claim that $W_{n,p}$ is open and dense for any $n,p\in \mathbb{N}$. Indeed, it is clear that each set $W_{n,p}$ is open.  Apply Theorem \ref{t:main} for $\eps=\frac{1}{p}$.
 Property \ref{i:m5} \edited{and the fact that DOS has the whole interval $J$ as its support imply that the set of parameters $\{a_k\}$ for which $\frac{m_k}{n}\in (1/4,1/2)$ is $r(\eps)$-dense in~$J$,}
where 
$r(\eps)\to 0$ as $\eps\to 0$ (or, equivalently, $p\to \infty$).


For each sufficiently large $n$ and each such $k$ with $\frac{m_k}{n}\in (1/4,1/2)$, the property \ref{i:m4} of Theorem~\ref{t:main} implies that
$$
\frac{1}{2m_k}\log \|T_{2m_k,a_k,\bo} \| < \frac{n\eps}{2m_k} = \frac{\eps}{2m_k/n} \le 2\eps=\frac{2}{p}
$$
for some $a_k\in J_{i_k,n}$. Hence, for any $n$ and any $p$ the set $W_{n,p}$
 is $r(\eps)$-dense in~$J$, where $\eps=\frac{1}{p}$. Since $W_{n,p'}\subseteq W_{n,p}$ if $p'\ge p$, this implies that $W_{n,p}$ is dense in~$J$.

Since $W_{n, p}$ is open and dense in $J$, the intersection $\bigcap_{n,p=1}^\infty W_{n,p}$ is a dense $G_\delta$-subset of $J$, and for any $a\in \bigcap_{n,p=1}^\infty W_{n,p}$ we have
$$
\liminf_{n\to \infty}\frac{1}{n}\log \|T_{n, a, \bo}\|=0.
$$
{\bf Hausdorff dimension:} First of all, notice that $H_{\eps}\subseteq J$ has zero Hausdorff dimension. Indeed, $H_{\eps}$ is contained in $V_{n',\eps}$, which is covered by $\left\{J_{i_k, n}\right\}_{n\ge n', \, k\le M_n}$. Taking into account property {\bf I} from Theorem \ref{t:main}, $d$-volume of this cover can be estimated as follows:
$$
\sum_{n\ge n'} M_n\left(\frac{|J|}{N(n)}\right)^d\le  \sum_{n\ge n'} \const \cdot n \frac{|J|^d}{N(n)^d} \le \const' \sum_{n\ge n'} n \exp(-d  \sqrt[4]{n}).
$$
Therefore it tends to zero as $n'$ tends to $\infty$. Since this holds for any $d>0$, we have $\text{dim}_H\, H_\eps=0$.

If $a\not\in H_\eps$, then due to property \ref{i:m2} from Theorem \ref{t:main} for all sufficiently large $n$ we have $\frac{1}{n}\log \|T_{n,a,\bo} \|\ge \lambda_F(a)-\eps,$ hence
$$
\liminf_{n\to\infty} \frac{1}{n}\log \|T_{n,a,\bo} \| \ge \lambda_F(a)-\eps.
$$
Taking a countable union of sets $H_\eps$ over a sequence of values of $\eps>0$ that tend to zero, we get a set of zero Hausdorff dimension that contains all values of $a\in J$ such that
$$
\liminf_{n\to\infty} \frac{1}{n}\log \|T_{n,a,\bo} \| < \lambda_F(a).
$$
This proves the last part of Theorem \ref{t:product}.
\end{proof}

\section{On finite products of random matrices}\label{s.5} 

In this section we prove Theorem~\ref{t:main}.

\subsection{Key proposition and the outline of the proof}\label{ss.5.1}

Theorem~\ref{t:main} describes the ``most-probable'' behaviour of a finite long product of random matrices, handling ``uniformly'' sufficiently small intervals of parameter. Hence, it is natural to inquire how does such a product change as we change the parameter. The answer, stated in terms of the corresponding projective dynamics on the circle and its lift to the real line, is given by Proposition~\ref{p:classes} below, and it is a key ingredient of the proof of Theorem~\ref{t:main}. We will  formulate it (with a geometric interpretation of its conclusion in Remark~\ref{r:classes} below), and then provide an informal outline of the rest of the proof of Theorem~\ref{t:main}.


First, together with the initial linear dynamics of $SL(2,\R)$-matrices $F_a(\omega)$, $\omega\in \Omega$, we consider
their projectivizations that act on the circle of directions $\Sc\cong \R P^1$, and lift this action to the action on
the real line $\R$ for which $\Sc=\R/\Z$:
let
$$
f_{a, \omega}:\Sc\to \Sc
$$
be the map induced by $F_a(\omega):\mathbb{R}^2\to \mathbb{R}^2$, and let
$$
\tf_{a, \omega}:\mathbb{R}\to \mathbb{R}
$$
be \edited{a} lift of $f_{a, \omega}:\mathbb{S}^1\to \mathbb{S}^1$.
The lifts $\tf_{a, \omega}$ can be chosen continuous in $a\in J$ and measurable in $\omega\in \Omega$. Also, denote by
$$
f_{n, a, \bar \omega}:\Sc\to \Sc
$$
the map induced by $T_{n,a, \bar \omega}:\mathbb{R}^2\to \mathbb{R}^2$, and \edited{define}
$$
\tf_{n, a, \bar \omega}:\mathbb{R}\to \mathbb{R}
$$
 to be the lift of $f_{n, a, \bar \omega}:\mathbb{S}^1\to \mathbb{S}^1$ \edited{given by $\tf_{n, a, \bar \omega}=\tf_{a, \bar \omega_n}\circ \ldots \circ \tf_{a, \bar \omega_1}$}. For any fixed value of parameter $a\in J$, the (exponential)
growth of norms of $T_{m,\omega, a}$ is related to the (exponential) contraction on the circle of the projectivized dynamics.
Namely, standard easy computation shows that for a unit vector $v_0$ in the direction given by the point $x_0$, one has
\begin{equation}\label{eq:der-norm}
f'_{n,a,\bo} (x_0) = \frac{1}{\|T_{n,a,\bo}(v_0)\|^2}.
\end{equation}

Fix some point $x_0\in \Sc$, for example, the point that corresponds to the vector $\left( \begin{smallmatrix} 1 \\ 0 \end{smallmatrix} \right)$. Denote by $\tx_0\in [0,1)$ its lift to $\mathbb{R}^1$.
Recall that the interval $J=[b_-, b_+]$ was divided into $N=[\exp(\sqrt[4]{n})]$ equal intervals $J_1,\dots,J_N$ that were denoted by $J_i=[b_{i-1},b_{i}]$, $i=1, \ldots, N$.

Let $\tx_{m,i}$ be the image of $\tx_0$ after $m$ iterations of the lifted maps that correspond to the value of the parameter $b_i$,
$$
\tx_{m,i}:=\tf_{m,b_i,\bar\omega}(\tx_0)
$$
(we omit here the explicit indication of the dependence on $\bo$), and
let
\begin{equation}\label{e.Xmi}
X_{m,i}:=[\tx_{m,i-1},\tx_{m,i}]
\end{equation}
be the interval that is spanned by $m$-th (random) image of the initial point $\tx_0$ while the
parameter~$a$ varies in~$J_i=[b_{i-1},b_{i}]$.


\begin{figure}[!h!]
\begin{center}
\includegraphics{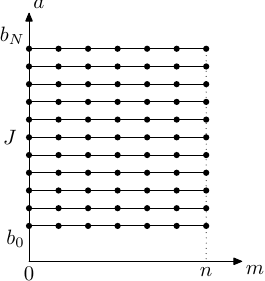} \qquad \qquad  \includegraphics{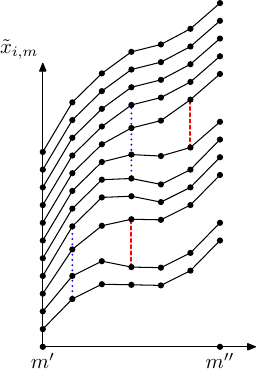}
\end{center}
\caption{Left: a grid of parameters and numbers of iterations. Right: graphs of $\tilde x_{m,i}$, where $m$ varies in a subinterval $0<m'<m<m''$, with the occurring suspicious intervals marked with blue (dotted) lines and the jumping ones with red (dashed) lines.}\label{f:jumps-2}
\end{figure}

\begin{figure}[!h!]
\begin{center}
\includegraphics[scale=0.8]{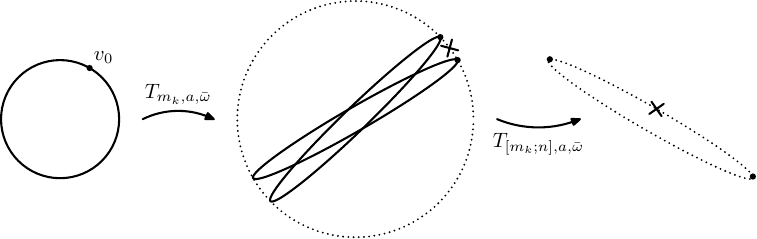}
\end{center}
\caption{Left: a unit circle with a marked point $x_0$. Center: its image after $m_k$ iterations under two different values of parameter $a=b_{i_k-1}$ and $a=b_{i_k}$,
together with a most contracted direction for $T_{[m_k,n],a,\bar{\omega}}$ for some $a\in J_{i_k}$, marked by a cross. Right: final image after $n$ iterations; note that the images of $x_0$ are almost opposite, meaning that they have made a full turn on the projective line of the directions.}\label{f:jumps-1}
\end{figure}


\begin{prop}[Types of the behavior]\label{p:classes}
For any $\eps'>0$ there exists $c_1>0$ such that for any sufficiently large $n$
with the probability at least $1-\exp(-c_1 \sqrt[4]{n})$ the following  holds. For each $i=1,\dots,N$ the lengths $|X_{m,i}|$ behave in one of the three possible ways:
\begin{itemize}
\item[\textbf{$\bullet$}]\textbf{(Small intervals)} The lengths $|X_{m,i}|$ do not exceed $\eps'$ for all $m=1,\dots,n$;
\item[\textbf{$\bullet$}]\textbf{(Opinion-changers)} There is $m_0$ such that $|X_{m_0, i}|>\eps'$, and
$$
|X_{m,i}|\edited{\le}\eps' \quad \text{if } m<m_0 \ \text{ or } \  m>m_0+\eps'n;
$$
\item[\textbf{$\bullet$}]\textbf{(Jump intervals)} There is $m_0$ such that $|X_{m_0,i}|>\eps'$, and
$$
|X_{m,i}|\edited{\le}\eps' \quad \text{if } m<m_0,
$$
$$
1<|X_{m,i}|\edited{\le} 1+\eps' \quad \text{if } m>m_0+\eps'n.
$$
\end{itemize}
\end{prop}
\begin{remark}\label{r:classes}
Let us explain the geometrical meaning of Proposition \ref{p:classes}. Consider the images $\tf_{n,a,\bo}(\tilde x_0)$
as a sequence of functions of the parameter~$a$. As the number $n$
of iterations grows, the increment of this function on $J$ grows asymptotically linearly in $n$, and is
roughly equal to~$(\rho(\bmax)-\rho(\bmin))\cdot n$. However, this increment is not distributed uniformly on $J$; rather,
most of it comes from ``jumps by one'', when a very small increment of the parameter $a$ leads to the increment of the
image by~$1$; see Figure~\ref{f:jumps-2}.

The latter happens exactly at the exceptional intervals $J_{i_k}$.
Namely, when we increase the parameter from $b_{i_k-1}$ to $b_{i_k}$,
the maps $f_{m_k,a,\bo}$ and $f_{[m_k,n],a,\bo}$ do not
change much and continue to be hyperbolic. However, the image of $\bar x_0$
under $f_{m_k,a,\bo}$ (that is exponentially close to the image of the most expanded
direction for $T_{m_k,a,\bo}$) moves past the most repelling point of $\tf_{[m_k,n],a,\bo}$
(that corresponds to the most contracted
direction for $T_{[m_k,n],a,\bo}$); see Figure~\ref{f:jumps-1}.

In particular, we find the values $a_k\in J_{i_k}$ (from property \ref{i:m4} in Theorem \ref{t:main})
as those where the image of the most expanded direction
under $T_{m_k,a_k,\bo}$ coincides with the most contracted direction
of~$T_{[m_k,n],a_k,\bo}$. Proposition \ref{p:classes} provides the formal justification of (part of) this picture.
\end{remark}

Now the proof of Theorem \ref{t:main} splits into two parts: deduction of Theorem \ref{t:main} from Proposition \ref{p:classes}, and the proof of Proposition \ref{p:classes}. Since both of these parts are somewhat technical, we start here with a brief informal outline of the proofs.

First, let us discuss how Proposition \ref{p:classes} will be used to prove Theorem~\ref{t:main}.
Consider the random products of matrices for the parameter values $b_j$, $j=0,1,\dots, N$. For each individual parameter value $a=b_j$,
the growth (with large probability) is exponential, as prescribed by Furstenberg Theorem, hence the
derivatives $\tf'_{m,b_j, \bo}(\tx_0)$ decrease exponentially.
Moreover, due to the (uniform in parameter)
Large Deviations Theorem (\cite[Theorem 4]{T}, reproduced below as Theorem \ref{t.diviations}), the probability of ``irregular behaviour'' is
exponentially small. Hence, as we have chosen the number $N$ to be subexponential
in~$n$, with the probability exponentially close to~1 the
derivatives $\tf'_{m,b_j, \bo}(\tx_0)$ admit a well controlled exponentially decreasing
bound for all $j=0,1,\dots, N$; this argument is formalized in Lemma~\ref{l:LD-RD} below.


Next, for each interval $J_i$ consider the increments of the images of $\tx_0$ over $J_i$, that is, the lengths of the corresponding intervals~$X_{m',i}$, $m'=1,\ldots, m$.
A modification of the standard distortion control technique implies that if $\sum_{m'=1}^m |X_{m',i}|$ is sufficiently small, then the logarithms of
the derivatives  of all the maps $\tf_{m,a,\bo}$, $a\in J_j$, at $\tx_0$ are sufficiently close
to each other.
This implies that the derivatives at~$\tx_0$ stay exponentially decreasing uniformly in
$a\in J_i$, and hence the products $T_{n,a,\bo}$ admit the desired exponential
growth lower bound uniformly in~$a$ on such~$J_i$. This argument handles both the ``small'' and the ``opinion-changing'' intervals
from  Proposition~\ref{p:classes}: in both these cases, the sum
of the lengths of $X_{m',i}$  does not exceed $2\eps'n$, which is sufficient to obtain the
desired control (see Lemma~\ref{l:DC} for the distortion control and Proposition~\ref{p:derivatives-control} for the
extension of the ``hyperbolic'' behavior inside the parameter intervals). In these cases the interval $J_i$ is not exceptional, and combining the obtained lower estimates with the upper estimates from Proposition~\ref{l:upper-finite} proves part \ref{i:m2} of Theorem \ref{t:main}.

The ``jump'' intervals from Proposition~\ref{p:classes} correspond to the exceptional intervals
$J_{i_k}$ from Theorem~\ref{t:main}. For these intervals,
we still have a sufficient control on the distortion ``before the jump'', thus obtaining a
uniform bound on the growth of the norm of the products $T_{m,a,\bo}$ for $m\le m_k$.
At the same time, ``after the jump'' we consider intervals $X'_{m,i}:=[x_{m,i-1}+1,x_{m,i}]$,
that are again of controlled lengths for all $m\ge m_k+\eps' n$. Applying again the control
of the distortion, we get a uniform lower bound for the norm of the product $T_{[m_k,n],a,\bo}$ for
all $a\in J_{i_k}$, thus establishing part \ref{i:m3} of Theorem~\ref{t:main}.

The obtained description for the norms of the maps $T_{m_k,a,\bo}$
and~$T_{[m_k,n],a,\bo}$ for $a\in J_{i_k}$ together with the ``jump by 1''
from Proposition~\ref{p:classes} implies that for some parameter value $a_k\in J_{i_k}$ the image
of the most expanded by $T_{m_k,a,\bo}$ direction will coincide with the
most contracted by $T_{[m_k,n],a,\bo}$ direction. This will imply the part \ref{i:m4} (Cancellation) of Theorem \ref{t:main}, see Section~\ref{ss:cancellation} for details.

Finally, the parts {\bf I} (Quantity) and  \ref{i:m5} (Measure) of Theorem \ref{t:main} are obtained by the same argument.
Namely,  most of the increment
$x_{m,i}-x_{m,i'} \approx m(\rho(b_i)-\rho(b_{i'}))$
comes from the ``jumps'' that has already occurred at this moment. Hence, the
number of exceptional intervals $J_{i_k}$ such that $m_k<m$ and $a_k\in [b_{i'}, b_{i}]$ can be approximated
as $m\cdot(\rho(b_i)-\rho(b_{i'}))$. Thus, if we denote
$$
\Pi=\left[0,\frac{m}{n}\right]\times
\left[b_{i',}b_{i}\right]\ \ \ \text{and}\ \ \ \xi=\frac{1}{n}\sum_{k=1}^{M}\Dirac_{\left(\frac{m_k}{n}, a_{k} \right)},
$$
then (with large probability) we have
$$
\xi(\Pi) \approx \frac{m}{n} \cdot (\rho(b_i)-\rho(b_{i'})) = \text{Leb}\times\DOS(\Pi).
$$
These arguments are formalized in Section~\ref{ss:law}, which concludes
the proof of Theorem~\ref{t:main}.

\vspace{4pt}

Let us now describe the main idea of the proof of Proposition~\ref{p:classes}.

Consider the lengths of all the intervals $X_{m,i}$, $i=1, \ldots, N$, $m=1, \ldots, n$. Let us say that
an interval $J_{i}$ is \emph{suspicious} if at some $m$ we have $|X_{m,i}|>\eps'$.
All the non-suspicious intervals are automatically ``small'' and satisfy the conclusion of Proposition~\ref{p:classes}.

The sum of lengths of all $|X_{m,i}|$ over all $i$ and $m$ grows with $n$ as
\begin{multline*}
\sum_{i=1}^N\sum_{m=1}^n |X_{m,i}| = \sum_{m=1}^n (\tx_{m,N}-\tx_{m,0}) \sim
\sum_m m(\rho(\bmax)-\rho(\bmin)) \sim \\
(\rho(\bmax)-\rho(\bmin)) \frac{n^2}{2} = O(n^2),
\end{multline*}
hence there are at most $\sim \frac{n^2}{2\eps'}(\rho(\bmax)-\rho(\bmin))  = O(n^2)$ suspicious intervals.

Suppose now that $J_i$ is a suspicious interval, and $m$ is the first iterate when $|X_{m,i}|>\eps'$. With large probability, under subsequent iterates  the images of the points $\tx_{m, i-1}$ and $\tx_{m, i}$ either quickly become  very close, or diverge to a distance that is very close to~$1$,
and stays exponentially close or at the distance close to~$1$ under all the remaining iterates. Indeed, for any specific value of the parameter $a$ Furstenberg Theorem implies that with large probability a given pair of points on the circle converge exponentially fast under a random sequence of projective maps. In our case the points $\tx_{m, i-1}$ and $\tx_{m, i}$ will be iterated by the sequence of maps that correspond to different values of parameter, namely $b_{i-1}$ and $b_i$, but since these values are very close to each other, it does not change the picture qualitatively. Finally, the probability of such a behavior approaches~$1$ faster than any inverse power of $n$, thus for all sufficiently large $n$ with large probability this description holds simultaneously for all the suspicious intervals.

 The formal presentation of these arguments is contained in Section~\ref{ss:jump-classes}.

\subsection{Distortion control}\label{ss:distortion}
The distortion estimates is a standard tool in smooth one dimensional dynamics, e.g. see \cite[Lemma 12.1.3]{KH} and \cite[Lemma 6.1]{W2}. In our case we need the distortion estimates for compositions of different but very close to each other maps. Here is the statement that we need:

\begin{lemma}[Distortion control]\label{l:DC}
For any $\bo\in \Omega^{\mathbb{N}}$, $\bo=\omega_1 \omega_2 \ldots \omega_m \ldots$, the following holds. Given $m'<m''$, $y_1<y_2$, and $\bar a_1<\bar a_2$, define the sequence of intervals $Y_{m}=[y_{m,1},y_{m,2}]$, $m=m',...,m''$,
 by
$$ 
y_{m',j}=y_j, \quad  y_{m+1,j}=\tf_{\bar a_j ,\omega_m}(y_{m,j}), \quad j=1,2, \ \  m=m',...,m''-1.
$$
Then for any $\bar a_3 \in [\bar a_1,\bar a_2]$, any $m=m',\dots,m''$, and any $y_3\in [y_1,y_2]$ we have
$$
\left| \log \tilde f'_{[m',m], \bar a_3,\bar\omega}(y_3)- \log \tilde f'_{[m',m],\bar a_1,\bar\omega}(y_1) \right| \le \kappa \sum_{k=m'}^{m''-1} |Y_k| + C |\bar a_2-\bar a_1| \cdot (m''-m'),
$$
where the constants $\kappa$ and $C$ are defined by 
$$
\kappa:= \sup_{y\in \mathbb{R}^1, \, \omega\in\Omega,\, a\in J} |\partial_y \log \tilde f'_{a,\omega}(y)|,
\quad C:= \sup_{y\in \mathbb{R}^1, \, \omega\in\Omega,\, a\in J} |\partial_a \log \tilde f'_{a,\omega}(y)|.
$$
\end{lemma}
\begin{proof}
By the monotonicity assumption, for any $m$ and any $a \in [\bar a_1, \bar a_2]$ we have $\tf_{a,\omega_m}(Y_m)\subset Y_{m+1}$.
The difference of logarithms can be estimated as
\begin{multline}
\left| \log \tf'_{[m',m], \bar a_3, \omega}(y_3)- \log \tf'_{[m',m], \bar a_1, \omega}(y_1) \right| =
\\ = \left| \sum_{k=m'}^{m-1} \log \tf'_{\bar a_3,\omega_k} (y_{k,3}) -
\log \tf'_{\bar a_1,\omega_k} (y_{k,1})  \right| \le
\\
\le \sum_{k=m'}^{m-1} \left( \left | \log \tf'_{\bar a_3,\omega_k} (y_{k,3}) -
\log  \tf'_{\bar a_1,\omega_k} (y_{k,3})  \right| +  \left | \log \tf'_{\bar a_1,\omega_k} (y_{k,3}) -
\log \tf'_{\bar a_1,\omega_k} (y_{k,1})  \right| \right) \le
\\
\le  \sum_{k=m'}^{m-1} |\bar a_3-\bar a_1| \cdot \sup_{y,\omega, a} |\partial_a \log \tf'_{a,\omega}(y)| + \sum_{k=m'}^{m-1} |y_{k,3}-y_{k,1}|  \cdot \sup_{y,\omega, a} |\partial_y \log \tf'_{a,\omega}(y)| \le \\
\le C  |\bar a_2-\bar a_1|  \cdot (m''-m') + \kappa  \sum_{k=m'}^{m''-1} |Y_k|.
\end{multline}
\end{proof}
Another estimate that we will need shows how fast nearby points can diverge under iterates of different but close maps.
\begin{lemma}\label{l.shift}
In notations of Lemma \ref{l:DC}, we have
\begin{equation}\label{eq:shift}
|y_{m'', 1}-y_{m'', 2}|\le \Lx^{m''-m'}|y_{m', 1}-y_{m', 2}|+\La(m''-m')\cdot \Lx^{m''-m'-1}|\bar a_2-\bar a_1|,
\end{equation}
where $\Lx=  \sup_{y\in \mathbb{R}^1, a\in J, \omega\in \Omega}|\tf'_{a, \omega}(y)|$ and $\La=\sup_{y\in \mathbb{R}^1, a\in J, \omega\in \Omega}|{\partial_a}\tf_{a, \omega}(y)|$ are the Lipschitz constants for the maps~$\tf_{a,\omega}(y)$ in space and parameter directions respectively.
\end{lemma}
\begin{proof}
By induction. The base,~$m''=m'$, is evident: in this case, left and right hand sides of~\eqref{eq:shift} coincide. For the induction step, once $m''>m'$, we decompose the difference $|y_{m'', 1}-y_{m'', 2}|=|\tf_{\bar a_1, \omega_{m'}}(y_{m',1})-\tf_{\bar a_2, \omega_{m'}}(y_{m',2})|$ into two parts:
\begin{multline*}
|y_{m'', 1}-y_{m'', 2}|\le  |\tf_{\bar a_1, \omega_{m''}}(y_{m''-1,1})-\tf_{\bar a_1, \omega_{m''}}(y_{m''-1,2})| \\+|\tf_{\bar a_1, \omega_{m''}}(y_{m''-1,2})-\tf_{\bar a_2, \omega_{m''}}(y_{m''-1,2})|.
\end{multline*}
The first summand does not exceed $\Lx |y_{m''-1,1}-y_{m''-1,2}|$, the second one does not exceed $\La\cdot |\bar a_1-\bar a_2|$, as $\Lx$ and $\La$ are Lipschitz constants in the circle- and parameter directions respectively. Applying the induction assumption (and using the inequality $\Lx\ge 1$), we finally get
\begin{multline*}
    |y_{m'', 1}-y_{m'', 2}|\le \Lx |y_{m''-1, 1}-y_{m''-1, 2}|+\La|\bar a_1-\bar a_2|\le \\
    \Lx\left(\Lx^{m''-m'-1}|y_{m', 1}-y_{m', 2}|+\La(m''-m'-1)\Lx^{m''-m'-2}|\bar a_2-\bar a_1|\right)+ \La|\bar a_1-\bar a_2|= \\
    \Lx^{m''-m'}|y_{m', 1}-y_{m', 2}|+\La(m''-m'-1) \cdot \Lx^{m''-m'-1}|\bar a_2-\bar a_1|+\La|\bar a_1-\bar a_2|\le \\
        \Lx^{m''-m'}|y_{m', 1}-y_{m', 2}|+\La(m''-m')\cdot \Lx^{m''-m'-1}|\bar a_2-\bar a_1|.
\end{multline*}
\end{proof}

\subsection{Large deviations: convenient versions}\label{ss.5.3}

Here we formulate several versions of Large Deviation Theorem in the context of random matrix products that will be specifically useful in our setting.

Let us first formulate the classical Large Deviation Theorem for the random matrix products. Initially it was obtained in \cite{L}, see also \cite{BL, BQ}. \edited{Improved versions of the Large Deviation Theorem, with explicit relation between $\varepsilon$ and $\zeta$ (in the notations of Theorem \ref{t.diviations}) as well as with relaxed or removed conditions on irreducibility of the cocycle, were obtained by Duarte and Klein in \cite{DK1, DK2}.} Here we will use the version of Large Deviation Theorem that is uniform in the parameter.

\begin{theorem}[Proposition 3.6 from \cite{BDFGVWZ}, Theorem 4 from \cite{T}]\label{t.diviations}
For each $\eps > 0$ there exists an $\zeta > 0$ such that for all
$\|u\| = 1$,
$$
\mathbb{P}\left\{\left|\frac{1}{n}\log \|T_{n, a, \bar\omega}u\|-\lambda_F(a)\right|>\eps\right\}<e^{-\zeta n}
$$
for all $a\in J$.
\end{theorem}

%

Let us recall that together with the random products of matrices $\{F_a(\omega)\}$ we consider the random dynamics of corresponding projective maps $\{f_{a, \omega}\}$ and their lifts $\{\tf_{a,\omega}\}$. By (\ref{eq:der-norm}), if
 $v_0$ is a unit vector in the direction given by the point $x_0\in \Sc$, and $\lim_{n\to \infty} \frac{1}{n}\edited{\log}\|T_{n, a, \bo}(v_0)\|=\lambda_F(a)$, then $\lim_{n\to \infty} \frac{1}{n}\edited{\log}|\tf'_{n, a, \bo}(\tx_0)|=-2\lambda_F(a)$. Let us denote
 $$
 \lambda_{RD}(a)=-2\lambda_F(a).
 $$
 From Theorem~\ref{t.diviations}  one can deduce the following statement:
\begin{lemma}\label{l.prelimLD}
For any $\ep>0$ there exists $\zeta_1>0$ such that for all  sufficiently large $n\in \mathbb{N}$  the following holds. For any $a\in J$, any given $0\le m_1<m_2\le n$, and $\tx_0\in \mathbb{R}$ with probability at least $1-\exp(-\zeta_1n)$ one has
\begin{equation}\label{eq:LD-RD}
\log \tf'_{[m_1, m_2],a,\bar\omega}(\tf_{m_1, a, \bar\omega}(\tx_0)) \in U_{\ep n} (\lR(a)\cdot (m_2-m_1) ).
\end{equation}
\end{lemma}
\begin{remark}
Notice that in the case $m_1=0, m_2=n$ the statement of Lemma \ref{l.prelimLD} turns into Theorem~\ref{t.diviations}.
\end{remark}
\begin{proof}
Set
$$
\eps^*=\min\left(\frac{\ep}{2\sup_{\tx\in \mathbb{R}, a\in J, \omega\in \Omega}|\log\tf_{a,\omega}'(\tx)|}, \frac{\ep}{2\max_{a\in J}\lambda_{RD}(a)}\right).
$$
If $m_2-m_1<\eps^*n$, then
\begin{multline*}
\left|\log \tf'_{[m_1,m_2], a, \bar\omega}(\tx_0)\right|\le \sum_{k=m_1+1}^{m_2}\left|\log\tf'_{a, \omega_k}(\tf_{k-1, a, \bar\omega}(\tx_0)\right|\le \\  (m_2-m_1)\cdot \sup_{\tx\in \mathbb{R}}\sup_{\omega\in \Omega}\sup_{a\in J}|\log\tf_{a,\omega}'(\tx)|\le (m_2-m_1)\frac{\ep}{2\eps^*}\le \frac{\ep}{2}n,
\end{multline*}
and
$$
\lambda_{RD}(a)(m_2-m_1)<\eps^*\lambda_{RD}(a)n\le \frac{\ep}{2}n.
$$
Therefore,
$$
\log \tf'_{[m_1,m_2], a, \bar\omega}(\tx_0)\in U_{\ep n}(\lambda_{RD}(a)(m_2-m_1)).
$$
If $\eps^*n\le m_2-m_1\le n$, then by Theorem~\ref{t.diviations} we have
\begin{multline*}
\P\left(\log \tf'_{[m_1,m_2], a, \bar\omega}(\tf_{m_1, a, \bar\omega}(\tx_0))\not\in U_{\ep n}(\lambda_{RD}(a)(m_2-m_1))\right)\le \\
\P\left(\log \tf'_{[m_1,m_2], a, \bar\omega}(\tf_{m_1, a, \bar\omega}(\tx_0))\not\in U_{\ep (m_2-m_1)}(\lambda_{RD}(a)(m_2-m_1))\right)\le \\ e^{-\zeta(m_2-m_1)}\le e^{-\zeta\eps^*n}.
\end{multline*}
Hence, Lemma \ref{l.prelimLD} holds with $\zeta_1=\eps^*\zeta$.
\end{proof}
Let us recall that the interval $J$ is divided into $N=[\exp(\sqrt[4]{n})]$ equal subintervals $J_1,\dots,J_N$ detoted  $J_i=[b_{i-1},b_{i}]$, $i=1, \ldots, N$. With large probability (\ref{eq:LD-RD}) holds simultaneously for all possible $m_1, m_2$ with $0\le m_1<m_2\le n$ and all parameter values that form the grid $\{b_0, b_1, \ldots, b_N\}$. Namely, the following statement holds:
\begin{lemma}\label{l:LD-RD}
For any $\ep>0$ there exists $\zeta_2>0$ such that for all  sufficiently large $n\in \mathbb{N}$ the following holds. For a given $\tx_0\in \mathbb{R}$ with probability at least $1-\exp(-\zeta_2n)$ one has
\begin{equation}\label{eq:LD-RD-2}
\log \tf'_{[m_1, m_2],b_i,\bar\omega}(\tf_{m_1, \edited{b_i}, \bar\omega}(\tx_0)) \in U_{\ep n} (\lR(b_i)\cdot (m_2-m_1) ).
\end{equation}
 for all $m_1, m_2$ with $0\le m_1<m_2\le n$ and all $i=0,1,\dots, N$.
\end{lemma}

\begin{proof}
Let $\zeta_1$ be given by Lemma \ref{l.prelimLD}, and take any positive $\zeta_2<\zeta_1$. For a given $a\in\{b_0,b_1,\dots, b_N\}$ and given $m\in \{1,\dots, n\}$ the event \eqref{eq:LD-RD} holds with probability at least $1-\exp(-\zeta_1 n)$. Intersecting the events \eqref{eq:LD-RD} for all  $a\in\{b_0,b_1,\dots, b_N\}$ and  all $m_1, m_2=0, 1,\dots, n$ with $m_1<m_2$  we observe that  \eqref{eq:LD-RD-2} holds with probability at least $1-\frac{n(n+1)}{2}(N+1)\exp(-\zeta_1 n)$. Since $N=[\exp(\sqrt[4]{n})]$ and $\zeta_2<\zeta_1$, we get
$$
1-\frac{n(n+1)}{2}(N+1)\exp(-\zeta_1 n)> 1-\exp(-\zeta_2n)
$$
for all sufficiently large $n$.
\end{proof}

We will also need  Large Deviation Theorem stated in the context of the rotation number.

\begin{prop}\label{p.ldtrot}
For $\mu^\mathbb{N}$-almost every $\bar \omega\in \Omega^{\mathbb{N}}$, the sequence $\frac{1}{n}\tf_{n, a, \bo}(\tx_0)$ converges to $\rho(a)$ uniformly in $a\in J$. Moreover, for every $\eps>0$ there exists a constant $\zeta_3>0$ such that \edited{for all sufficiently large $n\in \mathbb{N}$}
$$
\P\left(\ \left|\frac{1}{n} \tf_{n,a,\omega}(\tx_0)- \rho(a)\right| > \eps' \ \text{for some}\ \ a\in J \right)\le e^{-\zeta_3n}.
$$
\end{prop}

\begin{coro}\label{c.ldtrot}
For any $\eps'>0$ there exists $\zeta_4>0$ such that \edited{for all sufficiently large $n\in \mathbb{N}$}
$$
\P\left(\ \left|\tf_{m,a,\omega}(\tx_0)- m\rho(a)\right| > \eps'n \ \text{for some}\ \ a\in J \ \text{and}\ m\le n\right)\le e^{-\zeta_4n}.
$$
\end{coro}

\begin{proof}[Proof of Proposition~\ref{p.ldtrot}]
Take any fixed $n_0\in \N$. Note first that (upon replacing~$\eps'$ with a smaller value, e.g.~$3\eps'/4$) we can restrict ourselves to $n$ that are
multiples of~$n_0$. Indeed, taking $k=[\frac{n}{n_0}]$, we get
$$
\tf_{n,a,\bo}=\tf_{[kn_0,n],a,\bo} \circ \tf_{kn_0,a,\bo},
$$
and as the increment $\tf_{[kn_0,n],a,\bo}(\ty)-\ty$ is uniformly bounded, the same holds for the difference
$$
\left| \tf_{n,a,\bo}(\tx_0)-\tf_{kn_0,a,\bo}(\tx_0) \right|.
$$
For $n=kn_0$ we can split the length $n$ composition $\tf_{n,a,\bo}$ into  groups of length $n_0$:
$$
\tf_{n,a,\bo} = \tf_{[(k-1)n_0, kn_0],a,\bo} \circ \dots \tf_{[n_0, 2n_0],a,\bo} \circ \tf_{n_0,a,\bo}.
$$
 If we denote (compare with Section~\ref{s.johnson})
$$
\varphi_{a, n_0}(\bo,\ty):=\tf_{n_0,a,\bo}(\ty)-\ty,
$$
then we have
$$
\tf_{kn_0,a,\bo}(\tx_0)-\tx_0 = \sum_{j=0}^{k-1} \varphi_{a,n_0}(\sigma^{jn_0}\bo, \tf_{jn_0,a,\bo}(\tx_0)).
$$
Now for any $n_0$ and $\bo\in \Omega^{\mathbb{N}}$ we have $\osc_{\ty\in\R} (\tf_{n_0,a,\bo}(\ty)-\ty) \le 1$. Define
$$
\bar\varphi_{a,n_0}(\bo):= \max_{\ty\in\R} (\tf_{n_0,a,\bo}(\ty)-\ty).
$$
Notice that $\bar\varphi_{a, n_0}$ depends only on the first $n_0$ letters of the word~$\bo$.
We have
$$
\frac{1}{kn_0} \cdot \left| (\tf_{kn_0,a,\bo}(\tx_0)-\tx_0) - \sum_{j=0}^{k-1} \bar\varphi_{a,n_0}(\omega_{jn_0+1},\dots, \omega_{jn_0+n_0}) \right| \le \frac{1}{n_0}.
$$

In particular, passing to the limit $k\to \infty$ for an individual~$a$, we see that
\begin{equation}
\left|\rho(a)-\frac{1}{n_0}\E \bar\varphi_{a,n_0}  \right| <\frac{1}{n_0}.
\end{equation}
Now, take $n_0>\frac{10}{\eps'}$. Then, we have
\begin{multline*}
\frac{1}{kn_0} \cdot \left| \tf_{kn_0,a,\bo}(\tx_0) - kn_0 \cdot \rho(a) \right| \le \left|\rho(a) - \frac{1}{n_0} \E \bar\varphi_{a,n_0}  \right| +  \\
\frac{1}{kn_0} \cdot \left| (\tf_{kn_0,a,\bo}(\tx_0)-\tx_0) - \sum_{j=0}^{k-1} \bar\varphi_{a,n_0}(\omega_{jn_0+1},\dots, \omega_{jn_0+n_0}) \right| + \frac{|\tx_0|}{kn_0}+ \\
 \left| \frac{1}{k}\sum_{j=0}^{k-1} \frac{1}{n_0}\bar\varphi_{a,n_0}(\omega_{jn_0+1},\dots, \omega_{jn_0+n_0}) - \frac{1}{n_0}\E \bar\varphi_{a,n_0} \right|.
\end{multline*}
Each of the first three summands on the right hand side does not exceed $\frac{\eps'}{10}$. Hence, for any~$a\in J$ the event
$$
\left|\frac{1}{kn_0} \tf_{kn_0,a,\bo}(\tx_0)- \rho(a)\right| > \eps'
$$
is contained in the event
\begin{equation}\label{eq:rho-a}
  \left| \frac{1}{k}\sum_{j=0}^{k-1}\frac{1}{n_0} \bar\varphi_{a,n_0}(\omega_{jn_0+1},\dots, \omega_{jn_0+n_0}) - \frac{1}{n_0}\E \bar\varphi_{a,n_0} \right| > \frac{7\eps'}{10}.
\end{equation}

Now, for any fixed $a$ the event in the left hand side of~\eqref{eq:rho-a} can be estimated using the standard  Large Deviations Theorem from the theory of probability: we have a sum of bounded i.i.d. random variables.

Let us now extend these argument to the full interval $J$. Notice that for a fixed $n_0$ the displacements $\varphi_{a,n_0}(\bo,\ty)$ are continuous in $a$ uniformly in both $\ty$ and $\bo$. Therefore, $\bar\varphi_{a, n_0}(\bo)$ is also continuous in $a$ uniformly in $\bo\in \Omega^{\mathbb{N}}$.  Hence, any $a'\in J$ is contained in an open interval $J_{a'}$ such that
$$
\left| \bar\varphi_{a, n_0}(\bo) - \bar\varphi_{a',n_0}(\bo) \right| < \frac{\eps'}{10}
$$
for any $a\in J_{a'}$ and $\bo\in \Omega^{\mathbb{N}}$.
In particular, this implies  that
$$
|\E\bar\varphi_{a, n_0}-\E\bar\varphi_{a', n_0}|<\frac{\eps'}{10},
$$
and, moreover, for any
%
 $a\in J_{a'}$ the event~\eqref{eq:rho-a} is contained in the similar event for $a'$,
\begin{equation}\label{eq:a-prime}
 \left|\frac{1}{k} \sum_{j=0}^{k-1}\frac{1}{n_0} \bar\varphi_{a',n_0}(\omega_{jn_0+1},\dots, \omega_{jn_0+n_0}) - \frac{1}{n_0}\E \bar\varphi_{a',n_0} \right| > \frac{5\eps'}{10}=\frac{\eps'}{2}.
\end{equation}

As $J$ is compact, we can extract a finite cover $J_{a_i}$ of $J$; for each $a'=a_i$, the event~\eqref{eq:a-prime} has exponentially small probability: less than $e^{- \zeta_{(i)} k}$ for all sufficiently large~$k$. As there is a finite number of them, we get the desired estimate with any $\zeta_3<\frac{1}{n_0} \min_i \zeta_{(i)}$.

Finally,  uniform convergence $\frac{1}{n}\tf_{n, a, \bo}(\tx_0)\to \rho(a)$ for $\mu^{\mathbb{N}}$-a.e. $\bo$ directly follows from the Large Deviation estimate and Borel-Cantelli type arguments.
\end{proof}
\begin{remark}
Since uniform limit of continuous functions is continuous, Proposition \ref{p.ldtrot} implies continuity of the rotation number $\rho(a)$. In fact, it is known that the function $\rho(a)$ must be H\"older continuous, see \cite{L}, but we are not using this fact in our proof.
\end{remark}

\subsection{Uniform growth estimates}\label{ss:norms} Here we deduce parts \ref{i:m2} and \ref{i:m3} of Theorem ~\ref{t:main} from Proposition~\ref{p:classes}. \edited{Let us recall that ``jump intervals'' in terms of Proposition~\ref{p:classes} correspond to the exceptional intervals in Theorem \ref{t:main}.}

First let us show that the distortion control given by Lemma \ref{l:DC} together with Proposition~\ref{p:classes} allows us to use Lemma~\ref{l:LD-RD}
to estimate the derivatives at~$\tx_0$ at all parameter values $a\in J$:

\begin{prop}\label{p:derivatives-control}
There exists a constant $C_1$ such that \edited{for any $\varepsilon'>0$}  the following property holds for all sufficiently large~$n$.
Assume that $\bo$ is such that the conclusions of Lemma~\ref{l:LD-RD} and Proposition~\ref{p:classes} hold.
Then, for any $a\in J$:
\begin{itemize}
\item If $a\in J_i$, and $J_i$ is either ``small'' or ``opinion-changing'' interval in terms of Proposition~\ref{p:classes}, then
\begin{equation}\label{eq:derivative-u}
\forall\ m=1,\dots, n \quad \log \tf'_{m,a,\bar\omega}(\tx_0) \in U_{C_1\eps' n} (\lR(a)\cdot m).
\end{equation}
\item If $a\in J_i$, and $J_i$ is a ``jump'' interval in terms of Proposition~\ref{p:classes}, with the associated moment $m_0$, then
\begin{equation}\label{eq:derivative-upl}
\forall\ m=1,\dots, m_0' \quad \log \tf'_{m,a,\bar\omega}(\tx_0) \in U_{C_1\eps' n} (\lR(a)\cdot m),
\end{equation}
where $m_0':=m_0+\eps' n$, and 
\begin{equation}\label{eq:derivative-upr}
\forall\ m={m_0'}+1,\dots, n \quad \log \tf'_{[m_0', m],a,\bar\omega}(\tx_1) \in U_{C_1\eps' n} (\lR(a)\cdot (m-{m_0'})),
\end{equation}
for any $\tx_1\in X_{m_0, i}'$, where $m_0':=m_0+\eps'n$ and we denote $ X_{m_0', i}':=[\tx_{m_0',i-1}+1,\tx_{m_0',i}]$.
\end{itemize}
\end{prop}

\begin{proof}
In the first case, regardless of whether the interval $J_i$ is a ``small'' one or an ``opinion-changer'', we have an
upper bound for the sum of the corresponding lengths
\begin{equation}\label{eq:sum-no-jump}
\sum_{m=0}^{n-1} |X_{m,i}|=\sum_{|X_{m,i}|<\eps'} |X_{m,i}|+\sum_{|X_{m,i}|\ge \eps'} |X_{m,i}| \le n\cdot \eps'+ n\eps'\cdot 1= 2n\eps'.
\end{equation}
Lemma~\ref{l:DC} implies that for all $a\in J_i$ and all $m=1, \ldots, n$ we have
$$
|\log \tf'_{m,a,\bar\omega}(\tx_0)-\log \tf'_{m,b_i,\bar\omega}(\tx_0)| \le 2\kappa \eps'  n + C \cdot \frac{|J|}{N} n.
$$

Since $\lR(a)$ is a continuous function of the parameter $a\in J$ (\edited{see Remark \ref{r.cont}}), for a given $\eps'>0$ and sufficiently large $n$ we have:
 $$
 |\lR(a)-\lR(b_i)|\le \eps', \ \ \text{and}\ \ \ \frac{|J|}{N}<\eps'.
 $$
Together with the estimate~\eqref{eq:LD-RD-2} this gives
%
\begin{multline}\label{eq:log-f-prim}
|\log \tf'_{m,a,\bar\omega}(\tx_0) - \lR(a) m | \le |\log \tf'_{m,a,\bar\omega}(\tx_0) - \log \tf'_{m,b_i,\bar\omega}(\tx_0) |+
\\
 +  |\log \tf'_{m,b_i,\bar\omega}(\tx_0) - \lR(b_i) m | + |\lR(b_i) m - \lR(a) m |\le \\ 2\kappa \eps'  n + C \eps' n + \eps' n + \eps' m   \le  (2\kappa +C+2) \eps' n.
\end{multline}
Therefore~\eqref{eq:derivative-u} holds once $C_1>2\kappa+C+2$.

Suppose now that $J_i$ is a ``jump'' interval. Checking~\eqref{eq:derivative-upl} goes exactly in the same way as in~\eqref{eq:sum-no-jump}:
$$
\sum_{m=0}^{m_0'} |X_{m,i}|=\sum_{m=0}^{m_0-1} |X_{m,i}|+\sum_{m=m_0}^{m_0'-1} |X_{m,i}| \le n\cdot \eps'+ n\eps'\cdot 2= 3n\eps'.
$$
Hence, in the same way as in~\eqref{eq:log-f-prim}, we have for any $m\le m_0'$
$$
|\log \tf'_{m,a,\bar\omega}(\tx_0) - \lR(a) m | \le 3\kappa \eps'  n + C \eps' n + \eps' n + \eps' m   \le  (3\kappa +C+2) \eps' n,
$$
and we have the desired~\eqref{eq:derivative-upl} once $C_1>3\kappa+C+2$.

Finally, the intervals $X_{i,m}'$ for $m\ge m_0'$ also satisfy the assumptions of Lemma~\ref{l:DC}. One has
$$
\sum_{m=m_0'}^{n} |X_{m,i}| \le \eps' n,
$$
and thus (again, together with~\eqref{eq:LD-RD-2}) we get
$$
|\log \tf'_{[m_0',m],a,\bar\omega}(\tx_1) - \lR(a) (m-m_0') | \le \kappa \eps'  n + C \eps' n + \eps' n + \eps' m   \le  (\kappa +C+2) \eps' n.
$$
This proves~\eqref{eq:derivative-upr} for any $C_1>\kappa +C+2$, and thus concludes the proof of Proposition \ref{p:derivatives-control}.
\end{proof}

Proposition \ref{p:derivatives-control} implies the parts \ref{i:m2} and \ref{i:m3} of Theorem~\ref{t:main}. Indeed,
 for any $A\in SL(2,\R)$ and for any vector $v\neq 0$ one has
\begin{equation}\label{eq:der-norms}
f_A'(x_v)= \frac{|v|^2}{|Av|^2},
\end{equation}
where $x_v\in \Sc$ is the direction corresponding to the vector~$v$. \edited{Therefore,} for any point $x$ on the circle
one has $\log \|A\| \ge -\frac{1}{2} \log f_A'(x)$ (as the right hand side of~\eqref{eq:der-norms} is not less than~$\frac{1}{\|A\|^2}$).
In particular, for any $m, a,\bo$ we have
\begin{equation}\label{eq:norms-lower}
\log \|T_{m,a,\bo}\| \ge -\frac{1}{2} \log f'_{m,a,\bo}(\bar x).
\end{equation}

\edited{If $a$ belongs to ``small'' or ``opinion-changing'' interval $J_i$, } by joining this estimate with~\eqref{eq:derivative-u}, we obtain a lower bound for the norm
$$
\log \|T_{m,a,\bo}\| \ge -\frac{1}{2} \cdot (\lR (a) m + C_1 n \eps') = \lambda_F(a) m - \frac{C_1}{2} \eps' n.
$$
Hence, to obtain the lower bound in the ``Uniformity'' part, it suffices to take
$$
\eps'<\frac{2\eps}{C_1}.
$$

On the other hand, Proposition~\ref{l:upper-finite} states that the upper bound
$$
\log \|T_{m,a,\bo}\| < \lambda_F(a) m + n\eps
$$
holds with the probability $1-\exp(c_3 n)$. We thus obtain the desired
$$
\log \|T_{m,a,\bo}\| \in U_{n \eps}(\lambda_F(a)m)
$$
for all $a\in J_i$, provided that the interval $J_i$ was ``small'' or ``opinion-changing''. Now, assume that $a\in J_i$,
and the interval $J_i$ is a ``jump'' interval. \edited{Set $\bar m:=m_0+\varepsilon'$.} Then again, joining~\eqref{eq:norms-lower} with~\eqref{eq:derivative-upl}--\eqref{eq:derivative-upr}, we obtain
$$
\forall m=1,\dots, \bar{m} \quad \log \|T_{m,a,\omega}\| \ge \lambda_F(a) m - \frac{C_1}{2} \eps' n> \lambda_F(a) m - \eps n
$$
and
$$
\forall m=\bar{m}+1,\dots, n \quad \log \|T_{[\bar{m};m],a,\omega}\| \ge \lambda_F(a) (m-\bar m) - \frac{C_1}{2} \eps' n> \lambda_F(a) (m-\bar m) -\eps n,
$$
where the last inequalities come from the choice of~$\eps'$.

Again, Proposition~\ref{l:upper-finite} gives the upper bounds
$$
\forall m=1,\dots, \bar{m} \quad \log \|T_{m,a,\bo}\| < \lambda_F(a) m + n\eps
$$
and
$$
\forall m=\bar{m}+1,\dots, n \quad \log \|T_{[\bar{m};m],a,\bo}\| < \lambda_F(a) (m-\bar m) + n\eps.
$$
This implies the desired ``Uniformity'' estimates
$$
\forall m=1,\dots, \bar{m} \quad \log \|T_{m,a,\bo}\| \in U_{n\eps}( \lambda_F(a) m)
$$
$$
\forall m=\bar{m}+1,\dots, n \quad \log \|T_{[\bar{m};m],a,\bo}\| \in U_{n\eps}( \lambda_F(a) (m-\bar m) ),
$$
thus concluding the proof of parts \ref{i:m2} and \ref{i:m3}  of Theorem~\ref{t:main}.

\subsection{Cancellation lemmas}\label{ss:cancellation}

The arguments in this paragraph use in essential way the properties of projective dynamics. This is not an artifact of the proof.
In fact, we expect the behavior of generic parameter-dependent random
dynamical system on the circle to be different.

For any $A\in SL(2, \mathbb{R})$ denote by $f_A$ the corresponding projective map of \edited{${\Sc}$}. Also, for $A\notin SO(2,\R)$ let $x^-(A)\in \Sc$ be the point where $f_A$ has the largest derivative, and $x^+(A)\in \Sc$ be the image under $f_A$ of the point where $f_A$ has the smallest derivative. Equivalently, $x^+(A)$ is the direction of the large axis of the ellipse, obtained by applying $A$ to the unit circle, and $x^-(A)=x^+(A^{-1})$.

Let $\alpha$ and $\beta$ be the angles of $x^-(A)$ and $x^+(A)$ respectively. Then, \edited{using singular value decomposition, we} see that
$$
A=\pm R_{\beta} \left(
\begin{matrix}
\|A\| & 0 \\
0 & \|A\|^{-1}
\end{matrix}
\right) R_{\alpha+\pi/2}^{-1}.
$$
In particular, one has the following useful
\begin{lemma}[Cancellation for matrices]\label{l:cancel}
Let $A,B\in SL(2,\R)\setminus SO(2,\R)$ be two matrices such that $x^+(A)=x^-(B)$. Then
$$
\|BA\|=\max (\frac{\|B\|}{\|A\|}, \frac{\|A\|}{\|B\|}).
$$
\end{lemma}
\edited{The proof of Lemma \ref{l:cancel} is straightforward, and is left to the reader. }

We will also use the following lemma, saying, roughly speaking, that a direction that is expanded is sent close to the maximally expanded direction. \edited{Here we will measure a distance between two directions by a smallest angle between those directions, i.e. interpret the projective space as $\mathbb{R}/2\pi \mathbb{Z}$.}
\begin{lemma}\label{l:x-C-image}
Let $A\in SL(2,\R)\setminus SO(2,\R)$, $x\in \Sc$ be a point on the circle, and $v_x$ be some vector in the corresponding direction. Then:
\begin{itemize}
\item $\dist (f_A(x),x^+(A))\le \frac{\pi}{2} \cdot \frac{|Av_x|/|v_x|}{\|A\|}$,
\item $\dist (x,x^-(A))\le \frac{\pi}{2} \cdot \frac{|v_x|/|Av_x|}{\|A\|}$,
\item If we have $f_A'(x)<\frac{1}{C}$, then $\|A\|\ge \sqrt{C}$ and $x^+(A)$ belongs to $\frac{\pi }{2C}$-neighborhood of $f_A(x)$.
\end{itemize}
\end{lemma}

\begin{figure}[!h!]
\begin{center}
\includegraphics{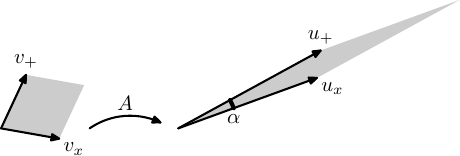}
\end{center}
\caption{Vectors $v_x$, $v_+$, their images $u_x$, $u_+$, and the angle between these images.}\label{f:A-derivative}
\end{figure}

\begin{proof}
Take $v_+$ to be the unit vector in the most expanded direction, that is, $|Av_+| = \|A\|$.
Let $u_+:=Av_+$, and let $\alpha$ be the angle between $Av_x$ and $u_+$ (see Fig.~\ref{f:A-derivative}).
In particular, $u_+$ is a vector in the direction given by $x^+(A)$,
and we can assume (changing the sign of one of the vectors if necessary) that $\alpha$ is the distance
between~$x^+(A)$ and~$f_A(x)$.

Now, the area of the parallelogram defined by unit vectors $v_x$ and $v_+$ is at most $|v_x| \cdot |v_+|$, hence the same holds for the area of the parallelogram defined by their images. We thus have
$$
|v_x| \ge |Av_x| \cdot |u_+| \cdot \sin \alpha = \|A\| \cdot |A v_x| \cdot \sin \alpha;
$$
hence $\sin \alpha <  \frac{|v_x|/|Av_x|}{\|A\|}$, and we get the desired $\alpha < \frac{\pi}{2} \cdot \frac{|v_x|/|Av_x|}{\|A\|}$.

The second part is obtained from the first one by replacing $A$ with $A^{-1}$.

For the last part, recall that $f_A'(x)=\frac{|v_x|^2}{|Av_x|^2}$, where $v_x$ is any vector in the direction given by $x\in\Sc$. Hence,
$$
\|A\|\ge \frac{|Av_x|}{|v_x|} = \frac{1}{\sqrt{f'_A(x)}} > \sqrt{C};
$$
joining this with the first part, we get the desired estimate:
$$
\frac{\pi}{2} \frac{|v_x|/|Av_x|}{\|A\|} < \frac{\pi}{2} \frac{1/\sqrt{C}}{\sqrt{C}} = \frac{\pi}{2C}.
$$
\end{proof}

Let us now prove the ``Cancellation'' part~\ref{i:m4} of the conclusions of Theorem~\ref{t:main}; to do that, we have to handle the ``jump'' intervals. 
Namely, assume that the conclusions of Lemma~\ref{l:LD-RD} hold, and $J_i$ is a ``jump'' interval in terms of Proposition~\ref{p:classes}. Set $\bar m:=m_0+\eps' n$, where $m_0$ is given by the definition of ``jump interval'' in Proposition~\ref{p:classes}.
Notice (we will use it later) that 
\begin{equation}\label{e.assumption}
|X_{\bar m+1,i}|\ge 1+ \delta \frac{|J|}{N},
\end{equation}  
where $\delta>0$ is given by the standing assumption $(A4)$.
Indeed, the inequality $\tx_{\bar m,i}\ge \tx_{\bar m,i-1}+1$ implies that
\begin{multline*}
\tx_{\bar m+1,i}-\tx_{\bar m+1,i-1}=\tf_{\omega_{\bar m+1},b_i}(\tx_{\bar m,i})- \tf_{\omega_{\bar m+1},b_{i-1}}(\tx_{\bar m,i-1})
\\ =(\tf_{\omega_{\bar m+1},b_i}(\tx_{\bar m,i})-\tf_{\omega_{\bar m+1},b_{i-1}}(\tx_{\bar m,i}))
\\ + (\tf_{\omega_{\bar m+1},b_{i-1}}(\tx_{\bar m,i})-\tf_{\omega_{\bar m+1},b_{i-1}}(\tx_{\bar m,i-1}));
\end{multline*}
the former summand is bounded from below by $\delta (b_i-b_{i-1})$, and the latter is at least~$1$.

We start by handling the case when the jump moment happens too close to the first or the last iteration.
\begin{lemma}\label{l:close}
Let $\eps', \eps''>0$, and assume that the conclusions of  Proposition~\ref{p:classes} 
hold, and also that the conclusions of the part~\ref{i:m3} of Theorem \ref{t:main} hold with the value $\eps'$ instead of $\eps$. Suppose $J_i$  is a ``jump'' interval with associated index $m_0$, and set $\bm:=m_0+\eps' n$. Assume that 
 $\bm\le \eps'' n $ or $\bm \ge (1-\eps'') n$.
Then the conclusions of the ``Cancellation'' part~\ref{i:m4} of Theorem \ref{t:main} are satisfied for an arbitrary $a\in J_i$, provided that one has
$$
2 \eps' +2\lambda_F (a) \eps''<\eps.
$$
\end{lemma}
\begin{remark}
\edited{Notice that using $\varepsilon'$ instead of $\eps$ in the part~\ref{i:m3} of Theorem \ref{t:main} does not lead to any problems. Indeed, largeness of $n\in \mathbb{N}$ needed for which the conclusions hold does depend on the value of $\eps$, but those conclusions hold for all  $n$ larger than some threshold,  so by increasing that threshold we can assume that the conclusions of  Proposition~\ref{p:classes}
and  the conclusions of the part~\ref{i:m3} of Theorem \ref{t:main} hold with the value $\eps'$ instead of $\eps$ simultaneously for all sufficiently large $n$.}
\end{remark}
\begin{proof}
Consider first the case $\bm\le \eps''n$.
For $m\le \bm$, due to the conclusions of part~\ref{i:m3}
we have
$$
\log \|T_{m,a,\bo}\|\le n\eps' + \lambda_F (a) \bm, \quad \psi_{\bm}(m)=m\le \bm,
$$
and hence
$$
\left| \log \|T_{m,a,\bo}\| - \lambda_F(a) \psi_{\bm}(m) \right| \le n\eps' + 2 \lambda_F (a) \bm \le (\eps'+ 2 \lambda_F(a) \eps'') n < \eps n
$$
thus guaranteeing the desired~\eqref{e.IV}. On the other hand, once $m\ge \bm$, we have
$$
\log \|T_{\bm,a_i,\bo} \| \le n\eps' + \lambda_F (a) \eps'' n, \quad \log \|T_{[\bm, m],a,\bo}\| \in U_{n\eps'} (\lambda_F(a) \cdot (m-\bm)),
$$
hence
\begin{equation}\label{eq:left}
\log \|T_{m,a,\bo}\| \in U_{2 n\eps'+\lambda_F (a) \eps'' n} (\lambda_F(a) \cdot (m-\bm)).
\end{equation}
Finally, the functions $(m-\bm)$ and $\psi_{\bm}(m)$ differ by at most $\bm$, and we get from~\eqref{eq:left} the desired
\begin{multline*}
\left| \log \|T_{m,a,\bo}\| -\lambda_F(a) \psi_{\bm}(m) \right| \le 2 n\eps'+\lambda_F (a) \eps'' n + \lambda_F (a) \eps'' n
\\ = (2\eps'+2\eps'' \lambda_F(a)) n < \eps n.
\end{multline*}

The case $\bm\ge (1-\eps'') n$ is handled in the same way: for $m\le \bm$, the conclusions of part~\ref{i:m4} coincide with the conclusions of part~\ref{i:m3}. At the same time, if $m\ge \bm$, one has
$$
T_{m,a,\bo} = T_{[\bm, m],a,\bo} T_{\bm,a,\bo},
$$
and hence $\log \|T_{m,a,\bo}\|$ is $2n\eps'+ \lambda_F(a) n\eps''$-close to $\lambda_F(a) \bm$. And the latter is $\lambda_F(a) n\eps''$-close to $\lambda_F(a)\psi_{\bm}(m)$, finally implying the desired
$$
\left| \log \|T_{m,a,\bo}\| - \lambda_F(a)\psi_{\bm}(m) \right| \le (2\eps'+ 2\lambda_F(a) \eps'') n.
$$
\end{proof}

Let us now consider the case when the jump moment is sufficiently away from the endpoints of the interval of iterations, $\eps''n<\bm<(1-\eps'') n$.
First, we find the corresponding value of the parameter $a\in J_i$. Denote
$$
\lambda_{\min}:=\min_{a\in J} \lambda_F(a), \quad \lambda_{\max}:=\max_{a\in J} \lambda_F(a).
$$
\begin{lemma}\label{l:m-half}
Let $\eps', \eps''>0$ satisfy
\begin{equation}\label{eq:eps-pp}
\frac{\lambda_{\min}}{2C_1} \eps'' > \eps',
\end{equation}
where $C_1>1$ is given by Proposition \ref{p:derivatives-control}.
For all sufficiently large $n$, the following statement holds.

Assume that the conclusions of Lemma~\ref{l:LD-RD} and of Proposition~\ref{p:classes} hold,
 $J_i$  is a ``jump'' interval with associated index $m_0$, and set $\bm:=m_0+\eps' n$.
 Assume also that the conclusions of the part~\ref{i:m3} hold with the value~$\eps'$ instead of~$\eps$. Then there exists $a\in J_i$ such that
$$
x^+(T_{\bar m,a,\bo})=x^-(T_{[\bar m, \bm'],a,\bo}),
$$
where $\bm':= \min (2\bm, n)$.
\end{lemma}
\begin{proof}
Note that the uniformity estimates imply that the products $T_{\bar m,a,\bo}$ and $T_{[\bar m;\bm'],a,\bo}$
are of norm bounded away from~1 for all $a\in J_i$. Indeed, the conclusions of the part~\ref{i:m3} imply that
$$
\log \|T_{\bm,a,\bo} \| > \bm \lambda_F(a) - n \eps' > n( \eps'' \lambda_{\min} -\eps') >0,
$$
$$
\log \|T_{[\bm,\bm'], a,\bo} \| > (\bm' -\bm) \lambda_F(a) - n \eps' > n( \eps'' \lambda_{\min} -\eps') >0,
$$
where the last inequalities are due to~\eqref{eq:eps-pp}.

Hence the directions
$x^+(T_{\bar m,a,\bo})$ and $x^-(T_{[\bar m;\bm'],a,\bo})$ depend continuously on~$a\in J_i$. To shorten the notations, we denote
$$
x^+(a):=x^+(T_{\bar m,a,\bo}), \quad x^-(a):=x^-(T_{[\bar m, \bm'],a,\bo}).
$$

 Lemma~\ref{l:x-C-image} implies that $x^+(a)$ stays $\frac{\pi}{2} f_{\bar m,a,\bo}'(x_0)$-close to the image~$f_{\bar m,a,\bo}(x_0)$
as $a$ varies in $J_i$. \edited{Also}, for any $a\in J_i$ we have
$$
\frac{\pi}{2}f_{\bar m,a,\bo}'(x_0)< \frac{\pi}{2}\exp(-2\lambda_F(a) \bar m + C_1n \eps') <\exp\left(-{\lambda_F(a)} \bar m \right) < \frac{\delta |J|}{2N},
$$
where the second inequality is due to the assumptions $\bar m \ge n\eps''$ and $\frac{\lambda_{\min}}{2C_1}\eps''>\eps'$, and the last one
is due to the subexponential growth of~$N=[\exp(\sqrt[4]{n})]$.

At the same time, due to (\ref{e.assumption}),
 we have $|X_{\bm, i}'|\ge  \frac{\delta |J|}{N}$, \edited{where $X_{\bm, i}'=[\tx_{\bm, i-1}+1, \tx_{\bm,i}]$}.
Hence, as $a$ varies over $J$, the point $x^+(a)$ passes through the midpoint
$$
r:=\pi\left(\frac{(\tx_{\bar m, i-1}+1)+\tx_{\bar m, i}}{2}\right)
$$
of the interval $\pi(X_{\bar m,i}')=\pi([\tx_{\bar m, i-1}+1,\tx_{\bar m, i}])$
at least twice, making the full turn in between; see Figure~\ref{f:x-plus-minus}.

\begin{figure}[!h!]
\begin{center}
\includegraphics{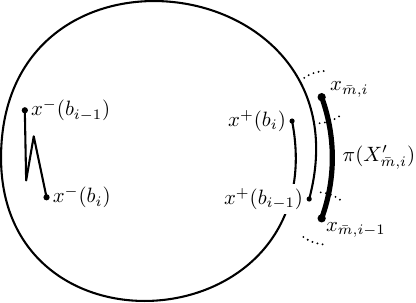}
\end{center}
\caption{While the parameter $a$ varies over a jump interval $J_i$, the $x^+(a):=x^{+}(T_{\bar m,\omega,a})$ makes more than a full turn,
staying in a neighborhood of the corresponding image $f_{\bar m,a,\omega}(\bar x)$. At the same time, the point $x^-(a):=x^-(T_{[\bm,\bm'],a,\omega})$
never enters the interval $\edited{\pi(X'_{\bar m, i})}$ (the arc shown in bold).}\label{f:x-plus-minus}
\end{figure}

We know from the
distortion control estimates \edited{given by} Proposition~\ref{p:derivatives-control}
that the derivatives of $f_{[\bar m,\bm'],a,\bo}$ on $X_{\bar m,i}'$ do not exceed
$$
\exp(\lR(\bm'-\bar m)+C_1\eps' n)= \exp(-2 n\eps''\lambda_F(a) +C_1\eps' n) < \exp( (-2\lambda_{\min} \eps'' + C_1\eps') n) < 1,
$$
again using~\eqref{eq:eps-pp} for the last inequality.

Hence the point $x^-(a)$ never crosses $r$ for $a\in J_i$. Thus, we can choose the lifts $\tilde x^+(a)$ and $\tilde x^-(a)$ on the real line of
$x^+(a)$, $x^-(a)$ respectively such that the difference $\tilde x^+(a)-\tilde x^-(a)$ changes sign while $a$ varies in $J_i$. Hence, there
exists a point $a\in J_i$ for which the directions $x^+(a)$ and $x^-(a)$ coincide.
\end{proof}

We are now ready to conclude the proof of the ``Cancellation'' part~\ref{i:m4}. Take $\eps', \eps''>0$ such that~\eqref{eq:eps-pp}
holds, as well as
$$
\eps'<\frac{\eps}{4}, \quad 2(\eps'+ \lambda_{\max} \eps'') < \eps.
$$

Assume that the conclusions of Lemma~\ref{l:LD-RD} hold and of Proposition~\ref{p:classes} hold,
that $J_i$ in its terms is a ``jump'' interval, with $\bm:=m_0+\eps' n$ being the corresponding jump moment.
Assume also that the conclusions of the part~\ref{i:m3} hold with the value~$\eps'$ instead of~$\eps$.

Let us show that then the part~\ref{i:m4} of conclusions of Theorem~\ref{t:main} are satisfied.
Indeed, if $\bm\le \eps'' n$ or $\bm\ge (1-\eps'') n$, this directly follows from Lemma~\ref{l:close}. Otherwise
we can apply Lemma~\ref{l:m-half}; take $a_i$ to be the value of the parameter $a$ given by Lemma~\ref{l:m-half},
and let us check that \eqref{e.IV} holds for all $m= 1,\dots, n$.

Note that for any $m\in [1,\bar m]$ the estimates of the part~\ref{i:m3} imply
\begin{equation}\label{eq:m1}
\log \|T_{m,a_i,\bo}\| \in U_{\eps' n}(\lambda_F(a)m) = U_{\eps' n}(\psi(m)).
\end{equation}

We have now to handle the case $m\in [\bm,n]$. The next steps depend on whether~$\bm$ is greater or less than $\frac{n}{2}$.


Consider first the case $\bm\le \frac{n}{2}$ (in this case $\bm'=2\bm$).
Then, applying Lemma~\ref{l:cancel} and the uniformity estimates on the intervals $[1,\bar m]$ and $[\bar m, 2\bar m]$, we get
\begin{multline}\label{eq:almost-zero}
\log \|T_{2\bar m,a_i,\bo}\| = \left| \log \|T_{\bar m,a_i,\bo}\| -  \log \|T_{[\bar m,2\bar m],a_i,\bo}\| \right|
\\
\le \left| \log \|T_{\bar m,a_i,\bo}\| - \lambda_F(a_i) \bar m \right| +\left| \log \|T_{[\bar m,2\bar m],a_i,\bo}\| - \lambda_F(a_i) \bar m \right|
\le 2n\eps'.
\end{multline}
For any $m\in [\bar m,2\bar m]$ we can represent
$$
T_{m,a_i,\bo} = T_{[m, 2\bar m],a_i,\bo}^{-1} T_{2\bar m,a_i,\bo}.
$$
The log-norm of the latter factor does not exceed $2\eps'n$ by~\eqref{eq:almost-zero}, while the log-norm of the former factor is $2n\eps'$-close to $\lambda_F(a_i) (2\bar m - m)=\psi(m)$ due to the conclusion of the part~\ref{i:m3} and Proposition \ref{l:upper-finite}. Indeed,
$$
\|T_{[m, 2\bar m],a_i,\bo}^{-1}\|=\|T_{[m, 2\bar m],a_i,\bo}\|,
$$
and
$$
T_{[m, 2\bar m],a_i,\bo}=T_{[\bm, 2\bar m],a_i,\bo}T_{[\bm,  m],a_i,\bo}^{-1}.
$$
Due to the part~\ref{i:m3} of Theorem \ref{t:main}
$$
\log\|T_{[\bm, 2\bar m],a_i,\bo}\|\in U_{\eps'n}(\lambda_F(a_i)\bm), \ \ \ \log\|T_{[\bm,  m],a_i,\bo}\|\in U_{\eps'n}(\lambda_F(a_i)(m-\bm)),
$$
and hence
$$
\log\|T_{[m, 2\bar m],a_i,\bo}\|\ge \log\|T_{[\bm, 2\bar m],a_i,\bo}\|-\log\|T_{[\bm,  m],a_i,\bo}\|\ge \lambda_F(a_i)(2\bm-m)-2\eps'n.
$$
On the other hand, by Proposition \ref{l:upper-finite}
$$
\log\|T_{[m, 2\bar m],a_i,\bo}\|\le \lambda_F(a_i)(2\bm-m)+\eps'n.
$$
Therefore, $$\log\|T_{[m, 2\bar m],a_i,\bo}\|\in U_{2n\eps}(\lambda_F(a_i)(2\bm-m)).$$
Hence
\begin{equation}\label{eq:m2}
\log \|T_{m,a_i,\bo}\| \in U_{4n\eps'}(\lambda_F(a_i)\psi(m)).
\end{equation}
Finally, for any $m\in [2\bar m,n]$ we have
$$
T_{m,a_i,\bo} = T_{[2\bar m, m],a_i,\bo} T_{2\bar m,a_i,\bo}.
$$
Again, the log-norm of the latter factor does not exceed $2\eps'n$ by~\eqref{eq:almost-zero}, while the log-norm of the former factor is $2n\eps'$-close to
$\lambda_F(a_i)(m-2\bar m)=\psi(m)$, due to the conclusion of the part~\ref{i:m3} and Proposition \ref{l:upper-finite}.
This implies the desired
\begin{equation}\label{eq:m3}
\log \|T_{m,a_i,\bo}\| \in U_{4n\eps'}(\lambda_F(a_i)\psi(m)).
\end{equation}

Together~\eqref{eq:m1}, \eqref{eq:m2} and~\eqref{eq:m3} cover all possible $m\le n$, thus implying
$$
\forall m=1,\dots, n \quad  \log \|T_{m,a_i,\bo}\| \in U_{4n\eps'}(\lambda_F(a_i)\psi(m)).
$$
As we have $\eps'<\frac{\eps}{4}$, we obtain the desired estimate.

Finally, consider the case $\bm>\frac{n}{2}$ (in this case $\bm'=n$). Then in the same way as in~\eqref{eq:almost-zero} the estimates of the part~\ref{i:m3} imply
$$
\log \|T_{\bm,a_i,\bo}\| \in U_{n\eps'}(\lambda_F(a_i) \bm), \quad  \log \|T_{[\bm,n],a_i,\bo}\| \in U_{n\eps'}(\lambda_F(a_i) (n-\bm)),
$$
and thus finally
\begin{equation}\label{eq:bm-p}
\log \|T_{n,a_i,\bo}\| \in U_{2n\eps'} (\lambda_F(a_i) (2\bm - n))= U_{2n\eps'} (\lambda_F(a_i) \psi_{\bm}(n)).
\end{equation}

Now, for any $m\in [\bm,n]$ we have two representations for $T_{m,a_i,\bo}$:
\begin{equation}\label{eq:steps}
T_{m,a_i,\bo}=T_{[\bm,m],a_i,\bo}T_{\bm,a_i,\bo} = T_{[m,n],a_i,\bo}^{-1} T_{n,a_i,\bo}.
\end{equation}
By Proposition \ref{l:upper-finite} we have
$$
\log \| T_{[\bm,m],a_i,\bo}\| \le \lambda_F(a_i) (m-\bm) + n\eps', \quad \log \| T_{[m,n],a_i,\bo}\| \le \lambda_F(a_i) (n-m) + n\eps',
$$
so from~\eqref{eq:bm-p}  and~\eqref{eq:steps} we get both the upper estimate
\begin{multline*}
\log \|T_{m,a_i,\bo}\| \le \log \|T_{n,a_i,\bo}\|  + \log \|T_{[m,n],a_i,\bo}\| \le
\\
\le ( \lambda_F(a_i) (2\bm-n)+ 2n\eps')+ (\lambda_F(a_i) (n-m) +n\eps') = \lambda_F(a_i) \psi_{\bm}(m) + 3 n \eps'
\end{multline*}
and the lower one
\begin{multline*}
\log \|T_{m,a_i,\bo}\| \ge \log \|T_{\bm,a_i,\bo}\|  - \log \|T_{[\bm,m],a_i,\bo}\| \ge
\\
\ge (\bm \lambda_F(a_i) - 2n\eps')  - (\lambda_F(a_i) (m-\bm) + n\eps') = \lambda_F(a_i) \psi_{\bm}(m) - 3 n \eps'.
\end{multline*}

Thus, in this case we also get the desired
$$
\log \|T_{m,a_i,\bo}\| \in U_{3n\eps'} (\psi_{\bm}(m) \lambda_F(a_i)),
$$
concluding the proof of the ``Cancellation'' part~\ref{i:m4} of Theorem \ref{t:main}.

\subsection{\edited{Contraction on average:} quantitative statements}\label{ss.5.6}
We will need a quantitative way to control the
exponential contraction of the corresponding random dynamics.
The Furstenberg Theorem implies that for the projective dynamics
on the circle the points almost surely approach each other exponentially
fast. Moreover, for a non-projective $C^1$-smooth dynamics (under mild assumptions)
such a statement also holds due to the Baxendale theorem~\cite{Bax}, that
implies negativity of the Lyapunov exponent. And even (quite surprisingly!) it
was recently shown by D. Malicet~\cite{M} for the case of homeomorphisms,
with no regularity assumptions at all.

However, here we will need a qualitative
estimate that can be used for the dynamics involving a parameter, so we cannot
make a pure reference to one of these papers.
The main result of this section is the following proposition (that was also simultaneously
and independently proven by Czudek, Szarek, and Zdunik \cite{Cz}):

\begin{prop}\label{p:phi}
There are constants $s\in (0,1]$ and $K\in \N$ such that for any $a\in J$
the function
\begin{equation}\label{eq:phi-def}
\varphi(x,y):=(\dist_{\Sc}(x,y))^s
\end{equation}
satisfies
\begin{equation}\label{eq:phi-contraction}
\E \varphi(f_{K, a,\omega}(x),f_{K, a,\omega}(y)) :=\int \varphi(f_{K, a,\omega}(x),f_{K, a,\omega}(y)) d\mu^{\mathbb{N}}(\bar \omega) \le \frac{1}{2} \varphi(x,y).
\end{equation}
\end{prop}

\begin{proof}
Notice first that for any $x\in \Sc$ we have
$$
\E \frac{1}{m} \log f'_{m,a,\bo}(x) \to \lR(a)<0 \quad \text{as } m\to\infty.
$$
Moreover, the convergence here is uniform in $x\in\Sc$ (this follows from the uniqueness of the stationary measure~$\nu_a$
on the circle, see \cite[Theorem 4.1]{BL}, in the same way as unique ergodicity implies uniform convergence of time averages, compare with the proof of \cite[Theorem 4.3.1]{HK}).
Hence, taking $K_0$ to be sufficiently large, we can find $\delta_1>0$ such that
\begin{equation}\label{eq:log-der-delta}
 \quad \forall x\in \Sc \quad  \E \log f'_{K_0,a,\bo}(x)<-\delta_1.
\end{equation}
Compactness arguments show that $K_0$ and $\delta_1>0$ in (\ref{eq:log-der-delta}) can be chosen uniformly also in $a\in J$.

For any $d>0$, as $s\to 0$, Taylor's formula gives
$$
d^s=\exp(s\log d) =1+ s \log d +s^2\frac{(\log d)^2d^\xi}{2!},
$$
for some $\xi\in (0,s)$.
Therefore we have
$$
(f'_{K_0,a,\bo}(x))^s=1+s(\log f'_{K_0,a,\bo}(x))+s^2\frac{(\log f'_{K_0,a,\bo}(x))^2(f'_{K_0,a,\bo}(x))^\xi}{2!},
$$
and, since $|f'_{K_0,a,\bo}(x)|$ is uniformly bounded,
$$
\E (f'_{K_0,a,\bo}(x))^s = 1 + s \E \log f'_{K_0,a,\bo}(x) + O(s^2),
$$
where $O(s^2)$ is uniform both in $x\in\Sc$ and $a\in J$. Hence~\eqref{eq:log-der-delta}
implies that for a sufficiently small $s>0$ there exists $\delta_2>0$ such that
\begin{equation}\label{eq:der-power}
\forall a\in J \ \ \forall x\in \Sc \quad \E (f'_{K_0,a,\bo}(x))^s < 1-2\delta_2.
\end{equation}

Next, the Mean Value Theorem implies that for the function $\varphi(x,y)$ defined by~\eqref{eq:phi-def}
for any sufficiently close $x,y\in\Sc$
and any $\bo$ we have
$$
\varphi(f_{K_0,a,\bo}(x),f_{K_0,a,\bo}(y))= (f'_{K_0,a,\bo}(z))^s \varphi(x,y)
$$
for some $z$ on the shortest arc connecting $x$ and $y$ (closeness here is needed to
ensure that the image of this arc is the shortest arc connecting $f(x)$ and $f(y)$).

As the function $(f'_{K_0,a,\bo}(z))^s$ is continuous in $z$ uniformly in $a,z,\bo$, there
exists $r>0$ such that if $\dist_{\Sc}(x,y)<r$, then
$$
\left|(f'_{K_0,a,\bo}(x))^s-(f'_{K_0,a,\bo}(z))^s\right|<\delta_2.
$$
Hence, for any $x,y$ with $\dist_{\Sc}(x,y)<r$ one has
$$
\E \varphi(f_{K_0,a,\bo}(x),f_{K_0,a,\bo}(y)) \le (1-\delta_2) \varphi(x,y).
$$
Let us fix $K_0$, $r$ and $\delta_2$ as above.

Next, let us handle case of two initial points being far away from each other.
The contraction of orbits for random dynamical systems on the circle is well-known: after many iterations
the images of two initial points will be most probably very close to each other.
We will need its version that is uniform in parameter $a$ and in the initial points $x,y$.
\begin{lemma}\label{l:r-contr}
For any $\eps_1,\eps_2>0$ there exists \edited{$K_0$} such that for any $a\in J$ and any $x,y\in \Sc$
we have
$$
\P \left( \dist( f_{K_0, a, \bo} (x), f_{K_0, a, \bo} (y)) <\eps_1 \right) > 1- \eps_2.
$$
\end{lemma}

Let us show that it suffices to conclude the proof of Proposition~\ref{p:phi}. Note first that it implies the following
\begin{coro}
There exists $K_1>0$ such that for any $x,y\in \Sc$ with $\dist_{\Sc}(x,y)\ge r$ and any $a\in J$ one has
\begin{equation}\label{eq:phi-s}
\E \varphi(f_{K_1,a,\bo}(x),f_{K_1,a,\bo}(y)) \le \frac{1}{2} \varphi(x,y).
\end{equation}
\end{coro}
\begin{proof}
Indeed, take
$$
\eps_1:=\frac{r}{4^{1/s}}, \quad \eps_2:=\frac{r^s}{4},
$$
and let~$K_1$ be the corresponding number of iterations from the conclusion of Lemma~\ref{l:r-contr}.
Then for any $a\in J$ and any $x,y\in \Sc$ with $\dist(x,y)\ge r$ we have
\begin{equation}\label{eq:r-s}
\varphi(x,y)\ge r^s,
\end{equation}
and
\begin{multline}\label{eq:long}
\E \varphi(f_{K_1,a,\bo}(x),f_{K_1,a,\bo}(y)) \le \P \left( \dist( f_{K_1, a, \bo} (x), f_{K_1, a, \bo} (y)) \ge \eps_1 \right)\cdot 1^s+
\\
+ \P \left( \dist( f_{K_1, a, \bo} (x), f_{K_1, a, \bo} (y)) <\eps_1 \right) \cdot \eps_1^s \le
\\
\le  \eps_2 +\eps_1^s = \frac{r^s}{4} + \frac{r^s}{4}  = \frac{r^s}{2}.
\end{multline}
Joining~\eqref{eq:r-s} and~\eqref{eq:long}, we get the desired~\eqref{eq:phi-s}.
\end{proof}

We are now ready to conclude the proof of Proposition~\ref{p:phi} (modulo Lemma~\ref{l:r-contr}).
Indeed, consider the following random process on the pairs of points $(x,y)$: if they are closer
than $r$, we do $K_0$ random iterations of $f_{a,\omega}$, otherwise $K_1$ iterations.
Repeating this process untill the total number of random iterations exceeds a given number $K$,
we define a random Markov moment $t(\bo)$ such that $K\le t(\bo)\le K+\max(K_0,K_1)$. Then
\begin{equation}\label{eq:K-pr}
\E \varphi(f_{t(\bo),a,\bo}(x),f_{t(\bo),a,\bo}(y)) \le \lambda_K \varphi(x,y),
\end{equation}
where
$$
\lambda_K=\max ((1-\delta_2)^{K/K_0},(1/2)^{K/K_1}).
$$
At the same time, application of any $f_{a,\omega}$ (or its inverse) changes the distances with multiplier at most~$M^2$ (recall
that the norms of all the matrices $T_{a,\omega}$ are uniformly bounded by~$M$).
Hence~\eqref{eq:K-pr} implies that if instead we stop the process exactly after~$K$ random iterations, we get
$$
\E \varphi(f_{K,a,\bo}(x),f_{K,a,\bo}(y)) \le M^{2s\max(K_0,K_1)} \lambda_K \varphi(x,y).
$$
The first factor is constant, while $\lambda_K \to 0$ as $K\to \infty$. Taking $K$ such that
$M^{2s\max(K_0,K_1)} \lambda_K<1/2$, we obtain \eqref{eq:phi-contraction}, as desired.
\end{proof}


%
%

%

%

For the sake of completeness, we provide here a proof of Lemma~\ref{l:r-contr}.

\begin{proof}[Proof of Lemma~\ref{l:r-contr}]
We start by recalling some standard general arguments from the theory of random dynamical systems. Namely, it is known that the Furstenberg's theorem implies an individual contraction of orbits: 
\begin{prop}\label{p:contr1}
For any $a\in J$, for any $x,y\in \Sc$ for almost all $\bo$ one has
$$
\lim_{n\to\infty} \dist (f_{n,a,\bo}(x),f_{n,a,\bo}(y)) =0.
$$
\end{prop}

Note, that this automatically implies the (almost-sure) existence of a (random) ``repelling'' point $r_-(a,\bo)$,
such that all the points except for it approach each other:
\begin{lemma}\label{l:r-}
For any $a\in J$, for almost all $\bo$ there exists a (random) point $r_-=r_-(a,\bo)\in \Sc$ such that
\begin{equation}\label{eq:r-}
\forall x,y\in \Sc, \, x,y\neq r_- \quad \lim_{n\to\infty} \dist (f_{n,a,\bo}(x),f_{n,a,\bo}(y)) =0.
\end{equation}
\end{lemma}
\begin{proof}
Proposition~\ref{p:contr1} implies that for any two points $x_0,y_0$ the length of the positive direction arc $[x_n,y_n]$, joining their images
$$
x_n:=f_{n,a,\bo}(x_0), \quad y_n:=f_{n,a,\bo}(y_0),
$$
tends either to 0, or to~1. Now, take an arbitrary $l$ and consider $l$ initial points $x_0^{i}=i/l, \, i=1,\dots, l$ on the circle. For any fixed $l$, the images
$$
x_n^{(i)}:=f_{n,a,\bo}(x_0^{(i)})
$$
of these points almost surely approach each other, and hence (almost surely) exactly one of the arcs $[x_n^{(i)},x_n^{(i+1)}]$ has its length tending to~$1$, while the length of the other ones tend to zero. We denote this arc by $I_l$ (omitting the dependence on $a$ and $\omega$). If neither of two initial points $x$ and $y$ does not belong to $I_l$, the distance between their images also tends to zero. Now, as $l$ becomes larger and larger, the arcs $I_l$ become smaller and smaller, and in the limit we see that there exists a random point $r_-$ such that~\eqref{eq:r-} holds.

In fact, translating the above description, we see that the preimages of the Lebesgue measure by the dynamics converge to the Dirac measure:
$$
(f_{n,a,\bo}^{-1})_* \Leb \to \delta_{r_-(\bo)} \quad \text{as } n\to\infty.
$$
\end{proof}

The  description above implies that the length of the composition that we have to apply to bring two points $x$ and $y$ close to each other with a high probability can be chosen uniformly in $x$ and $y$, at least for any fixed parameter $a$:
\begin{lemma}\label{l:K0}
For any $a\in J$ and any $\eps_1,\eps_2>0$ there exists $K_0$ such that
\begin{equation}\label{eq:f-K0}
\forall x,y\in \Sc \quad \P \left( \dist( f_{K_0, a, \bo} (x), f_{K_0, a, \bo} (y)) <\eps_1 \right) > 1- \eps_2.
\end{equation}
\end{lemma}
\begin{proof}
Note that the point $r_-(a,\bo)$ from the conclusion of Lemma~\ref{l:r-} satisfies the relation
$$
r_-(a,\bo) = f_{a,\omega_1}^{-1} (r_-(a,\sigma \bo));
$$
indeed, the application of $f_{a,\omega_1}$ sends the conclusion~\eqref{eq:r-} for $\bo$ to the conclusion~\eqref{eq:r-} for $\sigma \bo$. In particular,
the distribution of values of $r_-$, that is, the measure $\nu_-$ defined as $\nu_-:=(r_-(a,\cdot))_* \mu^{\N}$ satisfies
$$
\nu=\int (f_{a,\omega}^{-1})_*(\nu) \, d\mu(\omega).
$$
In other words, the measure $\nu$ is \emph{stationary} for the system of the inverse maps $f_{a,\cdot}^{-1}$.

Such a measure $\nu_-$ is known to be non-atomic: otherwise, the set of atoms of highest possible weight
would be completely invariant (e.g. see~\cite[Proposition~6]{KN}), and this would contradict to
the Furstenberg condition~(A1). Hence, given $\eps_2>0$, we can find a (sufficiently large) $l$ such that
$$
\forall i=1,\dots, l \quad \P \left(r_-(a,\bo)\in [x_0^{(i)},x_0^{(i+1)}] \right) < \frac{\eps_2}{3}.
$$

In turn, for every $i$ the lengths $|[x_n^{(i)},x_n^{(i+1)}]|$ tend to 0 or 1. Hence, for a sufficiently large $n$ one has with the probability at least $1-\frac{\eps_2}{3}$
$$
\forall i \quad |[x_n^{(i)},x_n^{(i+1)}]|<\eps_1 \, \text{ or } |[x_n^{(i)},x_n^{(i+1)}]|>1-\eps_1,
$$
and the second possibility happens for the interval that contains $r_-(a,\bo)$. We will denote the index $i$ for such an interval (if it exists) by~$i_-=i_-(a,\bo)$.

Denote such $n$ by $K_0$ and show that for it the conclusion of the lemma holds. Indeed, for any two points $x,y\in\Sc$ there are at most two indices $i$ such that
$$
x\in (x_0^{(i)},x_0^{(i+1)}) \, \text{ or } y\in (x_0^{(i)},x_0^{(i+1)}).
$$
Hence, with the probability at least $1-\frac{\eps_2}{3}-2\frac{\eps_2}{3}=1-\eps_2$ the index $i_-$ is defined, and  we have
$$
x,y \notin (x_0^{(i_-)},x_0^{(i_-+1)}).
$$
On the other hand, if this is the case, one of the two arcs $[x,y]$ and $[y,x]$ does not intersect $(x_0^{(i_-)},x_0^{(i_-+1)})$. Hence, its image, joining $f_{K_0,a,\bo}(x)$ and $f_{K_0,a,\bo}(y)$, does not intersect the image $(x_0^{(i_-)},x_0^{(i_-+1)})$, that is of length more than $1-\eps_1$. Hence, we get the desired
$$
\dist (f_{K_0,a,\bo}(x) ,f_{K_0,a,\bo}(y))<\eps_1.
$$
\end{proof}
\begin{remark}\label{r:more}
Note that if for some $a\in J$ the conclusions of Lemma~\ref{l:K0} hold for some $K_0$, they automatically hold for any $K>K_0$. Indeed, we can decompose
$$
f_{K,a,\bo}(x)= f_{K_0,a,\bo'}(x'), \quad f_{K,a,\bo}(y)= f_{K_0,a,\bo'}(y'),
$$
where
$$
x'=f_{K-K_0,a,\bo}(x), \quad  y'=f_{K-K_0,a,\bo}(y), \quad \bo' = \sigma^{K-K_0} \bo.
$$
For any $x,y\in \Sc$ conditionally to any $\omega_1,\dots, \omega_{K-K_0}$ the points $x', y'$ are non-random, while
$\bo'$ is independent from them. Applying Lemma~\ref{l:K0} to $x', y'$ and then averaging over $\omega_1,\dots, \omega_{K-K_0}$
(in other words, applying the total probability formula), we get the desired estimate.
\end{remark}

We are now ready to conclude the proof of Lemma~\ref{l:r-contr}. Namely, the interval $J$ is a compact interval, and for any $a\in J$
there exists the corresponding $K_0(a)$ in the sense of Lemma~\ref{l:K0}. On the other hands, its conclusion~\eqref{eq:f-K0} is an open condition,
hence due to the continuous dependence on~$a$ for the same value $K_0(a)$ the same conclusion holds in some open neighborhood $I_a\ni a$
of~$a$ in~$J$.

Such neighborhoods form an open cover of $J$. Due to the compactness of $J$ there exists a finite subcover $I_{a_1}, \dots, I_{a_p}$. Take
$$
K=\max_i K_0(a_i).
$$
Then, for each of the neighborhoods $I(a_i)$, we have $K>K_0(a_i)$, and due to Remark~\ref{r:more}, the desired conclusion holds for all $a\in I(a_i)$.
As these neighborhoods form a cover of~$J$, we finally get the conclusion of the lemma for all $a\in J$.
\end{proof}

Finally, we use Proposition~\ref{p:phi} to estimate the behavior of random iterations
with different parameters:
\begin{coro}\label{c:l-a}
Fix constants $K,s$ given by Proposition~\ref{p:phi}. There exists a constant $C_\varphi$ such that for any $a,a'\in J$, $x,y\in \Sc$ one has
\begin{equation}\label{eq:C-a}
\E \varphi(f_{K,a,\bar\omega}(x),f_{K,a',\bar\omega}(y)) \le \frac{1}{2} \varphi(x,y) + C_{\varphi} |a-a'|^s.
\end{equation}
\end{coro}
\begin{proof}
Since $s\in (0,1]$, we have
\begin{multline*}
    \varphi(f_{K,a,\bar\omega}(x),f_{K,a',\bar\omega}(y)) =(\dist(f_{K,a,\bar\omega}(x),f_{K,a',\bar\omega}(y)))^s\le \\
    (\dist(f_{K,a,\bar\omega}(x),f_{K,a,\bar\omega}(y))+\dist(f_{K,a,\bar\omega}(y),f_{K,a',\bar\omega}(y)))^s\le \\
    (\dist(f_{K,a,\bar\omega}(x),f_{K,a,\bar\omega}(y))^s+(\dist(f_{K,a,\bar\omega}(y),f_{K,a',\bar\omega}(y)))^s\le \\
    \varphi(f_{K,a,\bar\omega}(x),f_{K,a,\bar\omega}(y))+C_\varphi|a-a'|^s.
\end{multline*}
Application of Proposition~\ref{p:phi} completes the proof.
\end{proof}
Iterating Corollary \ref{c:l-a}, we get
\begin{coro}\label{c:l-a1}
There are positive constants $C_{\varphi}'$ and $C_{\varphi}''$ (that depend on $K, s$, and constants $\Lx$, $\La$ from Lemma \ref{l.shift}) such that for any $l\in \N$, $k'<K$, and any $a,a'\in J$, $x,y\in \Sc$ we have
\begin{equation}\label{e.in}
\E \varphi(f_{lK+k',a,\bar\omega}(x),f_{lK+k',a',\bar\omega}(y)) \le \frac{C_{\varphi}'}{2^l}\varphi(x,y)+ C_{\varphi}'' |a-a'|^s.
\end{equation}
\end{coro}
\begin{proof}
Corollary~\ref{c:l-a} says that
\begin{equation}\label{eq:averaging-phi}
\E \varphi(f_{K,a,\bar\omega}(x),f_{K,a',\bar\omega}(y)) \le g(\varphi(x,y)),
\end{equation}
where
$$
g:d\mapsto \frac{d}{2}+ C_{\varphi} |a-a'|^s.
$$

The map $g$ is linear, with the unique fixed point $d_*:=2C_{\varphi} |a-a'|^s$. Iterating the application of
Corollary \ref{c:l-a}, we get
\begin{multline*}
    \E \varphi(f_{lK,a,\bar\omega}(x),f_{lK,a',\bar\omega}(y)) \le g^l(\varphi(x,y)) = \frac{1}{2^l} (\varphi(x,y)-d_*)+d_* \le \\ \frac{1}{2^l}\varphi(x,y)+2C_{\varphi} |a-a'|^s.
\end{multline*}
Notice that this proves (\ref{e.in}) for $k'=0$. To prove (\ref{e.in}) for $k'>0$, use Lemma \ref{l.shift} to replace $x$ and $y$ by $f_{k',a,\bar\omega}(x)$ and $f_{k',a',\bar\omega}(y)$.
\end{proof}

\subsection{Distribution of jump intervals}\label{ss:law}

This section is devoted to the proof of the ``Quantity'' and the  ``Measure'' parts of Theorem~\ref{t:main}, i.e. parts {\bf I} and {\bf V}.

Let us recall that for a given large $n$, $M$ is the number of exceptional (``jump'') intervals on $[b_-, b_+]$, and those intervals were denoted by $\{J_{i_k}\}_{k=1,\ldots, M}$. Let us also recall that for a given exceptional interval $J_{i_k}$ the value of the corresponding iterate $m_k$ from part {\bf IV} of Theorem \ref{t:main} was defined as $m_0+\eps'n$, where $m_0$ is the index that corresponds to the first moment when $|X_{m, i_k}|$ becomes larger than $\eps'$, as defined in Proposition \ref{p:classes}.

We know that  $\frac{1}{n} \tf_{n,a,\bar \omega}(\tx_0)$ converges to $\rho(a)$ uniformly on $J$ as $n\to \infty$. Moreover, we know that due to Proposition \ref{p.ldtrot}, the Large Deviation principle for the rotation number, with probability exponentially close to one for any $a\in J$
$$
\left|\frac{1}{n}\tf_{n, a, \bo}(\tx_0)-\rho(a)\right|\le \eps'.
$$
In particular,
$$
\left|\tf_{n, b_+, \bo}(\tx_0)-\tf_{n, b_-, \bo}(\tx_0)-\left(\rho(b_+)-\rho(b_-)\right)n\right|\le 2\eps'n.
$$
On the other hand, below we prove the following statement. Recall that we denoted $\tx_{m, i}=\tf_{m, b_i, \bo}(\tx_0)$, and intervals $X_{m,i}$ were defined by (\ref{e.Xmi}).
\begin{prop}\label{p:ones}
 For any $\eps'>0$ there exists $\zeta_5>0$ such that for any $m\le n$
\begin{equation}\label{eq:ones}
\Prob\left( \frac{(\tx_{m,N}-\tx_{m,0}) - \#\{j \,:\, |X_{m,j}|\ge 1 \}}{n}>\eps'\right) < \exp(-\zeta_5 \sqrt[4]{n}).
\end{equation}
\end{prop}
Proposition \ref{p:ones}, applied to $m=n$, gives that with probability at least $1-\exp(-\zeta_5 \sqrt[4]{n})$, $M=\#\{j \,:\, |X_{n,j}|\ge 1 \}$ is $\eps'n$-close to $\tf_{n, b_+, \bo}(\tx_0)-\tf_{n, b_-, \bo}(\tx_0)$ and, hence, $3\eps'n$-close to $\left(\rho(b_+)-\rho(b_-)\right)n$. This gives the part {\bf I} (``Quantity'') of Theorem \ref{t:main}.

The part {\bf V} (``Measure'') follows from Proposition \ref{p:ones} and Corollary \ref{c.ldtrot} in a similar way. Namely, define the measure
$$
\xi_n=\frac{1}{n}\sum_{k=1}^{M}\delta\left(\frac{m_k}{n}, a_{i_k} \right).
$$
Let us show that for arbitrarily small  $\eps>0$, the measure $\xi_n$ is $\eps$-close to $\Leb\times \DOS$ for sufficiently large $n$.
In order to do that it is enough to show that for any $a\in J$, $s\in [0,1]$
$$
\xi_n([0,s]\times[\bmin,a]) = \frac{1}{n} \# \{ k=1, \ldots, M \mid m_{i_k} \le ns, \, a_{i_k}\in [\bmin,a]\}
$$
is sufficiently close to $s\cdot(\rho(a)-\rho(\bmin))$.

From Corollary \ref{c.ldtrot} we know that with probability exponentially close to one we have
$$
\left|\tf_{m, \edited{a}, \bo}(\tx_0)-\tf_{m, b_-, \bo}(\tx_0)-\left(\rho(\edited{a})-\rho(b_-)\right)m\right|\le 2\eps'n.
$$
At the same time, with probability at least $1-\exp(-\zeta_5 \sqrt[4]{n})$, \edited{for all $m\in [0,n]$} the difference $\tf_{m, \edited{a}, \bo}(\tx_0)-\tf_{m, b_-, \bo}(\tx_0)$ is $\eps'n$-close to the number of jump intervals on $[\bmin, a]$ with the corresponding indices $m_k\le m$. Hence, if we take $m=sn$, then  $\xi_n([0,s]\times[\bmin,a])$ is $3\eps'$-close to $\frac{m}{n}\left(\rho(\edited{a})-\rho(b_-)\right)=s\cdot\left(\rho(\edited{a})-\rho(b_-)\right)$. This implies the part {\bf V} (``Measure'') of Theorem \ref{t:main}.


Let us now prove Proposition \ref{p:ones}.

Note that the increments $x_{m,N}-x_{m,0}$ and $x_{m+r,N}-x_{m+r,0}$ differ by at most $C_f r$, where $C_f$ is a uniform constant. Hence, instead of showing~\eqref{eq:ones} it suffices to establish that for some $\zeta_5>0$ for any $n$ sufficiently large we have \edited{for all $m\in [0,n]$}
\begin{equation}\label{eq:ones-2}
\Prob\left( \frac{(x_{m,N}-x_{m,0}) - \#\{j \,:\, |X_{m+r,j}|\ge 1 \}}{n}\ge \frac{\eps'}{2} \right) < \exp(-\zeta_5 \sqrt[4]{n}),
\end{equation}
where $r:=[\sqrt{n}]$. 

The main step in the proof of this proposition is the following lemma, allowing us to launch a ``bisection'' procedure.

\begin{lemma}\label{l:divido}
For all sufficiently large $n$ the following holds.
Let
$$
b_{(1)}<b_{(2)}<b_{(3)}, \quad b_{(j)}\in J, \quad b_{(2)}-b_{(1)}\ge \frac{|J|}{N},\quad b_{(3)}-b_{(2)}\ge \frac{|J|}{N}.
$$
Also, let $m\le n-[\sqrt[3]{n}]$, and let $z_{(1)}<z_{(2)}<z_{(3)}$ be  points on the real line. Define
$$
z_{t,j}=\tf_{[m, t],b_{(j)},\bo}(z_{(j)}), \ \ j=1,2,3.
$$
Then with  probability at least $1-\exp(-\sqrt[4]{n})$
$$
[z_{m',2}-z_{m',1}]+ [z_{m',3}-z_{m',2}] \ge [z_{(3)}-z_{(1)}]
$$
where $m'=m+ [\sqrt[3]{n}]$.
\end{lemma}

Let us first deduce Proposition~\ref{p:ones} from Lemma~\ref{l:divido}.
\begin{proof}[Proof of Proposition~\ref{p:ones}]
Let us prove~\eqref{eq:ones-2}. To do so, we define inductively a branching random process on the set of intervals of parameter of the form $[b_i,b_j]$. That is, to each moment $m_q:=m+q [\sqrt[3]{n}]$, we associate a set of intervals $\{[b_{i_{q,l}},b_{j_{q,l}}]\}_{l=1}^{l_q}$, such that
 $\left[\tf_{[m,m_q],b_{j_{q,l}},\bo}(\tx_0)-\tf_{[m,m_q],b_{i_{q,l}},\bo}(\tx_0)\right]>0$. This will at the end provide us the desired intervals $X_{m+r,i}$ of length more than one; however, we reserve a (small) chance for the construction to result instead in FAIL.
Let
$$
I(i,j,m_q):=[x_{m_q,b_j}-x_{m_q,b_i}]
$$
be the integer part of the increment at the moment $m_q$ over the parameter interval $[b_i,b_j]$.

The branching process is defined in the following way:
\begin{itemize}
\item We start at the moment $m$ with the only interval $[b_0,b_N]=J$.
\item For each interval $[b_i,b_j]$ that is present at some moment $m_q=m+q [\sqrt[3]{n}]$, at the next moment $m_{q+1}$ we do as follows. If $j=i+1$, we leave it as it is. If $j>i+1$, we take $p=i+[(j-i)/2]$ and consider two parameter subintervals, $[b_i,b_p]$ and $[b_p,b_j]$.
\item For these intervals, if we have
$$
I(i,p,m_{q+1})+I(p,j,m_{q+1})< I(i,j,m_q)
$$
all the process results in FAIL.
\item Otherwise,
\begin{equation}\label{eq:keeping}
I(i,p,m_{q+1})+I(p,j,m_{q+1})\ge I(i,j,m_q)
\end{equation}
the descendants of this interval at the moment $m_{q+1}$ will be those among $[b_i,b_p],[b_p,b_j]$, for which
the corresponding integer parts of the increment are positive.
\end{itemize}

Note that in at most $R:=[\log_2 N]+1$ steps, if the process does not result in FAIL, all the descendants will be of the form $[b_i,b_{i+1}]$ (as the difference $j-i-1$ is reduced at least twice on each step). On the other hand, for each interval present at some moment of time, the corresponding integer part of the increment is at least~$1$, hence there is at most $C_f n$ descendants present at any moment. Hence, due to Lemma~\ref{l:divido} the total probability of the process resulting in FAIL is at most $C_f n \cdot ([\log_2 N]+1) \cdot \exp(- \sqrt[4]{n})$.

Then, by an induction on $k$ we obtain that the sum of $\sum_{l=1}^{l_q} I(i_{q,l},j_{q,l},m_q)$ of the integer increments corresponding to the selected intervals is non-decreasing: the induction step is exactly~\eqref{eq:keeping}. Thus, at the moment $m_{R}$ we find the desired parameter intervals $[b_{i_{R,l}},b_{i_{R,l}+1}]$ for which $|X_{m_{R},i_{R,l}}|\ge 1$ and such that
$$
\sum_{l=1}^{l_R} [|X_{m_{R},i_{R,l}}|] \ge [x_{m,N}-x_{m,0}].
$$

Proposition~\ref{p:classes} implies that the integer parts under the sum are not greater than~$1$ with the probability at least $1-\exp(-c_1 \sqrt[4]{n})$, and if this is the case, the sum in the left hand side is equal to the number $l_R$ of summands. As the integer part of the increment cannot decrease, and $r=[\sqrt{n}]> R [\sqrt[3]{n}]$, we finally get the desired
$$
\#\{i \,:\, |X_{m+r,i}|\ge 1 \} \ge \#\{i \,:\, |X_{m_R,i}|\ge 1 \} \ge l_R = \sum_{l=1}^{l_R} [|X_{m_{R},i_{R,l}}|] \ge  [x_{m,N}-x_{m,0}],
$$
concluding the proof of~\eqref{eq:ones-2}, as we have a lower bound for the probability
$$
1- C_f n \cdot ([\log_2 N]+1) \cdot \exp(- \sqrt[4]{n}) - \exp(-c_1 \sqrt[4]{n}) >  1-\exp(-\zeta_5 \sqrt[4]{n})
$$
for any $\zeta_5<\min(c_1,1)$ for all $n$ sufficiently large.

\end{proof}

Now, all that is left is to prove Lemma~\ref{l:divido}.
\begin{proof}[Proof of Lemma~\ref{l:divido}]
Note that we can increase $z_{(1)}$ and decrease $z_{(3)}$ as soon as we do not change the value of $[z_{(3)}-z_{(1)}]$: if the conclusion of Lemma~\ref{l:divido} is satisfied for the new values, it
is also satisfied for the old ones. Moreover, increasing $z_{(1)}$ by~$1$ increases all its images exactly by~$1$, and the same
applies to $z_{(3)}$. Hence, it suffices to consider the situation
$$
z_{(1)}<z_{(2)}<z_{(3)}=z_{(1)}+1,
$$
to which a general case can be reduced. Let now $x$ and $y$ be the points on the circle that
are projections of $z_{(1)}$ (and thus of $z_{(3)}$) and of $z_{(2)}$ respectively, i.e. $\pi(z_{(1)})=\pi(z_{(3)})=x$, $\pi(z_{(2)})=y$.

Consider the iterations of $x$ and $y$ under the random dynamical system on the circle corresponding to the parameter~$b_{(2)}$ (for the same
sequence of iterations defined by~$\bo$). Let us show that if they approach each other at the moment $m'$ at the distance less
than $\frac{|J|\delta}{N}$, then we are done. Indeed, we have
\begin{multline*}
z_{m',1}=  \tf_{[m,m'],b_1,\bo}(z_{(1)}) =  \tf_{\omega_{m'},b_1} (\tf_{[m,m'-1],b_1,\bo}(z_{(1)})) \\ < \tf_{\omega_{m'},b_{(2)}} (\tf_{[m,m'-1],b_{(2)},\bo}(z_{(1)})) - \delta(b_{(2)}-b_{(1)})\le \tf_{[m,m'],b_{(2)},\bo}(z_{(1)}) - \frac{|J| \delta}{N},
\end{multline*}
where we have used the monotonicity assumption~(A4) and the assumption $b_{(2)}-b_1\ge \frac{|J|}{N}$.
In the same way we have
$$
z_{m',3} \ge  \tf_{[m,m'],b_{(2)},\bo}(z_{(3)}) + \frac{|J| \delta}{N}.
$$

Now, if the points $x$ and $y$ approach each other in such a way that the (positive direction) arc $[x,y]$ is expanded on almost all the circle (that is, becomes of length grater than $(1- \frac{|J|\delta}{N})$), then we have
$$
\tf_{[m,m'],b_{(2)},\bo}(z_{(1)})<\tf_{[m,m'],b_{(2)},\bo}(z_{(2)}) - \left(1- \frac{|J|\delta}{N}\right)
$$
and hence
\begin{multline*}
z_{m',1} <\tf_{[m,m'],b_{(2)},\bo}(z_{(1)}) - \frac{|J| \delta}{N}  \\
< \tf_{[m,m'],b_{(2)},\bo}(z_{(2)}) - \left(1- \frac{|J|\delta}{N}\right)- \frac{|J| \delta}{N} = z_{m',2}-1,
\end{multline*}
thus implying the desired $[z_{m',2}-z_{m',1}]\ge 1$.

In the same way, if the points $x$ and $y$ approach each other in such a way that the (positive direction) arc $[y,x]$ is expanded on almost all the circle (that is, becomes of length more than $(1- \frac{|J|\delta}{N})$), then we have
$$
\tf_{[m,m'],b_{(2)},\bo}(z_{(3)})>\tf_{[m,m'],b_{(2)},\bo}(z_{(2)}) + \left(1- \frac{|J|\delta}{N}\right)
$$
and hence
\begin{multline*}
z_{m',3} >\tf_{[m,m'],b_{(2)},\bo}(z_{(2)}) + \frac{|J| \delta}{N}  \\
> \tf_{[m,m'],b_{(2)},\bo}(z_{(2)}) + \left(1- \frac{|J|\delta}{N}\right) + \frac{|J| \delta}{N} = z_{m',2}+1,
\end{multline*}
thus implying the desired $[z_{m',3}-z_{m',2}]\ge 1$.

Let us now show that indeed the points $x$ and $y$ approach each other with the desired probability.
Applying Proposition~\ref{p:phi}, we get that
\begin{equation}\label{eq:expect-phi}
\E \varphi(f_{[m,m'],b_{(2)},\bo}(x),f_{[m,m'],b_{(2)},\bo}(y)) \le \frac{1}{2^{[(m' - m)/K]}} = \frac{1}{2^{[[\sqrt[3]{n}]/K]}}.
\end{equation}
On the other hand, two points $x', y'$ are $\frac{|J| \delta}{N}$-close to each other if and only if
$$
\varphi(x',y')< \left( \frac{|J| \delta}{N} \right)^s.
$$
Combining the Chebyshev inequality with~\eqref{eq:expect-phi}, we see that the probability that the random images $f_{[ m, m'],b_{(2)},\bo}(x),f_{[m, m'],b_{(2)},\bo}(y)$ of $x$ and $y$ will not be $\frac{|J| \delta}{N}$-close to each other is at most
\begin{multline*}
\Prob \left( \dist (f_{[ m, m'],b_{(2)},\bo}(x),f_{[ m, m'],b_{(2)},\bo}(y)) > \frac{|J| \delta}{N}\right)
\\
< \left( \frac{|J| \delta}{N} \right)^{-s} \cdot \frac{1}{2^{[[\sqrt[3]{n}]/K]}}= \frac{1}{|J|^s \delta^s} \cdot  \frac{[\exp(\sqrt[4]{n})]}{\exp([[\sqrt[3]{n}]/K] \log 2)},
\end{multline*}
and the right hand side is smaller than $\exp(-\sqrt[4]{n})$ for all $n$ sufficiently large.

We have obtained the desired lower bound for the probability that the random images of $x$ and $y$ will be sufficiently close to each other. This concludes the proof of Lemma~\ref{l:divido} (and hence of Proposition~\ref{p:ones}).
\end{proof}


\subsection{Intervals characterization}\label{ss:jump-classes}

This section is devoted to the proof of Proposition~\ref{p:classes}, describing the behaviour of the intervals~$X_{m,i}$.
Our first step will be to understand the behaviour of an individual interval, that is, for a specific index~$i$. To do so,
we take an initial moment $m<n$, two points $\tilde y_{m,1}, \tilde y_{m,2}\in \mathbb{R}$ (that will be later interpreted
as the end points of the interval $X_{m,i}$) and define
$$
\tilde y_{k,1}=\tf_{[m,k], b_{i-1}, \bar \omega}(\tilde y_{m,1}), \tilde y_{k,2}=\tf_{[m,k], b_{i}, \bar \omega}(\tilde y_{m,2}),
$$
where $k=m+1, \ldots, n$.

We then show that
\begin{itemize}
\item if this interval was small, it will stay small till the last ($n$-th) iteration with high probability (see Lemma~\ref{l.52} below); 
\item for any initial interval, it quickly (in $\eps'n$ steps) becomes either of length close to~$0$,
or of length close to (and larger than)~$1$, and stays like that till the last ($n$-th) iteration (see Lemma~\ref{l.intervals}).
\end{itemize}

Note that initially all the intervals are quite small (they vanish at $m=0$, and are of length $\sim \frac{\const}{N}$ at the moment $m=1$).
But the above statements do not guarantee that they all will stay small: even if each individual interval stays small with high
probability, there are too many of them ($N=[\exp(\sqrt[4]{n})]$), so among this huge number there may be ones making
and ``individually-improbable'' growth. In fact, there should be: we know from Proposition~\ref{p:ones} that there should be jump intervals,
and that most of the increment $\tx_{m,N}-\tx_{m,0}$ is concentrated on them.

The key to the proof here is the following argument. Instead of considering the evolution of all the $N$ intervals $\{X_{m,i}\}_{m=1}^{n}$,
we consider only those among them that at some moment $m$ become larger than~$\eps'$; we call such intervals \emph{suspicious}.
The non-suspicious intervals are automatically small in the sense of Prop.~\ref{p:classes}, and hence for them there is nothing to prove.

At the same time, at each moment~$m$ there is at most $\frac{\tx_{m,N}-\tx_{m,0}}{\eps'}<\const \cdot m$ suspicious intervals, hence,
there is at most $\const\cdot n^2$ of them in total. Thus, Lemmas~\ref{l.52},~\ref{l.intervals} can be applied to them~\emph{simultaneously}: the probability
of a bad behavior of an individual interval is at most $\exp(-\const \sqrt[4]{n})$. This is done in Corollary~\ref{c:bad} and Lemma~\ref{l.easy} below, and their application concludes the proof of Proposition~\ref{p:classes}.

We call this scheme \emph{the dystopia argument}: as an analogy, even if a ``dystopic state'' does not have a power to
control all of its ``population'' ($\exp(\sqrt[4]{n})$~intervals), it suffices for it to control only those few ($\const \cdot n^2$) that it finds
``suspicious''. The reader is referred to~\cite{Za} for comparison.  

Let us realize this program. As we have already said, we first study the behavior of the end points of the
intervals $X_{m, i}$ for a specific index $i=1, 2, \ldots, N$. In order to do that take some $m<n$ and two
points $\tilde y_{m,1}, \tilde y_{m,2}\in \mathbb{R}$ (that will be later interpreted as the end points of the
interval $X_{m,i}$) and define 
$$
\tilde y_{k,1}=\tf_{[m,k], b_{i-1}, \bar \omega}(\tilde y_{m,1}), \quad \tilde y_{k,2}=\tf_{[m,k], b_{i}, \bar \omega}(\tilde y_{m,2}),
$$
where $k=m+1, \ldots, n$.

We first consider how the interval $[\ty_{k,1},\ty_{k,2}]$ can become longer than~$1$ (so that its projection covers all the circle).
It is easy to see that at the first moment when it happens, the projections of $\ty_{k,1}$ and $\ty_{k,2}$ are very close to each other.

Let us denote $\gamma:=\exp\left(-\sqrt[4]{n}\right)$; then the parameter increments $b_i-b_{i-1}=\frac{|J|}{N}$ are comparable to~$\gamma$.
The next lemma shows that once these two orbits are sufficiently close to each other, they most probably stay close till the last ($n$-th) iteration:

\begin{lemma}\label{l.52}
For all sufficiently large $n$, if $|\tilde y_{m,1}-\tilde y_{m,2}|\le \gamma^{1/3}$, then  with probability at least $1-\gamma^{s/20}$ we have
$$
|\tilde y_{k,1}-\tilde y_{k,2}|\le \gamma^{1/12}
$$
for all $k= m,\dots, n$.
\end{lemma}
\begin{proof}
If $|\tilde y_{m,1}-\tilde y_{m,2}|\le \gamma^{1/3}$, then $\varphi(y_{m,1}, y_{m,2})\le \gamma^{s/3} $, and due to Corollary \ref{c:l-a1} for any $k>m, k=Kl+k', k'<K$,
\begin{multline*}
   \E\varphi(y_{k, 1}, y_{k, 2})\le \frac{C_{\varphi}'}{2^l}\varphi(y_{m,1}, y_{m,2})+C_{\varphi}''|b_i-b_{i-1}|^s\le \\
\frac{C_{\varphi}'}{2^l} \gamma^{s/3}+C_{\varphi}''|J|^s \gamma^{s}<
\gamma^{s/6}
\end{multline*}
for large $n$. Chebyshev inequality implies that
\begin{equation}\label{eq:phi-gamma}
\P\left(\varphi(y_{k, 1}, y_{k, 2})>\gamma^{s/12} \right)\le\frac{\gamma^{s/6} }{\gamma^{s/12} }=\gamma^{s/12}.
\end{equation}
Notice that
$$
\varphi(y_{k+1, 1}, y_{k+1, 2})=|\tilde y_{k+1,1}-\tilde y_{k+1, 2}|^s \quad \text{provided that } |\tilde y_{k,1}-\tilde y_{k, 2}|<\frac{1}{2}.
$$
Since $|\tilde y_{m,1}-\tilde y_{m,2}|\le \gamma^{1/3}$, at the first moment $k>m$ such that $|\tilde y_{\edited{k},1}-\tilde y_{\edited{k},2}|> \gamma^{1/12}$, if such moment exists, one has $\varphi(y_{k,1},y_{k,2})>\gamma^{s/12}$. Hence, for any $k=m+1,\dots,n$ the probability that this is the first such moment is upper bounded by
$$
\P\left(\left.|\tilde y_{k, 1}- \tilde y_{k, 2}|>\gamma^{1/12}  \right|  |\tilde y_{t, 1}- \tilde y_{t, 2}|\le \gamma^{1/12} \ \text{for}\ m\le t<k \right)\le \gamma^{s/12}
$$
due to~\eqref{eq:phi-gamma}. Summing it over $k$, we finally get
\begin{multline*}
  \P\left(|\tilde y_{k, 1}- \tilde y_{k, 2}|>\gamma^{1/12} \ \text{for some}\ k=m+1, \ldots, n\right)= \\
  \sum_{k=m+1}^n\P\left(\left.|\tilde y_{k, 1}- \tilde y_{k, 2}|>\gamma^{1/12}  \right|  |\tilde y_{t, 1}- \tilde y_{t, 2}|\le \gamma^{1/12} \ \text{for}\ m\le t<k \right) \\
\le  \sum_{k=m+1}^n\P\left(\varphi(y_{k, 1}, y_{k, 2})>\gamma^{s/12}\right)\le n\gamma^{s/12}<\gamma^{s/20}.
\end{multline*}
\end{proof}
Substituting $\ty_{m,1}+1$ instead of $\ty_{m,1}$ (shift by~$1$ commutes with the dynamics), we easily get the following
\begin{coro}\label{c.52}
Similarly, if $|\tilde y_{m,1} + 1 -\tilde y_{m,2}|\le \gamma^{1/3}$, then  with probability at least $1-\gamma^{s/20}$ we have
$$
|\tilde y_{k,1} + 1 -\tilde y_{k,2}|\le \gamma^{1/12}
$$
for all $k= m,\dots, n$.
\end{coro}

Denote by $y_{k,1}, y_{k,2}\in \Sc$ the projections of $\tilde y_{k,1}, \tilde y_{k,2}\in \mathbb{R}$.
 Let $K,s$, and the function $\varphi$ be as in Proposition \ref{p:phi}. The following lemma shows
 that the projections on the circle of their images most probably quickly become close
 to each other (so that Lemma~\ref{l.52} and Corollary~\ref{c.52} become applicable).

\begin{lemma}\label{l:pair1}
 In the setting above, for any $y_{m,1}, y_{m,2}\in \Sc$ and for all sufficiently large $n\in \mathbb{N}$, with probability at least $1-\gamma^{s/3}$ one has
\begin{equation}\label{e.sinchro}
\dist_{\Sc}(y_{k,1}, y_{k,2})\le \gamma^{1/3},
\end{equation}
for all $k\ge m+K[\sqrt[3]{n}]$.
\end{lemma}
\begin{proof}
Let us apply Corollary~\ref{c:l-a1}. For any $k\ge m+K[\sqrt[3]{n}]$ we get
\begin{equation}\label{eq:expect-gamma}
\E_{\omega_{m+1},\dots,\omega_{k}} \varphi(y_{k,1},y_{k,2}) \le \frac{C_{\varphi}'}{\displaystyle 2^{[\sqrt[3]{n}]}} +  C_{\varphi}'' \left(b_{i}-b_{i-1}\right)^s \le
3C_{\varphi}''|J|^s \gamma^{s};
\end{equation}
recall that
$$
b_i-b_{i-1}=\frac{|J|}{N}\sim |J| \cdot \gamma, \quad \gamma= \exp(-\sqrt[4]{n}))\gg 2^{-\sqrt[3]{n}}.
$$
Due to Chebyshev inequality, we have
\begin{multline*}
\P\left(\dist_{\Sc}(y_{k,1},y_{k,2})> \gamma^{1/3} \right)=\P\left(\varphi(y_{k,1},y_{k,2})>\gamma^{s/3}\right)\le\\
\le \frac{\E\varphi(y_{k,1},y_{k,2})}{\gamma^{s/3}}\le 3C_{\varphi}''|J|^s \gamma^{2s/3}
\end{multline*}
for large $n$; the last inequality here is due to~\eqref{eq:expect-gamma}. Therefore,
\begin{multline*}
\P\left(\dist_{\Sc}(y_{k,1},y_{k,2})> \gamma^{1/3}\ \text{for some}\ k\ge m+K[\sqrt[3]{n}] \right)\le  \\
\le n\cdot  3C_{\varphi}''|J|^s \gamma^{2s/3}< \gamma^{1/3}
\end{multline*}
for all sufficiently large $n$.
\end{proof}


Let us remind that
$$
\Lx=  \sup_{y\in \mathbb{R}^1, a\in J, \omega\in \Omega}|\tf'_{a, \omega}(y)|, \quad  \La=  \sup_{y\in \mathbb{R}^1, a\in J, \omega\in \Omega}|\partial_a \tf_{a, \omega}(y)|
$$
are the space- and parameter-wise Lipschitz constants respectively, and $\delta>0$ is a small constant from the monotonicity condition $(A4)$. The next few lemmas guarantee that if the length of the interval $X_{k,i}$ becomes close to~$1$ (and thus stays close to~$1$), there will be an actual ``jump'', that is, the interval will become longer than~$1$ sufficiently quickly. The first two of these lemmas are devoted to the moment of the jump:

\begin{lemma}\label{l.skok}
 Suppose for some $k'>m$ we have $\tilde y_{k,1}<\tilde y_{k,2}\le \tilde y_{k,1}+1$ for $k=m, \ldots, k'-1$, and $\tilde y_{k',2}>\tilde y_{k',1}+1$. Then
\begin{equation}\label{eq:skok}
\tilde y_{k',2}-(\tilde y_{k',1}+1)\le \frac{\La|J|}{N}\le \const \cdot \gamma.
\end{equation}
\end{lemma}
\begin{proof}
We have $\tilde y_{k'-1,2}\le \tilde y_{k'-1,1}+1$, and hence
$$
\tilde y_{k',2}=\tilde f_{b_i,\omega_{k'}}(\tilde y_{k'-1,2})\le \tilde f_{b_i,\omega_{k'}}(\tilde y_{k'-1,1}+1)=\tf_{b_i,\omega_{k'}}(\tilde y_{k'-1,1})+1.
$$
This implies that
\begin{multline*}
\tilde y_{k',2}-(\tilde y_{k',1}+1) \le \tf_{b_i,\omega_{k'}}(\tilde y_{k'-1,1}) - \tf_{b_{i-1},\omega_{k'}}(\tilde y_{k'-1,1}) \le
\\
\le \La \cdot  (b_i- b_{i-1})=\frac{\La|J|}{N}.
\end{multline*}
The definition $N=[\exp(\sqrt[4]{n})]$ then implies the second inequality of~\eqref{eq:skok}.
\end{proof}

\begin{lemma}\label{l.pereskok}
If $\tilde x, \tilde y\in \mathbb{R}$ are such that $|\tilde x -\tilde y|<\frac{\delta|J|}{\Lx}\gamma$, then for any $\omega\in \Omega$ and any $i=1, \ldots N$ we have
$$
\tf_{b_i, \omega}(\tilde y)>\tf_{b_{i-1}, \omega}(\tilde x).
$$
\end{lemma}
\begin{proof}
Indeed, due to monotonicity assumption $(A4)$ we have
\begin{multline*}
    \tf_{b_i, \omega}(\tilde y)-\tf_{b_{i-1}, \omega}(\tilde x)=\left(\tf_{b_i, \omega}(\tilde y)-\tf_{b_i, \omega}(\tilde x)\right)+\left(\tf_{b_i, \omega}(\tilde x)-\tf_{b_{i-1}, \omega}(\tilde x)\right)\ge\\
    -\Lx|\tilde y-\tilde x|+\delta|J|\cdot \gamma >0
\end{multline*}
\end{proof}

Let us introduce the notation $U_{\eps}^+(x):=[x,x+\eps)$ and $U_{\eps}^- (x):=(x-\eps,x]$ for the right- and left- $\eps$-neighborhoods of the point $x\in \R $ respectively.

\begin{lemma}\label{l.pereskokprob}
If $\tilde y_{m, 2}\in U^-_{\gamma^{1/3}}(\ty_{m,1}+1)$, then with probability at least $1-3\gamma^{s/20}$ we have
$$
\tilde y_{k,2}\in U^+_{\gamma^{1/3}}(\ty_{m,1}+1) 
$$
for all $k> m+K\sqrt[3]{n}$.
\end{lemma}
\begin{proof}
Since $|\tilde y_{m,2}-\tilde y_{m,1}-1|<\gamma^{1/3}$, Lemma~\ref{l.52} implies that with probability at least $1-\gamma^{s/20}$ for all $k\ge m$ we have
$$
|\tilde y_{k,2}-\tilde y_{k,1}-1|<\gamma^{1/12}.
$$
Together with Lemma \ref{l:pair1} this implies that with probability at least
$$
1-\gamma^{s/20}-\gamma^{s/3}>1-2\gamma^{s/20}
$$
we have
\begin{equation}\label{eq:y-closeness}
|\tilde y_{k,2}-\tilde y_{k,1}-1|<\gamma^{1/3}
\end{equation}
for all $k\ge m+K[\sqrt[3]{n}]$; in other words, the images of $\ty_{m,2}$ stay close to those of $\ty_{m,1}+1$.

However, as $b_i>b_{i-1}$, the images of $\ty_{m,2}$ are in a sense ``pushed forward'' with respect to those of $\ty_{m,1}+1$. That is, consider a sequence
of points $\{\tilde z_k\}\subset \mathbb{R}$, $k=m, \ldots, n$, given by
$$
\tilde z_k=\tf_{[m,k], b_{i-1}, \bar \omega}(\ty_{m,2}).
$$
Then monotonicity assumption $(A4)$ implies that for all $k>m$ we have $\tilde y_{k,2}>\tilde z_k$. Due to Corollary \ref{c:l-a1} (applied for $a=a'=b_{i-1}$) for $k_0=m+K[\sqrt[3]{n}]$ we have
$$
\E\varphi(z_{k_0}, y_{k_0,1})\le \frac{C_{\varphi}'}{\displaystyle 2^{\sqrt[3]{n}}}\varphi(z_m, y_{m,1})\le \frac{C_{\varphi}'}{\displaystyle{2^{\sqrt[3]{n}}}}
$$
Now (in the same way as before), we have an lower bound for the probability that the images $z_{k_0}$ and $y_{k_0,1}$ are close to each other. Indeed, by Chebyshev inequality
\begin{multline*}
  \P\left(\text{dist}_{\Sc}(z_{k_0}, y_{k_0,1})\ge \frac{\delta|J|}{\Lx}\gamma \right) =   \P\left(\varphi(z_{k_0}, y_{k_0,1})\ge \left(\frac{\delta|J|}{\Lx}\right)^s \gamma^s    \right)\le
  \\
 \le \frac{C'_{\varphi} L^s}{\delta^s |J|^s} \cdot
    \frac{\exp\left(-\sqrt[3]{n}\log 2\right)}{\exp\left(-s\sqrt[4]{n}\right)}< \exp\left(-\frac{\log 2}{2}\sqrt[3]{n}\right) =\frac{1}{\displaystyle{2^{\sqrt[3]{n}/2}}}
\end{multline*}
Hence with probability at least
$$
1-2\gamma^{s/20}-\frac{1}{\displaystyle 2^{\sqrt[3]{n}/2}}>1-3\gamma^{s/20}
$$
we have
$$
\tilde y_{k_0,2}>\tilde z_{k_0}\ge \tilde y_{k_0,1}+1-\frac{\delta|J|}{\Lx}\gamma.
$$
Hence, either $\tilde y_{k_0,2}> \tilde y_{k_0,1}+1$, or, due to Lemma \ref{l.pereskok},
$\tilde y_{k_0+1,2}> \tilde y_{k_0+1,1}+1$. In both cases,
$$
\ty_{k,2}>\ty_{k,1}+1
$$
for
for $k=k_0+1$, and hence, by monotonicity, for all $k> m+K\sqrt[3]{n}$. Joining it with~\eqref{eq:y-closeness}, we finally get the desired
$$
\tilde y_{k,2}\in U^+_{\gamma^{1/3}} (\ty_{k,1}+1)
$$
for $k=k_0+1,\dots, n$.
\end{proof}

\begin{lemma}\label{l.intervals}
If $\tilde y_{m, 2}\in \left(\tilde y_{m,1}, \tilde y_{m,1}+1\right)$, then with probability at least $1-5\gamma^{s/20}$ either
$$
\tilde y_{k,2}\in U^+_{\gamma^{1/3}}(\tilde y_{m,1}),
$$
or
$$
\tilde y_{k,2}\in U^+_{\gamma^{1/3}}(\tilde y_{m,1}+1)
$$
for all $k> m+2K\sqrt[3]{n}$.
\end{lemma}
\begin{proof}
Due to Lemma \ref{l:pair1}, with probability at least $1-\gamma^{s/3}$ we have
$$
\tilde y_{k,2}\in \bigcup_{l\in \mathbb{N}\cup\{0\}} U_{\gamma^{1/3}} (\ty_{k,1}+l).
$$
Lemma \ref{l.skok} implies that in this case for some $k'\le m+K[\sqrt[3]{n}]$ we have
$$
\tilde y_{k',2}\in U^+_{\gamma^{1/3}} \left( \tilde y_{k',1} \right) \bigcup  U_{\gamma^{1/3}} \left(\tilde y_{k',1}+1 \right).
$$
Lemma \ref{l.52} now implies that with probability at least
$$
1-\gamma^{s/3}-\gamma^{s/20}>1-2\gamma^{s/20}
$$
we have
$$
\tilde y_{k,2}\in U^+_{\gamma^{1/12}} \left(\tilde y_{k,1}\right) \bigcup U_{\gamma^{1/12}} \left(\tilde y_{k,1}+1\right)
$$
for all $k\ge m+K[\sqrt[3]{n}]$.

Since we assume that (\ref{e.sinchro}) holds, this in fact implies that
$$
\tilde y_{k,2}\in U^+_{\gamma^{1/3}} \left(\tilde y_{k,1}\right) \bigcup U_{\gamma^{1/3}} \left(\tilde y_{k,1}+1\right)
$$
for all $k\ge m+K[\sqrt[3]{n}]$.

Now Lemma \ref{l.pereskokprob} implies that if $\tilde y_{k,2}\in U^-_{\gamma^{1/3}} \left(\tilde y_{k,1}+1\right)$  for $k=m+K[\sqrt[3]{n}]$, then with probability at least $1-3\gamma^{s/20}$ for all $k>m+2K[\sqrt[3]{n}]$ we have $\tilde y_{k,2}\in U^+_{\gamma^{1/3}} \left(\tilde y_{k,1}+1\right)$.

Summarizing, with probability at least
$$
1-3\gamma^{s/20} - 2\gamma^{s/20}=1-5\gamma^{s/20}
$$
for all $k>m+2K[\sqrt[3]{n}]$ either $\tilde y_{k,2}\in U^+_{\gamma^{1/3}}\left(\tilde y_{k,1} \right)$, or  $\tilde y_{k,2}\in U^+_{\gamma^{1/3}}\left(\tilde y_{k,1}+1\right)$.
\end{proof}

Now we are ready to prove Proposition \ref{p:classes}. Fix $\varepsilon'>0$.

\begin{defi}
Let us say that an interval $J_i=[b_{i-1}, b_i]$ is \emph{$m_0$-suspicious} if $|X_{k,i}|\le \varepsilon'$ for $k=1, 2, \ldots, m_0-1$, and $|X_{m_0,i}|>\varepsilon'$.
\end{defi}

\begin{defi}
Let us say that an interval $J_i=[b_{i-1}, b_i]$ is \emph{good} if it satisfies the claim of Proposition \ref{p:classes}, i.e. it is either small ($|X_{m,i}|\le \varepsilon'$ for all $m=1, \ldots, n$), or opinion-changer ($J_i$ is $m_0$-suspicious, and $|X_{m,i}|<\varepsilon'$ for all $m>m_0+\varepsilon'n$), or jump interval ($J_i$ is $m_0$-suspicious, and $1<|X_{m,i}|<1+\varepsilon'$ for all $m>m_0+\varepsilon'n$). Otherwise $J_i$ will be called \emph{bad}.
\end{defi}

In these terms, we get from Proposition~\ref{l.intervals} an immediate
\begin{coro}\label{c:bad}
For any $i$, $m_0$ we have
$$
\P\left(\left.J_i \ \text{is bad}\ \right|\ J_i \ \text{is}\ \text{$m_0$-suspicious}\right)\le 5\gamma^{s/20}.
$$
\end{coro}

In other words,
\begin{equation}\label{eq:b-s}
\P\left(J_i \ \text{is bad and}\  m_0\ \text{suspicious}\right)\le 5\gamma^{s/20} \cdot \P\left( J_i \ \text{is}\ m_0\ \text{suspicious}\right).
\end{equation}

Now, the number of suspicious intervals is easily bounded from above:

\begin{lemma}\label{l.easy}
For any $m_0=1, \ldots, n$, and any $\bar \omega\in \Omega^n$, number of $m_0$-suspicious intervals is not greater than $\frac{M^*}{\varepsilon'} m_0$, where
$$
M^*= [ \La \cdot |J| ]+1. 
$$
\end{lemma}
\begin{proof}
Indeed,
$$
\tx_{m,N}-\tx_{m,0} = \tf_{m,\bmax,\omega}(\tx_0) - \tf_{m,\bmin,\omega}(\tx_0) \le M^* m,
$$
and at the same time
$$
\tx_{m,N}-\tx_{m,0} = \sum_{i=1}^N |X_{m,i}|  \ge \eps' \cdot \# \{i=1,\ldots,N \mid J_i \ \text{is}\ m_0\ \text{suspicious}\}.\qedhere
$$
\end{proof}


Applying this lemma, for any $m_0=1,\dots,n$ we get an upper bound for the probability of the presence of a bad $m_0$-suspicious interval:
\begin{multline*}
\P\left( \exists i\in \{1,\ldots,N\} \mid J_i \text{ is bad and $m_0$-suspicious} \right) \le
\\
\E \left( \# \left\{ i\in \{1,\ldots,N\} \mid J_i \text{ is bad and $m_0$-suspicious} \right\} \right) =
\\
= \sum_{i=1}^N \P\left(J_i \ \text{is bad and $m_0$-suspicious}\right) \le
\\
\le \sum_{i=1}^N 5\gamma^{s/20} \, \P\left(J_i \ \text{is $m_0$-suspicious}\right) =
\\
= 5\gamma^{s/20} \E \left( \# \left\{ i\in \{1,\ldots,N\} \mid J_i \text{ is $m_0$-suspicious} \right\} \right) \le
\\
\le 5\gamma^{s/20} \cdot \frac{M^* m_0}{\eps'} \le \frac{5M^*}{\eps'} \cdot n\gamma^{s/20},
\end{multline*}
where the last inequality comes from Lemma~\ref{l.easy}.

Finally, summing over $m_0$, we get the desired
\begin{multline*}
  \P\left(\text{at least one of the intervals $J_i$ is bad}\right)\le \\
   \sum_{m_0=1}^n \P\left( \exists i\in \{1,\ldots,N\} \mid J_i \text{ is bad and $m_0$-suspicious} \right) \le
\\
\le n \cdot \frac{5M^*}{\eps'} \cdot n\gamma^{s/20} < \exp\left(-\frac{s}{40}\sqrt[4]{n}\right)
\end{multline*}
for large $n$. This completes the proof of Proposition \ref{p:classes}.


\section{Anderson localization}\label{s.al}

In this section we prove Theorem \ref{t.vector}.

The following two lemmas use only linear algebra. We assume that  a sequence of
matrices $A_j\in \SL(2,\R)$,  $\|A_j\|\le M$, and an initial vector $v_0\in \R^2\setminus \{0\}$ are given. Then, we consider the
corresponding sequence $v_m$ of images, defined by
\begin{equation}\label{eq:v-def}
v_m=A_m v_{m-1}, \quad m=1,\dots, n,
\end{equation}
and describe its possible behavior.

\begin{defi}
Given matrices $A_1,\dots,A_n \in \SL(2,\R)$ with $\|A_j\|\le M$,
we say that the product $A_n\dots A_1$ is \edited{\it $(r,\lambda)$-hyperbolic}
if for any $0\le m< m'\le n$ for the product $T_{[m,m']}:=A_{m'}\dots A_{m+1}$ one has
$$
\log \|T_{[m,m']}\| \in U_{r}(\lambda(m'-m)).
$$
\end{defi}
For instance, the conclusion~\ref{i:m2} of Theorem~\ref{t:main} combined with Proposition \ref{l:upper-finite} implies $(2n\eps,\lambda_F)$-hyperbolicity
for the corresponding product $F_a(\omega_n)\dots F_a(\omega_1)$. At the same time, the conclusion~\ref{i:m3}
implies $(2n\eps,\lambda_F)$-hyperbolicity of both products $F_a(\omega_{m_k})\dots F_a(\omega_1)$
and~$F_a(\omega_{n})\dots F_a(\omega_{m_k+1})$.


\begin{figure}
\begin{center}
\includegraphics[width=0.3\textwidth]{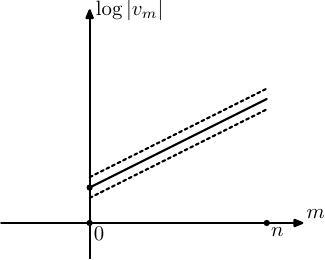} \hfill  \includegraphics[width=0.3\textwidth]{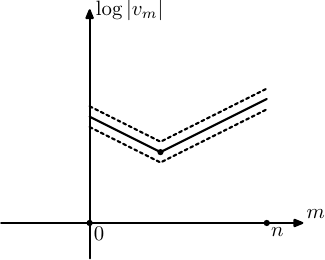}
\hfill \includegraphics[width=0.3\textwidth]{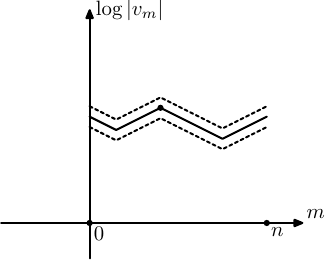}
\end{center}
\caption{Behaviour of log-norm of iterations of a given vector as in Lemmas \ref{l:line-shape}, \ref{l:V-shape}, \ref{l:W-shape}.
Dashed region corresponds to the $\eps n$-neighborhood}\label{f:norms}
\end{figure}

Let us first prove the following lemma.
\begin{lemma}[line-shape]\label{l:line-shape}
For any $M, \lambda, \eps>0$ there exists $\eps'>0$ with the following property. Assume that $v_0$ has the norm smaller than any other vector in the sequence~\eqref{eq:v-def}, i.e. $|v_0|\le |v_m|$ for all $m=1,\ldots, n$, and that the product $A_n\dots A_1$ is $(n\eps',\lambda)$-hyperbolic.
Then
$$
\forall m=0,1,\dots, n \quad \log |v_m| - \log |v_0| \in U_{n\eps}(m\lambda).
$$
\end{lemma}
Geometrically speaking, the conclusion this lemma states that if we plot the graph of $\log |v_m|$ for $m=0,1,\dots,n$, and then contract this graph $n$ times (both vertically and horizontally), then it will be in the $\eps$-neighborhood of a line with slope $\lambda$ (see Fig.~\ref{f:norms}).
\begin{proof}
Without loss of generality, we can assume that $v_0$ is a unit vector. Take another unit vector, $u_0$, that realizes the norm of the full product,
$$
|T_{[0,n]} u_0| = \|T_{[0,n]}\|,
$$
and consider the associated sequence of its intermediate images,
$$
u_m=A_m u_{m-1}, \quad m=1,\dots, n.
$$
Then, we have a lower bound for their norms: as $u_n=T_{[m,n]}u_m$,
\begin{equation}\label{eq:u-lower}
\log |u_m|\ge \log \frac{|u_n|}{\|T_{[m,n]}\|} = \log \|T_{[0,n]}\| - \log \|T_{[m,n]}\| \ge m\lambda - 2n\eps'.
\end{equation}

Next, $v_0$ and $u_0$ form a parallelogram of area at most $1$, hence
the same holds for the parallelogram formed by $v_m$ and $u_m$ for any~$m$.
As $|v_m|\ge 1$ by assumption and $|u_m|\ge \exp(m\lambda-2n\eps')$, the angle between the lines
passing through $v_m$ and $u_m$ does not exceed $\frac{\pi}{2}\exp(-m\lambda+2n\eps')$. Here we are using the inequality $\arcsin x \le \frac{\pi}{2}x$.

Now, we have
$$
\log |v_m|=\sum_{j=1}^m \log \frac{|A_j v_{j-1}|}{|v_{j-1}|} = \sum_{j=1}^m \phi_{A_j}([v_{j-1}]),
$$
where $\phi_A$ is a function on the projective line $\R P^1$, defined by
$$
\phi_A([v])=\log \frac{|Av|}{|v|}
$$
for any nonzero vector $v$ (where $[v]$ is the corresponding point of $\R P^1$).

\edited{The} family of the functions $\phi_A$ for $A\in SL(2,\R)$, $\|A\|\le M$, is equicontinuous on $\R P^1$. Hence, for any $\eps>0$
there exists $\delta>0$ such that
\begin{equation}\label{eq:eps-delta}
|\phi_A([u])-\phi_A([v])|< \frac{\eps}{2}
\end{equation}
for all $A\in \SL(2,\R)$ with $\|A\|\le M$ and all $u,v$ with the
angle between the corresponding lines less than~$\delta$. At the same time,
\begin{equation}\label{eq:Am-v}
\log |v_m| = \log |u_m| + \sum_{j=1}^m \left(\varphi_{A_j}([v_{j-1}]) - \varphi_{A_j}([u_{j-1}]) \right).
\end{equation}
The first summand is within~$2n\eps'$ from $m\lambda$ due to~\eqref{eq:u-lower} and the assumption on $(n\varepsilon', \lambda)$-hyperbolicity. The sum in the second summand can be decomposed into two parts: where the angle between~$u_{m-1}$ and~$v_{m-1}$ is greater than~$\delta$ and where it is smaller than~$\delta$. The summands of the second type give the contribution of at most $m\frac{\eps}{2}$ due to~\eqref{eq:eps-delta}, while there will be at most
$$
\frac{2n\eps'+ \log \frac{\pi}{2}+|\log \delta|}{\lambda} \le \frac{3\eps'}{\lambda_F} n
$$
summands of the first one (assuming $n$ to be sufficiently large), giving their total contribution of
at most $2\log M \cdot \frac{3\eps'}{\lambda_F} n$. Adding up, we get an estimate
$$
|\log |v_m| - m\lambda | \le 2n\eps' + m \frac{\eps}{2} + 2\log M \cdot \frac{3\eps'}{\lambda} n \le
\left(\left(2 +  \frac{6 \log M}{\lambda} \right) \eps' + \frac{\eps}{2} \right) n.
$$
Fix $\eps' = \left(2 +  \frac{6 \log M}{\lambda} \right)^{-1} \cdot \frac{\eps}{2}$, and we get the desired
$$
|\log |v_m| - m\lambda | \le \eps n.
$$
This completes the proof of Lemma \ref{l:line-shape}.
\end{proof}
\begin{remark}
In fact, the proof of Lemma~\ref{l:line-shape} uses only the exponential growth of the product of lengths $|v_m|\cdot |u_m|$. Hence,
the assumption of $v_0$ being the shortest vector of a sequence $\{v_m\}$ can be weakened to a mere lower bound on the allowed
exponential decrease speed. Namely, it suffices to assume that for some  $c>0$, $\lambda'<\lambda$ we have
$|v_m|\ge c e^{-m\lambda'} |v_0|$ all  $m=1,\dots, n$ for the conclusion of Lemma~\ref{l:line-shape} to hold for all sufficiently large~$n$.
\end{remark}

The next lemma allows to get rid of the assumption of $v_0$ being the shortest vector in the sequence of iterations.

\begin{lemma}[V-shape]\label{l:V-shape}
For any $M, \lambda, \eps>0$ there exists $\eps'>0$ with the following property. Assume that
the product $A_n\dots A_1$ is $(n\eps',\lambda)$-hyperbolic, and $v_m$ be a
sequence of intermediate images associated to some $v_0\in\R^2\setminus \{0\}$ given by (\ref{eq:v-def}).
Then there exists $m'\in \{0,\dots,n\}$, such that
$$
\forall m=0,1,\dots, n \quad \log |v_m| -\log |v_{m'}| \in U_{n\eps}(\lambda\cdot |m-m'|).
$$
\end{lemma}
Again, this lemma admits a geometric interpretation in terms of the graph of $\log |v_m|$: plotting this graph for $m=0,1,\dots,n$,
and then contracting it~$n$ times in both directions, we get a graph that is contained in the $\eps$-neighborhood of a~$V$-shaped piecewise-linear function with slopes $\pm\lambda$ (see Fig.~\ref{f:norms}).
\begin{proof}
We will choose $m'$ so that $v_{m'}$ is a least-norm vector in this sequence:
$$
|v_{m'}|=\min_{0\le m\le n} |v_m|.
$$
Now apply previous Lemma \ref{l:line-shape} separately on the intervals 
$[0,m']$ and $[m',n]$.
\end{proof}

Lemma \ref{l:V-shape} allows us to prove the first part of Theorem~\ref{t.vector}, on one-sided products.

\begin{proof}[Proof of the first part of Theorem~\ref{t.vector}]
If~\eqref{eq:lim-less} holds, then for some $\eps_0>0$ one has
\begin{equation}\label{eq:lim-l-eps}
\limsup_{n\to+\infty} \frac{1}{n} \log \|T_{n,a, \bar \omega} \left( \begin{smallmatrix} 1 \\ 0 \end{smallmatrix} \right)\| < \lambda_F(a)-\eps_0.
\end{equation}
Due to the standard argument of a countable intersection (considering a sequence of positive values of  $\eps_0$ that tends to zero)
it suffices to show that the conclusion of the theorem holds with~\eqref{eq:lim-less}
replaced with~\eqref{eq:lim-l-eps}. From now on, fix small $\eps_0>0$.

Take the point $x_0$ on the circle to be the projectivization
image of the vector~$v_0:=\left(1 \atop 0\right)$.
Note that the series
$$
\sum_n \exp(-\delta_0 \sqrt[4]{n})
$$
converges for any $\delta_0>0$. Hence, due to Borel--Cantelli lemma, for any $\eps,\eps'>0$ almost surely
for all sufficiently large~$n$ the conclusions of Theorem~\ref{t:main} and of Proposition~\ref{p:derivatives-control}
(for this specific choice of the point~$x_0$) hold.

We will fix sufficiently small values of $\eps$ and $\eps'$ for the arguments below to
work; in fact, as the reader will see, it suffices to take an arbitrary
\begin{equation}\label{e.choice}
\eps<\frac{\eps_0}{20} \text{\ \ \ and \ \  }
\eps':=\frac{1}{10 C_1}\eps,
\end{equation}
where $C_1$ is given by Proposition \ref{p:derivatives-control}.

Assume now that for some $a\in J$ the inequality~\eqref{eq:lim-l-eps} holds; it also implies that for all sufficiently large $m$
\begin{equation}\label{eq:n-eps0}
\frac{1}{m} \log \|T_{m,a, \bar \omega} \left( \begin{smallmatrix} 1 \\ 0 \end{smallmatrix} \right)\| < \lambda_F(a)-\eps_0.
\end{equation}
Let $n_1$ be such that conclusions of both Theorem~\ref{t:main} and of Proposition~\ref{p:derivatives-control}, as well as~\eqref{eq:n-eps0},
hold for all $m,n>n_1$.


For any $n>n_1$ consider the interval $J_{i}$ that contains~$a$. Note that for all sufficiently
large $n$ it is one of the exceptional intervals in the sense of Theorem~\ref{t:main}, in other words,
it cannot be neither small nor opinion-changing in terms of Proposition~\ref{p:classes}.
Indeed, otherwise Proposition~\ref{p:derivatives-control} would \edited{imply} the derivatives control~\eqref{eq:derivative-u},
and thus the derivatives at~$x_0$ would satisfy the exponential contraction with almost fastest possible speed:
$$
\log f'_{m,a,\bar\omega}(x_0) \le \lR(a)\cdot m + C_1\eps' n.
$$
Recalling the relation~\eqref{eq:der-norm} between the derivative and the norm change, we thus would get an almost fastest possible expansion:
$$
\log |T_{m,a,\bar\omega}(v_0)| \ge -\frac{1}{2}(\lR(a)\cdot m + C_1\eps' n) = \lambda_F(a)\cdot m -\frac{1}{2} C_1\eps' n.
$$
However, once $\frac{C_1\eps'}{2}<\eps_0$, we would get a contradiction with~\eqref{eq:n-eps0} at $m=n$.
Hence, $J_i$ should be an exceptional interval.

Moreover, the same arguments imply that for all sufficiently large $n$ the \edited{index} 
 $m_0'$ \edited{defined in (\ref{eq:derivative-upl})}, associated
to this $n$, satisfies $m_0'<\frac{1}{10} n$. Indeed, otherwise from~\eqref{eq:derivative-upl} for $m=m_0'$ we would get
$$
\frac{1}{m}\log | T_{m,a,\bo}(v_0)  |  \ge \frac{1}{m}(\lambda_F(a)\cdot m - \frac{C_1 \eps'}{2} n) \ge \lambda_F(a) - 5C_1\eps' >  \lambda_F(a) - \eps_0,
$$
thus again obtaining a contradiction with~\eqref{eq:n-eps0}.

Finally, the product $T_{[m_0';n],a,\bo}$ is also $(2n\eps',\lambda_F)$-hyperbolic. Thus, we can apply to it Lemma~\ref{l:V-shape}, obtaining
from the conclusion of this lemma the corresponding $m'\in [m_0', n]$.

Note now that the above arguments can be applied for all $n>n_1$, so for each such~$n$ we get the corresponding exceptional interval $J_{i_n,(n)}$,
the corresponding $m_{0,(n)}'$ and the moment $m'_{(n)}\in [m_{0,(n)}', n]$ obtained by the application of Lemma~\ref{l:V-shape}.

We then have the following auxiliary
\begin{lemma}\label{l:more-than-half}
$m'_{(n)}\ge \frac{1}{2}n$ for any $n>n_1$.
\end{lemma}
\begin{proof}
It is easy to show that if the statement of Lemma \ref{l:more-than-half} does not hold for some $n$, it does not hold also for $2n$.
Indeed, assume  $m'_{(n)}<\frac{1}{2}n$. Then due to Lemma~\ref{l:V-shape} we have
\begin{equation}\label{eq:growth}
\log |v_n| - \log |v_{n/2}| \ge \lambda_F \frac{n}{2} - 4 n \eps'.
\end{equation}
We already know that $m'_{0,(2n)} \le \frac{2n}{20} < \frac{n}{2}$, so having $m'_{(2n)}\ge \frac{2n}{2}=n$ would imply that from $n/2$ to $n$ we are on the ``decreasing'' branch
of the $V$-shaped graph for log-length, and thus
$$
\log |v_n| - \log |v_{n/2}| \le -\lambda_F \frac{n}{2} + 8 n \eps'.
$$
This would contradict~\eqref{eq:growth} as $\eps'< \frac{\lambda_F}{24}$. Hence, assuming $m'_{(n)}< \frac{1}{2}n$ we also get $m'_{(2n)}< \frac{1}{2}\cdot 2n$, and by induction $m'_{(2^k n)}< \frac{1}{2}\cdot 2^k n$ for all $k$. Note now that~\eqref{eq:growth} can be rewritten as a lower bound for the slope
$$
\frac{\log |v_n| - \log |v_{n/2}|}{n/2} \ge \lambda_F - 8 \eps'.
$$
Joining such inequalities for $n$, $2n$, $4n$, etc., we get
$$
\limsup_{k\to\infty} \frac{1}{2^k n} \log |v_{2^k n}| \ge  \lambda_F - 8 \eps',
$$
and we thus have a contradiction with~\eqref{eq:n-eps0}, as $8\eps'<\eps_0$. This completes the proof of Lemma~\ref{l:more-than-half}.
\end{proof}

Now, the inequality~$m'_{(n)}\ge \frac{1}{2}n$ implies that
\begin{equation}\label{eq:decrease}
\log |v_{n/2}| - \log |v_{n/4}| \le -\lambda_F \frac{n}{4} + 4 n \eps',
\end{equation}
or in terms of a slope,
$$
\frac{\log |v_{n/2}| - \log |v_{n/4}|}{n/4} \le -\lambda_F + 16  \eps'.
$$
Joining such inequalities for $n$, $n/2$, $n/4$, etc., until we hit $n_1$, we get the desired
$$
\limsup_{n\to\infty} \frac{\log |v_n|}{n} \le -\lambda_F + 16  \eps'.
$$
Finally, as $\eps'$ can be chosen arbitrarily small, we finally get
$$
\limsup_{n\to\infty} \frac{\log |v_n|}{n} \le -\lambda_F
$$
and hence, due to Proposition \ref{p.upper},
$$
\lim_{n\to\infty} \frac{\log |v_n|}{n} = -\lambda_F.
$$
This completes the proof of the first part of Theorem \ref{t.vector}.
\end{proof}

%
%

\begin{remark}
If the initial vector was not fixed, the statement of the one-sided version of Theorem~\ref{t.vector} would not hold. Moreover, almost surely
there exists a \emph{residual} set of parameters~$a\in J$, for each of which there exists a nonzero vector~$v_0$ such that for the norms of its images $v_n$ one has
$$
\limsup \frac{1}{n} \log |v_n| =0.
$$
\end{remark}

Now, the conclusions~\ref{i:m2} and~\ref{i:m3} of Theorem~\ref{t:main} together
imply that for any $a\in J$ the product $T_{n,a,\omega}$ either is $(n\eps,\lambda_F)$-hyperbolic
itself, or can be divided into two hyperbolic products. Thus, under the conclusions of Theorem~\ref{t:main} we have
\begin{lemma}[W-shape]\label{l:W-shape}
For any $\eps>0$ there exists $\eps'>0$ with the following property. Assume that the conclusions of Theorem~\ref{t:main} with the given $\eps'$
are satisfied for some finite product $F_a(\omega_n) \dots F_a(\omega_{1})$. Then for any sequence $\bv_m$ of nonzero vectors such that
$\bv_{m}=F_a(\omega_m)(\bv_{m-1})$, there exists a continuous piecewise-linear
function $\varphi(\cdot)$ with slopes~$\pm \lambda_F$ and at most one ``upwards'' break point, such that
$$
\forall m=0,1,\dots, n \quad \log |v_m| \in U_{n\eps}(\varphi(m)).
$$
\end{lemma}
As earlier, this lemma can be seen geometrically in terms of the corresponding graphs~(see Fig.~\ref{f:norms}).


Let us now conclude the proof of Theorem~\ref{t.vector}.

\begin{proof}[Proof of the second part of Theorem~\ref{t.vector}]
Let $v_n:=T_{n,a,\bo}(v)$ for all $n$. Without loss of generality, we can assume that $|v_0|=1$.
As in the proof of the first part, it suffices to show that
\begin{equation}\label{eq:two-eps0}
\limsup_{n\to\pm \infty} \frac{1}{|n|} \log |v_n | < \lambda_F(a)-\eps_0
\end{equation}
in fact forces
$$
\limsup_{n\to\pm \infty} \frac{1}{|n|} \log |v_n | =- \lambda_F(a).
$$
As before,~\eqref{eq:two-eps0} implies that for all sufficiently large $n$ we have
\begin{equation}\label{eq:two-finite-n}
\frac{1}{n} \log |v_n | < \lambda_F(a)-\eps_0, \quad \frac{1}{n} \log |v_{-n} | < \lambda_F(a)-\eps_0.
\end{equation}

Also as before, we can assume that for any $\eps,\eps'>0$ for all $n$ sufficiently large
the conclusions of Theorem~\ref{t:main} hold for the product
$$
T_{[-n;n],a,\bo}= F_a(\omega_n) \dots F_a(\omega_{-n}),
$$
and hence Lemma~\ref{l:W-shape} can be applied. We will take $\eps$ and $\eps'$ as in 
(\ref{e.choice}), and
let $n_2$ be such that the mentioned above statements hold for all $n>n_2$.

From now on, for any $n>n_2$ let $m'_{-,(n)}<m'_{0,(n)}<m'_{+,(n)}$ be the breakpoints of the function $\varphi_{(n)}$
given for it by Lemma~\ref{l:W-shape}, the central one being the upwards break point. 

Note first that one has $m'_{-,(n)}<0<m'_{+,(n)}$. Indeed, if one had $m'_{+,(n)}\le 0$, this would imply
that $\varphi_{(n)}$ is linear on $[0,n]$, and thus $\varphi_{(n)}(n)-\varphi_{(n)} (0)=n\lambda_F$. On the other hand,
\begin{equation}\label{eq:bv-diff-norm}
\log |v_n|-\log |v_0| \ge (\varphi_{(n)}(n)-\varphi_{(n)} (0)) -2 \eps n
\end{equation}
and thus we would get
$$
\log |v_n|-\log |v_0| \ge (\lambda_F-2\eps) n,
$$
and this  would contradict~\eqref{eq:two-finite-n} as $2\eps< \eps_0$. In the same way we get $m'_{-,(n)}< 0$.

Now, in the same way as in the first part, we are going to prove that
\begin{equation}\label{eq:2-break}
|m_{0,(n)}'| < \frac{1}{10} n.
\end{equation}
Indeed, 
we have
$$
\log |v_{m_{0,(n)}'}| \ge (\varphi_{(n)}(m_{0,(n)}')-\varphi_{(n)} (0)) - 2n\eps =  \lambda_F |m_{0,(n)}'| - 2n\eps,
$$
and from~\eqref{eq:two-finite-n} we know that
$$
\log |v_{m_{0,(n)}'}| \le (\lambda_F-\eps_0)\cdot |m_{0,(n)}'|,
$$
Hence,
$$
\eps_0\cdot |m_{0,(n)}'| \le 2n\eps,
$$
and thus
$$
|m_{0,(n)}'| \le \frac{2\eps}{\eps_0} n <\frac{1}{10}n.
$$

Now, in the same way as in the first part, we are going to prove the auxiliary
\begin{lemma}\label{l:2-half}
$m'_{+,(n)}, |m'_{-,(n)}| \ge \frac{1}{2}n$ for any $n>n_2$.
\end{lemma}
\begin{proof}
We will prove the conclusion for $m'_{+,(n)}$, the statement for $m'_{-,(n)}$ is absolutely analogous.
The proof goes in the same way as in Lemma~\ref{l:more-than-half}. Namely, we first note that if its conclusion does
not hold for some $n>n_2$, it does not hold for $2n$ neither. Indeed, if we had $m'_{+,(n)}\le \frac{n}{2}$, then we
would have
\begin{equation}\label{eq:2-growth}
\log |v_{n}| - \log |v_{n/2}| \ge \lambda_F \frac{n}{2} - 2\eps n.
\end{equation}
Then, we have $m'_{0,(2n)}\le \frac{n}{5}$, and if we had $m'_{+,(2n)}> n$, this would imply that $\varphi_{(2n)}$ is linear
on the interval $[\frac{n}{2},n]$, and hence
\begin{equation}\label{eq:2-decr}
\log |v_{n}| - \log |v_{n/2}| \le -\lambda_F \frac{n}{2} + 4\eps n.
\end{equation}
And as $\eps<\frac{\lambda_F}{3}$, the inequalities~\eqref{eq:2-growth} and~\eqref{eq:2-decr} contradict each other.

Thus, if the conclusion of Lemma \ref{l:2-half} did not hold for some $n>n_2$, it would also be wrong for $2n$, and
by induction $m'_{+,(2^k n)}< \frac{1}{2}\cdot 2^k n$ for all $k$. Note now that~\eqref{eq:2-growth} can be rewritten as a lower bound for the slope
$$
\frac{\log |v_n| - \log |v_{n/2}|}{n/2} \ge \lambda_F - 4 \eps.
$$
Joining such estimates for $2^k n$, we get
$$
\limsup_{k\to\infty} \frac{1}{2^k n} \log |v_{2^k n}| \ge  \lambda_F - 4 \eps,
$$
thus obtaining a contradiction with~\eqref{eq:two-eps0}. This contradiction proves Lemma \ref{l:2-half}. 
\end{proof}

Let us now conclude the proof of the second part of Theorem~\ref{t.vector}. Lemma~\ref{l:2-half} together with~\eqref{eq:2-break} imply that the
function $\varphi_{(n)}$ is linear on $[\frac{n}{4},\frac{n}{2}]$ and hence that
\begin{equation}\label{eq:2-decrease}
\log |v_{n/2}| - \log |v_{n/4}| \le -\lambda_F \frac{n}{4} + 2 n \eps;
\end{equation}
in terms of a slope, it means that
$$
\frac{\log |v_{n/2}| - \log |v_{n/4}|}{n/4} \le -\lambda_F + 8  \eps.
$$
Joining such inequalities for $n$, $n/2$, $n/4$, etc., until we hit $n_2$, we get the desired
$$
\limsup_{n\to\infty} \frac{\log |v_n|}{n} \le -\lambda_F + 8  \eps.
$$

As $\eps>0$ can be chosen arbitrarily small, we finally get
$$
\limsup_{n\to\infty} \frac{\log |v_n|}{n} \le -\lambda_F
$$
and hence, due to Proposition \ref{p.upper},
$$
\lim_{n\to\infty} \frac{\log |v_n|}{n} = -\lambda_F.
$$

The asymptotics at $-\infty$ can be handled in the same way. This completes the proof of Theorem \ref{t.vector}.
\end{proof}


\begin{appendix}

\section{Generalized Johnson's Theorem}\label{s.johnson}

Suppose that $\frak{M}$ is a compact metric space, $\sigma:\frak{M}\to \frak{M}$ is a homeomorphism, and $\frak{m}$ is an ergodic invariant Borel probability measure supported on $\frak{M}$. Assume also that we are given a continuous map $g_{\cdot}:\frak{M}\to Homeo^+(S^1)$. Then, one can consider an associated skew product
$$
F:(\omega, x)\mapsto (\sigma \omega, g_\omega (x)).
$$

Next, let us choose for any $\omega\in\frak{M}$ a lift $\tilde g_\omega:\mathbb{R}\to \mathbb{R}$ of the map $g_\omega\in Homeo^+(S^1)$,
$$
g_\omega(\pi(x))=\pi(\tilde g_\omega(x)),
$$
where $\pi:\mathbb{R}^1\to S^1=\mathbb{R} / \mathbb{Z}$ is a natural covering map, in such a way that $\{\tilde g_\omega(0)\}$ is a bounded measurable (in $\omega$) function (e.g. one can require $g_\omega(0)\in [0,1)$ for all $\omega\in \frak{M}$). We then can consider the associated lift of the skew product:
$$
\tilde F:(\omega, x)\mapsto (\sigma \omega, \tilde g_\omega (x)).
$$

Finally, let $G_{m,\omega}$ and $\tilde G_{m,\omega}$ be the length $m$ fiberwise compositions associated to these skew products:
$$
F^m (\omega,x)=(\sigma^m \omega, G_{m,\omega}(x)), \quad \tilde F^m (\omega,x)=(\sigma^m \omega, \tilde G_{m,\omega}(x)),
$$
so that for $m>0$ we have
$$
\tilde G_{m,\omega}=\tilde g_{\sigma^{n-1}\omega}\circ \ldots \circ \tilde g_{\sigma\omega}\circ\tilde g_{\omega}.
$$

Then, we have the following
\begin{prop}\label{p.rotnum}
In this setting above the following statement holds.
There exists a number $\rho\in \mathbb{R}$ such that for $\frak{m}$-a.e. $\omega\in \frak{M}$ and every $x\in \mathbb{R}$ the limit
\begin{equation}\label{eq:tG}
\lim_{n\to \infty}\frac{1}{n} (\tilde G_{n,\omega} (x)-x)
\end{equation}
exists and is equal to $\rho$.
\end{prop}

\begin{defi}
 The number $\rho$ from Proposition \ref{p.rotnum} is called \emph{rotation number}.
\end{defi}

\begin{remark}
Notice that the rotation number $\rho$ depends on the choice of lifts $\tilde g_\omega$.
\end{remark}

\begin{remark}
 It can happen that the lifts $\{\tilde g_\omega\}$ cannot be taken continuous in~$\omega$. At the same time, in the case when $\{g_\omega\}$ are projectivizations of the transfer matrices of a Schr\"odinger cocycle defined by a continuous potential, the lifts $\{\tilde g_\omega\}$ can always be chosen continuously in $\omega$ (since any Schr\"odinger cocycle is homotopic to a constant one).
 \end{remark}

\begin{remark}
 Some of the assumptions in Proposition \ref{p.rotnum} can be essentially relaxed. For example, one can start with a probability space $(\frak{M}, \frak{m})$ and a measure preserving transformation $\sigma$ instead on a measure preserving homeomorphism of a compact metric space, or relax the assumption on continuity of~$g_\omega$. To keep the presentation more transparent, we are not trying to give the statements in the most general form.
\end{remark}

\begin{remark}
 While the case that we consider in this paper in a sense corresponds to the case of linear cocycle (i.e. the maps $g_{\omega}$ are projective maps of the circle), in Proposition \ref{p.rotnum} the cocycle is non-linear (i.e. we allow arbitrary homeomorphisms of the circle, not necessarily projective). Notice that in fact many of the questions and results that we consider here can also be posted for non-linear case as well. For example, if one reformulates the Furstenberg Theorem as a statement on almost sure exponential convergence of vectors in projective space under random projective dynamics, then  non-linear analogs of Furstenberg Theorem are known \cite{A, Bax, DKN1, KN, GGKV, M}.
\end{remark}
Proposition \ref{p.rotnum} is certainly well known \edited{(see \cite[Section 5]{Her} and \cite{R2} for similar statements)}, 
but we provide the proof here for the convenience of a reader.
\begin{proof}[Proof of Proposition \ref{p.rotnum}]
Define the displacement function $\varphi:\frak{M}\times \mathbb{R}\to \mathbb{R}$ by
$$
\varphi(\omega, x)=\tilde g_\omega(x)-x.
$$

Then, the displacement under $n$ iterations in~\eqref{eq:tG} can be rewritten as a sum of~$n$~individual displacements:
\begin{equation}\label{eq:tG-phi}
    \tilde G_{n,\omega}(x) -x= \sum_{k=0}^{n-1} (\tilde G_{k+1,\omega}(x) -\tilde G_{k,\omega}(x)) = \sum_{k=0}^{n-1}  \varphi(\tilde F^k (\omega, x)).
\end{equation}
Moreover, note that the function $\varphi(\omega,x)$ is in fact $1$-periodic in the $x$ variable, and hence as a function of $x$ can be considered as a function on the circle. Indeed, if $y=x+k$, $k\in \mathbb{Z}$, then
$$
\varphi(\omega, y)=\tilde g_\omega(x+k)-(x+k)=\tilde g_\omega(x)-x =\varphi(\omega, x).
$$
Hence, a function $\psi:\frak{M}\times S^1\to \mathbb{R}$, $\psi(\omega, t)=\varphi(\omega, \pi^{-1}(t))$, is well defined,
and the sum in~\eqref{eq:tG-phi} can be written as
$$
\sum_{k=0}^{n-1}  \varphi(\tilde F^k (\omega, x)) =  \sum_{k=0}^{n-1}  \psi(F^k (\omega, x)).
$$
Thus,
\begin{equation}\label{eq:psi-F}
\frac{1}{n}(\tilde G_{n,\omega}(x) -x) = \frac{1}{n}  \sum_{k=0}^{n-1}  \psi(F^k (\omega, x))
\end{equation}
is a time-average of a bounded function $\psi$ on a compact space~$\frak{M}\times \Sc$.

Now, Krylov-Bogolyubov arguments imply that the map $F(\omega, t)=(\sigma \omega, g_\omega(t))$
has an invariant measure $\eta$ such that the projection of $\eta$ to the first coordinate of the product
$\frak{M}\times S^1$ gives the measure $\frak{m}$.
Birkhoff Ergodic Theorem then implies the existence of the limit~\eqref{eq:psi-F} for $\eta$-a.e. point
$(\omega, x)\in \frak{M}\times S^1$.

Finally, note that $\tilde G_{n,\omega}$ is the lift of $G_{n,\omega}$, and hence for any $x,y\in \R$ one has
$$
\left| (\tilde G_{n,\omega}(x)-x) - (\tilde G_{n,\omega}(y)-y) \right| <1.
$$
Hence, if the limit~\eqref{eq:psi-F} exists for some point $(\omega,x)$, it also exists
and takes the same value for any other point $(\omega,y)$ on the same fiber.
This limit thus defines a function $\rho(\omega)$ on $\frak{M}$.
Finally, as this function $\sigma$-invariant, and the measure $\frak{m}$ is ergodic,
this function is $\frak{m}$-almost everywhere equal to some constant~$\rho$.
\end{proof}

Let us now consider the dependence of the rotation number on a parameter. Namely, assume now that we are given a continuous
family $g_{\cdot,\cdot}: \frak{M}\times J \to Homeo_+(\Sc)$ of maps as above. Then, we can consider their lifts $\tg_{a,\omega}:\R\to\R$ to be chosen continuously in parameter $a\in J$. The corresponding skew products $G_a$ and $\tilde G_a$ as well as the fiberwise compositions
$F_{n,a,\omega}$ and $\tilde F_{n,a,\omega}$ then can be defined in the same way as before. The notion of monotonicity can then be applied in this situation, too.
\begin{defi}
The family is \emph{monotonous} if for any $\omega\in \frak{M}, x\in\Sc$ the function $\tg_{a,\omega}(x)$ is monotonous increasing in $a\in J$.
\end{defi}

An important note is  that the \emph{increments} of the images $\tG_{n,a',\omega}(x)-\tG_{n,a,\omega}(x)$ do not depend on a particular choice of lifts $\tg_{a,\omega}$. Moreover, this increment is a \emph{continuous} in $\omega$ and $x$ (and in fact is a well-defined function of the point~$x$ on the circle, not on the real line). Also, dividing by $n$ and passing to the limit, one gets that the difference of the corresponding rotation numbers $\rho(a')-\rho(a)$ does not depend on the choice of lifts~$\tg_{a,\omega}$, thus getting the following important note.

\begin{remark}
Even though the rotation number $\rho$ depends on a particular choice of the lifts $\tg_{a,\omega}$, the differences of rotation numbers $\rho(a')-\rho(a)$ do not. In particular, different choice of lifts $\tilde g_{a, \omega}$ leads to a shift of the rotation number $\rho(a)$ by a constant, and intervals of constancy of $\rho$ are independent of the choice of the lifts.
\end{remark}

The following result is known in many particular cases, e.g. see \cite[Theorem~4.8]{GJ}, \cite{Le}. For example, the ergodic Schr\"odinger cocycles satisfy the assumptions of Theorem~\ref{t.gjt}; the corresponding statement in the context of Schr\"odinger cocycles is known as {\it Johnson's Theorem}, see \cite{J86}. Generalizations to the cases of Jacobi matrices \cite{Ma} and CMV matrices \cite{DFLY} are also available. Monotone $SL(2, \mathbb{R})$ cocycles homotopic to a constant were treated in \cite[Proposition C.1]{ABD}. For the convenience of a reader we provide here the proof of the statement that is just slightly more general, but covers many of those cases.

\begin{theorem}\label{t.gjt}
Suppose that a family of $SL(2, \mathbb{R})$ cocycles is given by a continuous map
$$
A:\frak{M}\times J\to SL(2, \mathbb{R}),
$$
where $J\subset \mathbb{R}$ is an interval of parameters, and $A_a=A(\cdot, a):\frak{M}\to SL(2, \mathbb{R})$ is a cocycle corresponding to the parameter $a\in J$.

Assume that  for each $\omega\in \frak{M}$ and any vector $v\in \mathbb{R}^2\backslash \{0\}$, $\mathrm{arg}\, A_a(\omega)v$ as a function of the parameter $a$ is strictly  increasing.

Let $g_{a, \omega}:S^1\to S^1$ be a projective map induced by $A_a(\omega):\mathbb{R}^2\to \mathbb{R}^2$, and choose a family of lifts $\tilde g_{a, \omega}:\mathbb{R}^1\to \mathbb{R}^1$ as in Proposition \ref{p.rotnum} that depend continuously on the parameter $a$ for each $\omega$. Let $\rho(a)$ be the corresponding rotation number. Then $\rho$ is constant on an open interval $U\subset J$ if and only if the cocycle $A_a$ is uniformly hyperbolic for all $a\in U$.
\end{theorem}

The first step in the proof of this theorem does not require the cocycle to be projective:
\begin{lemma}\label{l:incr-2}
Let $\tg_{a,\omega}$ be a monotonous family as above, and assume that for some $m, a,a',\tx,\bo$ one has
\begin{equation}\label{eq:more-2}
\tG_{m,a',\bo}(\tx)-\tG_{m,a,\bo}(\tx)>2.
\end{equation}
Then $\rho(a')>\rho(a)$.
\end{lemma}

Note that a lower bound of an increment by~$1$ in the assumptions of Lemma~\ref{l:incr-2} would not suffice, even in the case of one circle homeomorphism.
Indeed, consider a very strong North-South map $g$, and a family of its perturbations $g_{\eps}:=R_{\eps} \circ g \circ R_{\eps}$. Then, on the one hand, the rotation number vanishes in a neighborhood of $\eps=0$. On the other hand, the images of the repelling fixed point $x$ can gain more than a full turn in such a neighborhood: see Fig.~\ref{f:circle-eps}. In fact, Proposition~\ref{l:f-unstable} below shows that this example is quite instructive.
\begin{figure}[!h!]
\begin{center}
\includegraphics{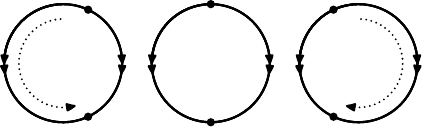}
\end{center}
\caption{Behavior of the family $R_{\eps} \circ g \circ R_{\eps}$ with a strong North-South map $g$; dashed arrows indicate the image
of the repeller $x$ of $g$ for positive (left) and negative (right) values of $\eps$.}\label{f:circle-eps}
\end{figure}

\begin{proof}[Proof of Lemma~\ref{l:incr-2}]
It suffices to consider the case $m=1$, otherwise passing to the $m$-th iteration of the initial system. Now,
as $\tg_{a,\omega}=\tG_{1,a,\omega}$ is a lift of a circle homeomorphism, inequality~\eqref{eq:more-2} implies that
$$
\forall y \quad \tg_{a',\bo}(y)-\tg_{a,\bo}(y)>1.
$$

The continuity of the increment in $\omega$ thus implies, that there exists a neighborhood $V\ni \bo$ such that
\begin{equation}\label{eq:incr-1}
\forall \omega\in V, \quad \forall y \quad \tg_{a',\omega}(y)-\tg_{a,\omega}(y)>1.
\end{equation}

Finally, for a generic~$\omega^*\in \frak{M}$, its orbit $\{\sigma^n \omega^*\}$ visits $V$ with an asymptotic frequency of~$\frak{m}(U)>0$.
On the other hand, if during $n$ iterations the orbit has visited $U$ at $k$ moments $n_1<\dots<n_k$, then it is easy to see from~\eqref{eq:incr-1} that
$$
\forall y \quad \tG_{n,a',\omega^*}(y)-\tG_{n,a,\omega^*}(y)>k.
$$
Taking a generic $\omega^*$, dividing by $n$ and passing to the limit, we get
$$
\rho(a') - \rho(a) \ge \frak{m}(V)>0.
$$
\end{proof}

Let us now pass to the proof of Theorem~\ref{t.gjt}: the arguments below will start using the projective nature of the cocycle. Denote by
$A_{m,a,\omega}$ the corresponding fiberwise composition: let
$$
{\bf A}_{a, m, \omega}=\left\{
  \begin{array}{ll}
    A_a({\sigma^{m-1}\omega})\cdot \ldots \cdot A_a({\omega}), & \hbox{if $m>0$;} \\
    Id_{\mathbb{R}^2}, & \hbox{if $m=0$;} \\
    A_a({\sigma^{m}\omega})^{-1}\cdot \ldots \cdot A_a(\sigma^{-1}{\omega})^{-1}, & \hbox{if $m<0$,}
  \end{array}
\right.
$$
so that $G_{m,a,\omega}$ is the projectivization of $A_{m,a,\omega}$, and $\tG_{m,a,\omega}$ is the corresponding lift.

\begin{proof}[Proof of Theorem~\ref{t.gjt}]
If the cocycle is uniformly hyperbolic for some parameter $a\in J$, then the cone condition holds for all parameters from some neighborhood $U$ of $a$  (with stable/unstable cones independent of parameter). Therefore, for all values $a'\in U$ for any $\omega\in \frak{M}$ and any $x\in S^1$ that corresponds to a vector from an unstable cone, the values of $\tG_{n,a',\omega}$ will remain on bounded distance from $\tG_{n,a,\omega}$ for all $n\in \mathbb{N}$. Hence, $\rho(a)$ is locally constant for uniformly hyperbolic cocycles.

Now assume that the cocycle is not uniformly hyperbolic for some value of the parameter. Without loss of generality we can set this value of the parameter to $0\in J$. We need to show that the rotation number $\rho(a)$ cannot be constant in any interval containing $0$.

Theorem \ref{t.gjt} certainly holds if $\frak{M}$ consists of just one periodic orbit of $\sigma$. Therefore we assume that this is not the case.

It is known that a cocycle is not uniformly hyperbolic if and only if there exists a Sacker-Sell solution, i.e. for some $\omega\in \frak{M}$, some $K>0$, and some unit vector $v\in \mathbb{R}^2$ we have
\begin{equation}\label{e.sssolution}
|A_0({\sigma^n\omega})\cdot \ldots \cdot A_0({\omega})v | \le K, \text{\ \ \ and\ \ \ } |A_0^{-1}({\sigma^{-n}\omega})\cdot \ldots \cdot A_0^{-1}({\sigma^{-1}\omega})v | \le K
\end{equation}
for all $n\in \mathbb{N}$, e.g. see \cite[Theorem 1.2]{DFLY}.

We will need the following statement.

\begin{prop}\label{l:f-unstable}
In the setting of Theorem~\ref{t.gjt}, let $\omega, v$ be such that the forward iterations of the vector $v$, associated to $\omega$, are bounded:
$$
\exists K : \quad \forall n\in \N  \quad | A_{n,0,\omega} v| \le K.
$$
Then for an arbitrary small $\alpha>0$ there exists $n>0$ such that
$$
\tG_{n,\alpha,\omega}(x_v)-\tG_{n,-\alpha,\omega}(x_v) >1,
$$
where $x_v\in\R$ is one of the lifted points associated to the direction of the vector~$v$.
\end{prop}

Postponing for the moment the proof of this lemma, let us see that it implies Theorem~\ref{t.gjt}. Indeed, due to~\eqref{e.sssolution}
it can be applied to both forward and backward iterations of the vector $v$. Thus, for an arbitrary $\alpha>0$ there exist $n,n'>0$ such that
$$
\tG_{n,\alpha,\omega}(x_v)-\tG_{n,-\alpha,\omega}(x_v) >1,
$$
and
$$
\tG_{-n',-\alpha,\omega}(x_v)-\tG_{-n',\alpha,\omega}(x_v) >1.
$$
Take an arbitrary $y\in [\tG_{-n',\alpha,\omega}(x_v); \tG_{-n',-\alpha,\omega}(x_v)-1]$, and let $\bo:=\sigma^{-n'}\omega$. Then,
$$
\tG_{n',-\alpha,\bo}(y) <  \tx -1, \quad  \tG_{n',\alpha,\bo}(y)> x_v,
$$
hence
$$
\tG_{n+n',-\alpha,\bo}(y) <  \tG_{n,-\alpha,\omega}(x_v) -1, \quad  \tG_{n+n',\alpha,\bo}(y)> \tG_{n,\alpha,\omega}(x_v),
$$
and finally
$$
\tG_{n+n',\alpha,\bo}(y) - \tG_{n+n',-\alpha,\bo}(y) >\tG_{n,\alpha,\omega}(x_v) -(\tG_{n,-\alpha,\omega}(x_v) -1)>2.
$$
An application of Lemma~\ref{l:incr-2} concludes the proof.

\begin{proof}[Proof of Proposition~\ref{l:f-unstable}]

We will consider the following two cases separately:

\vspace{4pt}

{\bf Case 1.} There is a constant $C>0$ and a sequence $\{n_k\}$ of indices such that $n_k\to +\infty$ as $k\to +\infty$, and $\|{\bf A}_{0, n_k, \omega}\|\le C$.

\vspace{4pt}

{\bf Case 2.} We have $\|{\bf A}_{0, n, \omega}\|\to \infty$ as $n\to +\infty$.

\vspace{4pt}

Consider Case 1 first. Suppose $\|{\bf A}_{0, n_k, \omega}\|\le C$. Let us show (by induction in~$k\in \mathbb{N}$) that for any small $\alpha>0$ there is $\delta_1=\delta_1(\alpha, C)>0$ such that for any $x\in \mathbb{R}$
\begin{equation}\label{e.want}
\tilde G_{0, n_k, \omega}^{-1}\circ\tilde G_{\alpha, n_k, \omega}(x)-x\ge k\delta_1.
\end{equation}

Since $\tilde G_{a, m, \omega}$ are strictly increasing functions of the parameter $a$, by compactness arguments  for some $\varepsilon_1=\varepsilon_1(\alpha)>0$, any $\omega'\in \frak{M}$, and any $x\in \mathbb{R}$ we have
$$
\tilde g_{\alpha, \omega'}(x)-\tilde g_{0, \omega'}(x)\ge \varepsilon_1.
$$
and hence (considering the last iteration) for any $m>0$,
\begin{equation}\label{e.epsforce}
\tilde G_{\alpha, m, \omega'}(x)-\tilde G_{0, m, \omega'}(x)\ge \varepsilon_1.
\end{equation}
Since $\tilde G_{0, n_k, \omega}^{-1}$ is a projectivization of a matrix of a norm at most~$C$, it is a monotone function with derivative bounded away from zero by some constant that depends only on $C$. Thus we have for some $\delta_1=\delta_1(\alpha, C)$
\begin{equation}\label{e.force}
\text{\rm if}\ \ y_2-y_1\ge \varepsilon_1, \text{\rm \ \ then}\ \ \ \tilde G_{0, n_k, \omega}^{-1}(y_2)-\tilde G_{0, n_k, \omega}^{-1}(y_1)\ge \delta_1.
\end{equation}

In particular,
$$
\tilde G_{0, n_k, \omega}^{-1}\circ\tilde G_{\alpha, n_k, \omega}(x)- x =\tilde G_{0, n_k, \omega}^{-1}\circ\tilde G_{\alpha, n_k, \omega}(x)-\tilde G_{0, n_k, \omega}^{-1}\circ\tilde G_{0, n_k, \omega}(x)\ge \delta_1.
$$
Take and fix such $\delta_1$. Assume now that for some $k\in \mathbb{N}$ and for any $y\in \mathbb{R}$ we have
$$
\tilde G_{0, n_{k-1}, \omega}^{-1}\circ\tilde G_{\alpha, n_{k-1}, \omega}(y)-y\ge (k-1)\delta.
$$
Then using (\ref{e.epsforce}) and (\ref{e.force}) we have
\begin{multline*}
   \tilde G_{0, n_k, \omega}^{-1}\circ\tilde G_{\alpha, n_k, \omega}(x)-x = \\
   (\tG_{0, n_k, \omega}^{-1}\circ\tilde G_{\alpha, n_k, \omega}(x)-\tG_{0, n_{k-1}, \omega}^{-1}\circ\tilde G_{\alpha, n_{k-1}, \omega}(x)) + \\ +
   (\tG_{0, n_{k-1}, \omega}^{-1}\circ\tilde G_{\alpha, n_{k-1}, \omega}(x)-x). 
\end{multline*}
The second summand in the right hand side is no less than~$(k-1)\delta_1$ by the induction assumption. At the same time, the first one can be rewritten as
$$
\tG_{0, n_{k}, \omega}^{-1}(y_2) - \tG_{0, n_{k}, \omega}^{-1}(y_1),
$$
where
$$
y_1=\tilde G_{0, n_k-n_{k-1},\sigma^{n_{k-1}} \omega}(y_0), \quad y_2=\tilde G_{\alpha, n_k-n_{k-1},\sigma^{n_{k-1}} \omega}(y_0), \quad y_0=
\tilde G_{0, n_{k-1}, \omega}(x).
$$
and joining~\eqref{e.epsforce} with~\eqref{e.force} we see that it is greater than~$\delta_1$. We finally get
$$
   \tilde G_{0, n_k, \omega}^{-1}\circ\tilde G_{\alpha, n_k, \omega}(x)-x >\delta_1 + (k-1)\delta_1=k\delta_1.
$$
This completes the step of induction, and hence proves~(\ref{e.want}).

Now, taking $k> \frac{1}{\delta}$, we have $k\delta>1$, and thus~\eqref{e.want} implies that
\begin{multline*}
  \tilde G_{\alpha, n_k, \omega}(x)   = \tilde G_{0, n_k, \omega}\left(\tilde G_{0, n_k, \omega}^{-1}\circ \tilde G_{\alpha, n_k, \omega}(x)\right)  \ge \\
\tilde G_{0, n_k, \omega}(x+k\delta)>\tilde G_{0, n_k, \omega}(x+1)=\tilde G_{0, n_k, \omega}(x)+1,
\end{multline*}
proving the conclusion of the Proposition in this case.

\vspace{4pt}

Let us now consider Case 2. First, decreasing $\varepsilon_1=\varepsilon_1 (\alpha)$ if needed we can be sure that additionally to (\ref{e.epsforce}) we also have that for any $\omega'\in \frak{M}$, any $x\in \mathbb{R}$, and we have
$$
\tg_{0,\omega'}^{-1}(\tg_{\alpha,\omega'}(x)) > x + \varepsilon_1, \quad \tg_{\alpha,\omega'}(\tg_{0,\omega'}^{-1}(x)) > x + \varepsilon_1.
$$
Joining the two together (applying one for the first and one for the last iteration), for any $m\ge 2$, any $x$ and any $\omega'$ we get
\begin{equation}\label{eq:shifted}
\tG_{\alpha, m, \omega'}(x) \ge \tG_{0, m, \omega'}(x+\varepsilon_1)+ \varepsilon_1.
\end{equation}
In the same way (again, reducing the value of $\varepsilon_1$ if necessary) we get for all $x,\omega'$, and $m\ge 2$
$$
\tG_{-\alpha, m, \omega'}(x) \le \tG_{0, m, \omega'}(x-\varepsilon_1)- \varepsilon_1.
$$
%
%

Now, take $n$ such that the norm $\|A_{0,n,\omega}\|$ becomes sufficiently large (we will choose the lower bound later).
As we will see, the point $x_v$ is close to a lift of the point $\tx_-:=x_-(A_{0,n,\omega})\in \R$.
Let
$$
\tx_+ < \tG_{0,n,\omega} (\tx_-) <\tx_+ +1
$$
be the two lifts of the image of the most expanded direction.
We will show that, assuming appropriate lower bound for the norm $\|A_{0,n,\omega}\|$, we have
\begin{equation}\label{eq:overjump}
\tG_{\alpha,n,\omega} (x_v) > \tx_+ +1, \quad \tG_{-\alpha,n,\omega} (x_v)  < \tx.
\end{equation}
Together, these estimates will imply the desired $\tG_{\alpha,n,\omega} (x_v)-\tG_{-\alpha,n,\omega} (x_v) >1$.

\begin{figure}[!h!]
\begin{center}
\includegraphics{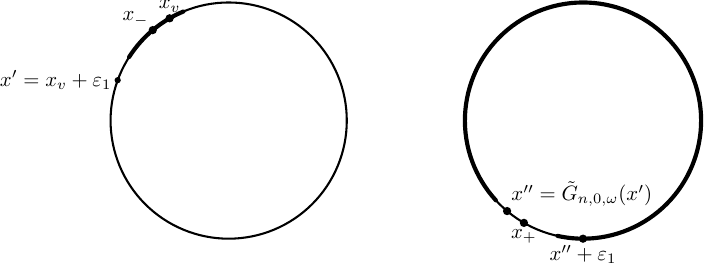}
\end{center}
\caption{Point $x_v$ on the circle and a lower bound $\tG_{0,n,\omega} (x_v+\eps_1)+\eps_1$ for its image~$\tG_{\alpha,n,\omega} (x_v)$.}\label{f:forward}
\end{figure}

Let us obtain the first of them. Indeed, due to Lemma~\ref{l:x-C-image} we have
$$
|x_v-\tx_-| \le \frac{\pi}{2} \frac{|A_{0,n,\omega} v|/|v|}{\|A_{0,n,\omega} \|} \le \frac{\pi K}{2 \|A_{0,n,\omega} \|}.
$$
In particular, provided that $\|A_{0,n,\omega} \|>\frac{\pi K}{\eps_1}$, we have $|x_v-\tx_-|\le \frac{\eps_1}{2}$.

Hence, $x':=x_v+\eps_1>\tx_- + \frac{\eps_1}{2}$. Now, an easy corollary to Lemma~\ref{l:x-C-image} is that for any $A, x$ we have
\begin{equation}\label{eq:dist-product}
\dist (f_A(x),x_+(A)) \cdot \dist (x,x_-(A)) \le \left( \frac{\pi}{2 \|A\|} \right)^2
\end{equation}
(it suffices to multiply the first two conclusions, and the numerators cancel out). As $x'-\tx_- > \frac{\eps_1}{2}$, we get an upper estimate for the distance from its image to $\tx_+ + 1$. Indeed, if we have
$\tG_{\alpha,n,\omega} (x')\le \tx_+ +1 - \frac{\eps_1}{2}$, the left hand side of~\eqref{eq:dist-product} is at least~$(\eps_1/2)^2$. If $\|A_{0,n,\omega}\|\ge \frac{\pi}{\eps_1}$, having this would imply a contradiction.

Adding up, once $\|A_{0,n,\omega}\|\ge \frac{\pi \max(K,1)}{\eps_1}$, we have
$$
\tG_{0,n,\omega} ( x_v + \eps_1) > \tx_+ +1 - \frac{\eps_1}{2},
$$
and thus we get the desired
$$
\tG_{\alpha,n,\omega} (x_v)>\tG_{0,n,\omega} ( x_v + \eps_1)+ \eps_1 > \tx_+ +1 + \frac{\eps_1}{2} > \tx_+ +1.
$$
The second inequality of~\eqref{eq:overjump} is absolutely analogous. We have obtained~\eqref{eq:overjump}, and thus have concluded the proof of the proposition.
\end{proof}
As Proposition~\ref{l:f-unstable} is proven, so is Theorem~\ref{t.gjt}.
\end{proof}

\end{appendix}

\section*{Acknowledgments}

We are grateful to David Damanik who attracted our attention to the question and provided numerous relevant references, to Abel Klein and  Lana Jitomirskaya for useful discussions and remarks, and to Jairo Bochi, Jake Fillman, and Zhenghe Zhang for sending us helpful comments on the first draft of the paper. \edited{Also, we would like to thank both  referees for extremely careful refereeing, multiple helpful remarks, and providing a few highly relevant references.}


\end{document}